\documentclass[10pt]{article}

\usepackage[margin=1in]{geometry}  % set the margins to 1in on all sides
\usepackage{graphicx}              % to include figures
\usepackage{amsmath}               % great math stuff
\usepackage{amsfonts}              % for blackboard bold, etc
\usepackage{amsthm}                % better theorem environments
\usepackage{amssymb}
\usepackage{booktabs}
\usepackage{longtable}
\usepackage{tikz}
\usetikzlibrary{arrows}

\usepackage{comment}
\usepackage{caption} 

\newtheorem{theorem}{Theorem}[section]
\newtheorem{lemma}[theorem]{Lemma}
\newtheorem{proposition}[theorem]{Proposition}
\newtheorem{corollary}[theorem]{Corollary}

\renewcommand{\gg}{\mathfrak{g}}  

\renewcommand{\aa}{\alpha}
\newcommand{\bb}{\beta}
\newcommand{\cc}{\gamma}

\begin{document}

\nocite{*}

\title{The subalgebras of $G_2$}

\author{Evgeny Mayanskiy}

\maketitle

\begin{abstract}
  We classify subalgebras of the complex simple Lie algebra of type $G_2$ up to conjugacy (by an inner automorphism).
\end{abstract}

\section{Introduction}

Classification of subalgebras of Lie algebras is a classical area of research \cite{Morozov}, \cite{Malcev}, \cite{Karpelevich}, \cite{Dynkin}. Nevertheless, an explicit description of \textit{all} subalgebras of exceptional Lie algebras does not seem to be readily available in the literature (see, however, \cite{Aschbacher}). Classification of subalgebras of complex semisimple Lie algebras of types $A_2$, $B_2=C_2$ and $D_2=A_1+A_1$ is the subject of recent work: \cite{RepkaA2}, \cite{RepkaC2}, \cite{RepkaD2}. In this note we do it for the complex simple Lie algebra of type $G_2$.\\

Our main result is the following. The actual subalgebras are listed in Table~\ref{table:1}, Table~\ref{table:2} and Table~\ref{table:3}. $G_2$~denotes the complex simple Lie algebra of type $G_2$.\\

\begin{theorem}\label{Theorem}
Any subalgebra $\mathfrak{g}\subset G_2$ is conjugate in $G_2$ (by an inner automorphism) to one and only one of the following subalgebras:
\begin{itemize}
\item $64$ types of regular subalgebras of $G_2$ listed in Table~\ref{table:1},
\item $2$ non-regular semisimple subalgebras of $G_2$ listed in Table~\ref{table:2},
\item $49$ types of non-regular solvable subalgebras of $G_2$ listed in Table~\ref{table:3}.
\end{itemize}
\end{theorem}

\vspace{3ex}

The proof proceeds by calculation in the Chevalley basis. Unless explicitly stated otherwise, $\oplus$ denotes the direct sum of vector spaces, not of Lie algebras. The base field is $\mathbb C$.\\

\section{$G_2$}

In this section we review the definition of $G_2$ in terms of the root spaces mostly following~\cite{Samelson}.\\

$G_2$ has $12$ roots which can be ordered as follows:
\begin{align*}
-(3\aa +2\bb) \prec -(3\aa +\bb) \prec -(2\aa +\bb) \prec -(\aa +\bb) \prec -\bb \prec \\
\prec -\aa \prec 0 \prec \aa \prec \bb \prec \aa + \bb \prec 2\aa + \bb \prec 3\aa + \bb \prec 3\aa + 2\bb,
\end{align*}
where $\aa$ is the short simple root and $\bb$ is the long simple root (see Figure~\ref{figure:roots}).\\

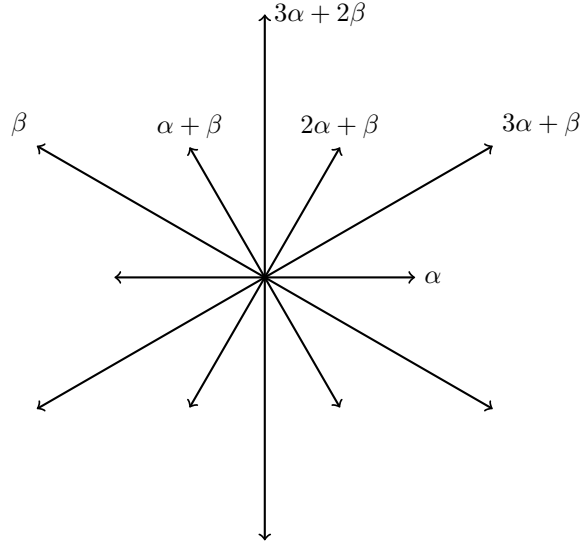
\begin{figure}
\centering
%\begin{tikzpicture}[
%    % arrow heads for all lines (with narrower arrow head width)
%    -{Straight Barb[bend,
%       width=\the\dimexpr10\pgflinewidth\relax,
%       length=\the\dimexpr12\pgflinewidth\relax]},
%  ]
	
\begin{tikzpicture}[->]

\foreach \i in {0, 1, ..., 5} {
	\draw[thick, black] (0, 0) -- (\i*60:2);
  \draw[thick, black] (0, 0) -- (30 + \i*60:3.5);
}

\node[right] at (2, 0) {$\alpha$};
\node[above left] at (5*30:3.5) {$\beta$};
\node[right] at (3*30:3.5) {$3\alpha + 2\beta$};
\node[above right] at (1*30:3.5) {$3\alpha + \beta$};
\node[above] at (2*30:2) {$2\alpha + \beta$};
\node[above] at (4*30:2) {$\alpha + \beta$};

\end{tikzpicture}
\caption{The root system of $G_2$}
\label{figure:roots}
\end{figure}

We choose a metric $(-,-)$ in the ambient Euclidean plane so that $(\aa, \aa)=1$. Then $(\bb, \bb)=3$ and $(\aa, \bb)=-3/2$. Let ${\Phi}^{+}=\{ \aa , \bb , \aa + \bb , 2\aa + \bb , 3\aa + \bb , 3\aa + 2\bb \}$ denote the set of positive roots and $\Phi = {\Phi}^{+}\cup (- {\Phi}^{+})$ the set of all roots. The Chevalley basis of $G_2$ consists of the coroots $H_{\aa}$, $H_{\bb}$ generating a Cartan subalgebra $\mathfrak h \subset G_2$ and the root vectors $X_{\aa}$, $X_{-\aa}$, $X_{\bb}$, $X_{-\bb}$, $X_{\aa + \bb}$, $X_{-(\aa + \bb)}$, $X_{2\aa + \bb}$, $X_{-(2\aa + \bb)}$, $X_{3\aa + \bb}$, $X_{-(3\aa + \bb)}$, $X_{3\aa + 2\bb}$, $X_{-(3\aa + 2\bb)}$:
$$
G_2 = \mathfrak h\oplus \bigoplus_{\cc\in\Phi} \mathbb C X_{\cc},
$$
where $\mathfrak h=\mathbb C H_{\aa} \oplus \mathbb C H_{\bb}$.\\

The multiplication table of $G_2$ in this basis is given as follows:
\begin{align*}
& [H_{\mu},X_{\nu}] = \frac{2(\mu , \nu)}{(\mu , \mu)} \cdot X_{\nu}, \qquad [X_{\mu}, X_{-\mu}]=H_{\mu}, \quad \mu, \nu \in {\Phi}, \\ 
& [X_{\mu}, X_{\nu}] = N_{\mu,\nu}\cdot X_{\mu+\nu},\quad \mu, \nu, \mu+\nu \in \Phi, \\
& 1=N_{\bb, \aa}=N_{\bb, 3\aa+\bb}=N_{3\aa+\bb, -(3\aa+2\bb)}=N_{2\aa+\bb, -(3\aa+\bb)}=N_{2\aa+\bb, -(3\aa+2\bb)}= \\
& =N_{-(3\aa+2\bb), \aa+\bb}=N_{-(3\aa+2\bb), \bb}=N_{-(3\aa+\bb), \aa}=N_{-(\aa+\bb), \bb}, \\
& 2=N_{\aa+\bb, \aa}=N_{\aa, -(2\aa+\bb)}=N_{-(2\aa+\bb), \aa+\bb}, \\
& 3=N_{\aa, 2\aa+\bb}=N_{\aa, -(\aa+\bb)}=N_{\aa+\bb, 2\aa+\bb},\\
& N_{\mu,\nu} = - N_{\nu,\mu} = - N_{-\mu,-\nu}.
\end{align*}

Only non-zero brackets are listed. See \cite{Carter} for the choice of signs.\\

We abuse language by speaking about the action of the Weyl group of $G_2$ on the whole of $G_2$, when we are interested only in the action of the corresponding elements of $Aut(G_2)$ on the dual of the fixed Cartan subalgebra ${\mathfrak{h}}^{*}$. We also speak about reflections in ${\mathfrak{h}}^{*}$ meaning reflections in its normal real form.\\

\section{Regular subalgebras (\cite{Morozov}, \cite{Karpelevich}, \cite{Dynkin})}

A subset $\Sigma \subset \Phi$ is called \textit{closed} if $(\Sigma + \Sigma)\cap \Phi \subset \Sigma$, i.e. for any $x,y\in \Sigma$, $x+y\in \Phi$ implies $x+y\in \Sigma$.\\

Given a closed subset $\Sigma \subset \Phi$ and a vector subspace $L\subset \mathfrak h$ such that $[X_{\cc}, X_{-\cc}]\in L$ for any $\cc \in \Sigma \cap (-\Sigma)$, one can construct a subalgebra $\gg (\Sigma, L) \subset G_2$ as follows:
$$
\gg (\Sigma, L) = L\oplus \bigoplus_{\cc\in\Sigma} \mathbb C X_{\cc}.
$$

Such subalgebras are called \textit{regular} \cite{Dynkin}, \cite{Chebotarev}. Let ${\Sigma}_s = {\Sigma} \cap (- \Sigma ) = \{ \gamma \in \Sigma \; \mid \; -\gamma \in \Sigma \}$. By \cite{Dynkin}, Theorem~$6.3$, the radical of $\gg (\Sigma, L)$ is
$$
\operatorname{rad} (\gg (\Sigma, L) ) = \tilde{L}\oplus \bigoplus_{\cc\in \Sigma \setminus {\Sigma }_s} \mathbb C X_{\cc},
$$
and a Levi subalgebra
$$
\operatorname{Levi} (\gg (\Sigma, L) ) = \left( \sum\limits_{\cc\in {\Sigma }_s} \mathbb C [X_{\gamma} , X_{-\gamma}] \right) \oplus \bigoplus_{\cc\in {\Sigma }_s} \mathbb C X_{\cc},\quad \tilde{L} = \left( \sum\limits_{\cc\in {\Sigma }_s} \mathbb C [X_{\gamma} , X_{-\gamma}] \right)^{\perp}_L\subset L .
$$
It follows from the explicit enumeration of regular subalgebras of $G_2$ in Table~\ref{table:1} that we may assume $\Sigma \setminus {\Sigma }_s \subset {\Phi}^{+}$ (cf. \cite{Chebotarev}, Theorem~$1$). Hence the nilpotent in $G_2$ elements of $\operatorname{rad} (\gg (\Sigma, L) )$ form a subalgebra
$$
\bigoplus_{\cc\in \Sigma \setminus {\Sigma }_s} \mathbb C X_{\cc}.
$$

Note also that $\sum\limits_{\cc\in {\Sigma }_s} \mathbb C [X_{\gamma} , X_{-\gamma}]$ is a Cartan subalgebra of $\operatorname{Levi} (\gg (\Sigma, L) )$.\\

Two regular subalgebras are conjugate in $G_2$ if and only if the corresponding pairs $(\Sigma, L)$ lie in the same orbit of the Weyl group. See Proposition~\ref{proposition:1} below. The enumeration of all pairs $(\Sigma, L)$ satisfying the conditions above up to the action of the Weyl group (i.e. of all regular subalgebras of $G_2$ up to conjugacy) is presented in Table~\ref{table:1}. The first column of Table~\ref{table:1} contains our notation for the subalgebra $\gg (\Sigma, L)$, the second and the third columns show the closed subset $\Sigma \subset \Phi$ and the possible subspaces $L\subset \mathfrak h$. Denote $\mathfrak{u}_{\gamma}=\mathbb C X_{\gamma}\subset G_2$, $\gamma \in \Phi$.

\begin{center}
\begin{longtable}{|c|c|c|}
\caption{Regular subalgebras of $G_2$ up to conjugacy}\label{table:1}\\
\hline
\begin{tabular}{@{}c@{}}Subalgebra \\ $\gg (\Sigma, L)$\end{tabular}  & $\Sigma$ & $L$
\\ \hline\hline
	$0$ & $\varnothing$ & $0$
	\\ \hline
	$L$ & $\varnothing$ & \begin{tabular}{@{}c@{}}$L\subset \mathfrak h$, $dim (L)=1$ \\ $\alpha \cdot \beta (L)\geq 0$\end{tabular}
		\\ \hline
	$\mathfrak h$ & $\varnothing$ & $\mathfrak h$ 
	\\ \hline\hline
	${\mathfrak u}_{\alpha}=\mathbb C X_{\alpha}$ & \begin{tikzpicture}[baseline=0,scale=1]
		\draw[->] (0, 0) node[circle,fill,inner sep=1pt]{} -- (1,0) node[right] {$\alpha$};
  \end{tikzpicture}    & $0$
  \\ \hline
	${\mathfrak u}_{\alpha}\oplus L$ & \begin{tikzpicture}[baseline=0,scale=1]
	\draw[->] (0, 0) node[circle,fill,inner sep=1pt]{} -- (1,0) node[right] {$\alpha$};
  \end{tikzpicture}    & \begin{tabular}{@{}c@{}}$L\subset \mathfrak h$, $dim (L)=1$ \\ $\alpha \cdot (3\alpha +2\beta ) (L)\geq 0$\end{tabular}
  \\ \hline
	${\mathfrak u}_{\alpha}\oplus {\mathfrak h}$ & \begin{tikzpicture}[baseline=0,scale=1]
	\draw[->] (0, 0) node[circle,fill,inner sep=1pt]{} -- (1,0) node[right] {$\alpha$};
  \end{tikzpicture}    & $\mathfrak h$
  \\ \hline\hline
	${\mathfrak u}_{\beta}=\mathbb C X_{\beta}$ & \begin{tikzpicture}[baseline=0,scale=1]
		\draw[->] (0, 0) node[circle,fill,inner sep=1pt]{} -- (1.732,0) node[right] {$\beta$};
  \end{tikzpicture}    & $0$
  \\ \hline
	${\mathfrak u}_{\beta}\oplus L$ & \begin{tikzpicture}[baseline=0,scale=1]
		\draw[->] (0, 0) node[circle,fill,inner sep=1pt]{} -- (1.732,0) node[right] {$\beta$};
  \end{tikzpicture}    & \begin{tabular}{@{}c@{}}$L\subset \mathfrak h$, $dim (L)=1$ \\ $\beta \cdot (2\alpha +\beta ) (L)\geq 0$\end{tabular}
  \\ \hline
	${\mathfrak u}_{\beta}\oplus {\mathfrak h}$ & \begin{tikzpicture}[baseline=0,scale=1]
		\draw[->] (0, 0) node[circle,fill,inner sep=1pt]{} -- (1.732,0) node[right] {$\beta$};
  \end{tikzpicture}    & $\mathfrak h$
  \\ \hline\hline
	$\tilde{A_1}$ & \begin{tikzpicture}[baseline=0,scale=1]
			\draw[->] (0, 0) node[circle,fill,inner sep=1pt]{} -- (1,0) node[right] {$\alpha$};
			\draw[->] (0, 0) -- (-1,0) node[left] {$-\alpha$};
  \end{tikzpicture}    & $\mathbb C H_{\alpha}$
  \\ \hline
	$\tilde{A_1}\oplus \mathbb C H_{3\alpha+2\beta}$ & \begin{tikzpicture}[baseline=0,scale=1]
			\draw[->] (0, 0) node[circle,fill,inner sep=1pt]{} -- (1,0) node[right] {$\alpha$};
			\draw[->] (0, 0) -- (-1,0) node[left] {$-\alpha$};
  \end{tikzpicture}    & $\mathfrak h$
  \\ \hline\hline
	$A_1$	& \begin{tikzpicture}[baseline=0,scale=1]
			\draw[->] (0, 0) node[circle,fill,inner sep=1pt]{} -- (1.732,0) node[right] {$\beta$};
			\draw[->] (0, 0) -- (-1.732,0) node[left] {$-\beta$};
  \end{tikzpicture}    & $\mathbb C H_{\beta}$
  \\ \hline
	$A_1\oplus \mathbb C H_{2\alpha+\beta}$	& \begin{tikzpicture}[baseline=0,scale=1]
			\draw[->] (0, 0) node[circle,fill,inner sep=1pt]{} -- (1.732,0) node[right] {$\beta$};
			\draw[->] (0, 0) -- (-1.732,0) node[left] {$-\beta$};
  \end{tikzpicture}    & $\mathfrak h$
  \\ \hline\hline
	$L\oplus \mathfrak{u}_{\alpha}\oplus \mathfrak{u}_{3\alpha+\beta}$ & \begin{tikzpicture}[baseline=0,scale=1]
			\draw[->] (0, 0) node[circle,fill,inner sep=1pt]{} -- (1,0) node[right] {$\alpha$};
			\draw[->] (0, 0) -- (1.5,0.866) node[right] {$3\alpha + \beta$};
  \end{tikzpicture}    & $L\subset \mathfrak h$
  \\ \hline	
	$L\oplus \mathfrak{u}_{\alpha}\oplus \mathfrak{u}_{3\alpha+2\beta}$	& \begin{tikzpicture}[baseline=0,scale=1]
			\draw[->] (0, 0) node[circle,fill,inner sep=1pt]{} -- (1,0) node[right] {$\alpha$};
			\draw[->] (0, 0) -- (0,1.732) node[right] {$3\alpha + 2\beta$};
  \end{tikzpicture}    & $L\subset \mathfrak h$
  \\ \hline	\hline	
	$\mathfrak{u}_{\beta}\oplus \mathfrak{u}_{3\alpha+2\beta}$	& \begin{tikzpicture}[baseline=0,scale=1]
			\draw[->] (0, 0) node[circle,fill,inner sep=1pt]{} -- (1.732,0) node[right] {$3\alpha + 2\beta$};
			\draw[->] (0, 0) -- (0.866,1.5) node[right] {$\beta$};
  \end{tikzpicture}    & $0$
  \\ \hline
	$L\oplus \mathfrak{u}_{\beta}\oplus \mathfrak{u}_{3\alpha+2\beta}$	& \begin{tikzpicture}[baseline=0,scale=1]
			\draw[->] (0, 0) node[circle,fill,inner sep=1pt]{} -- (1.732,0) node[right] {$3\alpha + 2\beta$};
			\draw[->] (0, 0) -- (0.866,1.5) node[right] {$\beta$};
  \end{tikzpicture}    & \begin{tabular}{@{}c@{}}$L\subset \mathfrak h $, $dim (L)=1$, \\ $(\alpha + \beta)\cdot (3\alpha + \beta) (L)\geq 0$\end{tabular}
  \\ \hline
	$\mathfrak{h}\oplus \mathfrak{u}_{\beta}\oplus \mathfrak{u}_{3\alpha+2\beta}$	& \begin{tikzpicture}[baseline=0,scale=1]
			\draw[->] (0, 0) node[circle,fill,inner sep=1pt]{} -- (1.732,0) node[right] {$3\alpha + 2\beta$};
			\draw[->] (0, 0) -- (0.866,1.5) node[right] {$\beta$};
  \end{tikzpicture}    & $\mathfrak{h}$
  \\ \hline\hline
	\begin{tabular}{@{}c@{}} $\mathfrak{u}_{\alpha}$ \\ $\oplus \mathfrak{u}_{2\alpha+\beta}\oplus \mathfrak{u}_{3\alpha+\beta}$ \end{tabular}	& \begin{tikzpicture}[baseline=0,scale=1]
			\draw[->] (0, 0) node[circle,fill,inner sep=1pt]{} -- (1,0) node[right] {$\alpha$};
			\draw[->] (0, 0) -- (0.5,0.866) node[above] {$2\alpha + \beta$};
			\draw[->] (0, 0) -- (1.5,0.866) node[right] {$3\alpha + \beta$};       
  \end{tikzpicture}    & $0$
  \\ \hline	
	\begin{tabular}{@{}c@{}} $L\oplus \mathfrak{u}_{\alpha}$ \\ $\oplus \mathfrak{u}_{2\alpha+\beta}\oplus \mathfrak{u}_{3\alpha+\beta}$ \end{tabular}	& \begin{tikzpicture}[baseline=0,scale=1]
			\draw[->] (0, 0) node[circle,fill,inner sep=1pt]{} -- (1,0) node[right] {$\alpha$};
			\draw[->] (0, 0) -- (0.5,0.866) node[above] {$2\alpha + \beta$};
			\draw[->] (0, 0) -- (1.5,0.866) node[right] {$3\alpha + \beta$};       
  \end{tikzpicture}    & \begin{tabular}{@{}c@{}}$L\subset \mathfrak h$, $dim (L)=1$, \\ $(\alpha +\beta)\cdot (3\alpha +\beta)(L)\geq 0$\end{tabular}
  \\ \hline
	\begin{tabular}{@{}c@{}} $\mathfrak{h}\oplus \mathfrak{u}_{\alpha}$ \\ $\oplus \mathfrak{u}_{2\alpha+\beta}\oplus \mathfrak{u}_{3\alpha+\beta}$ \end{tabular}	& \begin{tikzpicture}[baseline=0,scale=1]
			\draw[->] (0, 0) node[circle,fill,inner sep=1pt]{} -- (1,0) node[right] {$\alpha$};
			\draw[->] (0, 0) -- (0.5,0.866) node[above] {$2\alpha + \beta$};
			\draw[->] (0, 0) -- (1.5,0.866) node[right] {$3\alpha + \beta$};       
  \end{tikzpicture}    & $\mathfrak{h}$
  \\ \hline\hline
	$\tilde{A_1}\oplus \mathfrak{u}_{3\alpha +2\beta}$	& \begin{tikzpicture}[baseline=0,scale=1]
			\draw[->] (0, 0) node[circle,fill,inner sep=1pt]{} -- (1,0) node[above] {$\alpha$};
			\draw[->] (0, 0) -- (-1,0) node[above] {$-\alpha$};
			\draw[->] (0, 0) -- (0,1.732) node[right] {$3\alpha + 2\beta$};
  \end{tikzpicture}    & $\mathbb C H_{\alpha}$
  \\ \hline	
	$\tilde{A_1}\oplus \mathbb C H_{3\alpha +2\beta}\oplus \mathfrak{u}_{3\alpha +2\beta}$	& \begin{tikzpicture}[baseline=0,scale=1]
			\draw[->] (0, 0) node[circle,fill,inner sep=1pt]{} -- (1,0) node[above] {$\alpha$};
			\draw[->] (0, 0) -- (-1,0) node[above] {$-\alpha$};
			\draw[->] (0, 0) -- (0,1.732) node[right] {$3\alpha + 2\beta$};
  \end{tikzpicture}    & $\mathfrak h$
  \\ \hline\hline
	$A_1\oplus \mathfrak{u}_{2\alpha +\beta}$ & \begin{tikzpicture}[baseline=0,scale=1]
			\draw[->] (0, 0) node[circle,fill,inner sep=1pt]{} -- (1.732,0) node[above] {$-\beta$};
			\draw[->] (0, 0) -- (-1.732,0) node[above] {$\beta$};
			\draw[->] (0, 0) -- (0,1) node[above] {$2\alpha + \beta$};
  \end{tikzpicture}    & $\mathbb C H_{\beta}$
  \\ \hline
	$A_1\oplus \mathbb C H_{2\alpha +\beta} \oplus \mathfrak{u}_{2\alpha +\beta}$ & \begin{tikzpicture}[baseline=0,scale=1]
			\draw[->] (0, 0) node[circle,fill,inner sep=1pt]{} -- (1.732,0) node[above] {$-\beta$};
			\draw[->] (0, 0) -- (-1.732,0) node[above] {$\beta$};
			\draw[->] (0, 0) -- (0,1) node[above] {$2\alpha + \beta$};
  \end{tikzpicture}    & $\mathfrak h$
  \\ \hline\hline
	\begin{tabular}{@{}c@{}} $\mathfrak{u}_{\beta}$ \\ $\oplus \mathfrak{u}_{\alpha+\beta}\oplus \mathfrak{u}_{3\alpha+2\beta}$ \end{tabular} & \begin{tikzpicture}[baseline=0,scale=1]
			\draw[->] (0, 0) node[circle,fill,inner sep=1pt]{} -- (0,1) node[above] {$\alpha + \beta$};
			\draw[->] (0, 0) -- (0.866,1.5) node[right] {$3\alpha + 2\beta$};
			\draw[->] (0, 0) -- (-0.866,1.5) node[left] {$\beta$};            
	\end{tikzpicture}    & $0$
  \\ \hline	
	\begin{tabular}{@{}c@{}} $L\oplus \mathfrak{u}_{\beta}$ \\ $\oplus \mathfrak{u}_{\alpha+\beta}\oplus \mathfrak{u}_{3\alpha+2\beta}$ \end{tabular} & \begin{tikzpicture}[baseline=0,scale=1]
			\draw[->] (0, 0) node[circle,fill,inner sep=1pt]{} -- (0,1) node[above] {$\alpha + \beta$};
			\draw[->] (0, 0) -- (0.866,1.5) node[right] {$3\alpha + 2\beta$};
			\draw[->] (0, 0) -- (-0.866,1.5) node[left] {$\beta$};            
	\end{tikzpicture}    & \begin{tabular}{@{}c@{}}$L\subset \mathfrak h$, $dim (L)=1$, \\ $(\alpha +\beta ) \cdot (3\alpha +\beta)(L)\geq 0$\end{tabular}
  \\ \hline	
	\begin{tabular}{@{}c@{}} $\mathfrak{h}\oplus \mathfrak{u}_{\beta}$ \\ $\oplus \mathfrak{u}_{\alpha+\beta}\oplus \mathfrak{u}_{3\alpha+2\beta}$ \end{tabular} & \begin{tikzpicture}[baseline=0,scale=1]
			\draw[->] (0, 0) node[circle,fill,inner sep=1pt]{} -- (0,1) node[above] {$\alpha + \beta$};
			\draw[->] (0, 0) -- (0.866,1.5) node[right] {$3\alpha + 2\beta$};
			\draw[->] (0, 0) -- (-0.866,1.5) node[left] {$\beta$};            
	\end{tikzpicture}    & $\mathfrak{h}$
  \\ \hline	\hline	
	\begin{tabular}{@{}c@{}} $\mathfrak{u}_{\beta}$ \\ $\oplus \mathfrak{u}_{3\alpha+\beta}\oplus \mathfrak{u}_{3\alpha+2\beta}$ \end{tabular} & \begin{tikzpicture}[baseline=0,scale=1]
			\draw[->] (0, 0) node[circle,fill,inner sep=1pt]{} -- (1.732,0) node[right] {$3\alpha + \beta$};
			\draw[->] (0, 0) -- (0.866,1.5) node[right] {$3\alpha + 2\beta$};
			\draw[->] (0, 0) -- (-0.866,1.5) node[left] {$\beta$};                 
  \end{tikzpicture}    & $0$
  \\ \hline	
	\begin{tabular}{@{}c@{}} $L\oplus \mathfrak{u}_{\beta}$ \\ $\oplus \mathfrak{u}_{3\alpha+\beta}\oplus \mathfrak{u}_{3\alpha+2\beta}$ \end{tabular} & \begin{tikzpicture}[baseline=0,scale=1]
			\draw[->] (0, 0) node[circle,fill,inner sep=1pt]{} -- (1.732,0) node[right] {$3\alpha + \beta$};
			\draw[->] (0, 0) -- (0.866,1.5) node[right] {$3\alpha + 2\beta$};
			\draw[->] (0, 0) -- (-0.866,1.5) node[left] {$\beta$};                 
  \end{tikzpicture}    & \begin{tabular}{@{}c@{}}$L\subset \mathfrak h$, $dim (L)=1$, \\ $\alpha \cdot (3\alpha +2\beta)(L)\geq 0$\end{tabular}
  \\ \hline
	\begin{tabular}{@{}c@{}} $\mathfrak{h}\oplus \mathfrak{u}_{\beta}$ \\ $\oplus \mathfrak{u}_{3\alpha+\beta}\oplus \mathfrak{u}_{3\alpha+2\beta}$ \end{tabular} & \begin{tikzpicture}[baseline=0,scale=1]
			\draw[->] (0, 0) node[circle,fill,inner sep=1pt]{} -- (1.732,0) node[right] {$3\alpha + \beta$};
			\draw[->] (0, 0) -- (0.866,1.5) node[right] {$3\alpha + 2\beta$};
			\draw[->] (0, 0) -- (-0.866,1.5) node[left] {$\beta$};                 
  \end{tikzpicture}    & $\mathfrak{h}$
  \\ \hline\hline
	\begin{tabular}{@{}c@{}} $L\oplus \mathfrak{u}_{\alpha}$ \\ $\oplus \mathfrak{u}_{3\alpha+\beta}\oplus \mathfrak{u}_{3\alpha+2\beta}$ \end{tabular} & \begin{tikzpicture}[baseline=0,scale=1]
			\draw[->] (0, 0) node[circle,fill,inner sep=1pt]{} -- (1,0) node[right] {$\alpha$};
			\draw[->] (0, 0) -- (0,1.732) node[above] {$3\alpha + 2\beta$};
			\draw[->] (0, 0) -- (1.5,0.866) node[above] {$3\alpha + \beta$};
  \end{tikzpicture}    & $L\subset \mathfrak h$
	\\ \hline
	\begin{tabular}{@{}c@{}} $L\oplus \mathfrak{u}_{\alpha}\oplus \mathfrak{u}_{2\alpha+\beta}$ \\ $\oplus \mathfrak{u}_{3\alpha+\beta}\oplus \mathfrak{u}_{3\alpha+2\beta}$ \end{tabular} & \begin{tikzpicture}[baseline=0,scale=1]
			\draw[->] (0, 0) node[circle,fill,inner sep=1pt]{} -- (1,0) node[right] {$\alpha$};
			\draw[->] (0, 0) -- (0,1.732) node[above] {$3\alpha + 2\beta$};
			\draw[->] (0, 0) -- (1.5,0.866) node[right] {$3\alpha + \beta$};
			\draw[->] (0, 0) -- (0.5,0.866) node[above right] {$2\alpha + \beta$};
  \end{tikzpicture}    & $L\subset \mathfrak h$
  \\ \hline
	\begin{tabular}{@{}c@{}} $L\oplus \mathfrak{u}_{\beta}\oplus \mathfrak{u}_{2\alpha+\beta}$ \\ $\oplus \mathfrak{u}_{3\alpha+\beta}\oplus \mathfrak{u}_{3\alpha+2\beta}$ \end{tabular} & \begin{tikzpicture}[baseline=0,scale=1]
			\draw[->] (0, 0) node[circle,fill,inner sep=1pt]{} -- (1.732,0) node[right] {$3\alpha + \beta$};
			\draw[->] (0, 0) -- (0.866,1.5) node[right] {$3\alpha + 2\beta$};
			\draw[->] (0, 0) -- (-0.866,1.5) node[left] {$\beta$};
			\draw[->] (0, 0) -- (0.866,0.5) node[right] {$2\alpha + \beta$};
		\end{tikzpicture}    & $L\subset \mathfrak h$
  \\ \hline	\hline	
	\begin{tabular}{@{}c@{}} $A_1$ \\ $\oplus \mathfrak{u}_{3\alpha +\beta} \oplus \mathfrak{u}_{3\alpha +2\beta}$ \end{tabular} & \begin{tikzpicture}[baseline=0,scale=1]
			\draw[->] (0, 0) node[circle,fill,inner sep=1pt]{} -- (1.732,0) node[right] {$-\beta$};
			\draw[->] (0, 0) -- (0.866,1.5) node[right] {$3\alpha + \beta$};
			\draw[->] (0, 0) -- (-0.866,1.5) node[left] {$3\alpha + 2\beta$};
			\draw[->] (0, 0) -- (-1.732,0) node[left] {$\beta$};
  \end{tikzpicture}    & $\mathbb C H_{\beta}$
	\\ \hline	
	\begin{tabular}{@{}c@{}} $A_1\oplus \mathbb C H_{2\alpha +\beta}$ \\ $ \oplus \mathfrak{u}_{3\alpha +\beta} \oplus \mathfrak{u}_{3\alpha +2\beta}$ \end{tabular} & \begin{tikzpicture}[baseline=0,scale=1]
			\draw[->] (0, 0) node[circle,fill,inner sep=1pt]{} -- (1.732,0) node[right] {$-\beta$};
			\draw[->] (0, 0) -- (0.866,1.5) node[right] {$3\alpha + \beta$};
			\draw[->] (0, 0) -- (-0.866,1.5) node[left] {$3\alpha + 2\beta$};
			\draw[->] (0, 0) -- (-1.732,0) node[left] {$\beta$};
  \end{tikzpicture}    & $\mathfrak h$
	\\ \hline\hline	
	\begin{tabular}{@{}c@{}} $\mathfrak{u}_{\alpha}$ \\ $\oplus \mathfrak{u}_{\alpha+\beta}\oplus \mathfrak{u}_{2\alpha+\beta}$ \\ $\oplus \mathfrak{u}_{3\alpha+\beta}\oplus \mathfrak{u}_{3\alpha+2\beta}$ \end{tabular} & \begin{tikzpicture}[baseline=0,scale=1]
			\draw[->] (0, 0) node[circle,fill,inner sep=1pt]{} -- (1,0) node[right] {$\alpha$};
			\draw[->] (0, 0) -- (0,1.732) node[above] {$3\alpha + 2\beta$};
			\draw[->] (0, 0) -- (1.5,0.866) node[right] {$3\alpha + \beta$};
			\draw[->] (0, 0) -- (0.5,0.866) node[above right] {$2\alpha + \beta$};
			\draw[->] (0, 0) -- (-0.5,0.866) node[left] {$\alpha + \beta$};   
  \end{tikzpicture}    & $0$
  \\ \hline	
	\begin{tabular}{@{}c@{}} $L\oplus \mathfrak{u}_{\alpha}$ \\ $\oplus \mathfrak{u}_{\alpha+\beta}\oplus \mathfrak{u}_{2\alpha+\beta}$ \\ $\oplus \mathfrak{u}_{3\alpha+\beta}\oplus \mathfrak{u}_{3\alpha+2\beta}$ \end{tabular} & \begin{tikzpicture}[baseline=0,scale=1]
			\draw[->] (0, 0) node[circle,fill,inner sep=1pt]{} -- (1,0) node[right] {$\alpha$};
			\draw[->] (0, 0) -- (0,1.732) node[above] {$3\alpha + 2\beta$};
			\draw[->] (0, 0) -- (1.5,0.866) node[right] {$3\alpha + \beta$};
			\draw[->] (0, 0) -- (0.5,0.866) node[above right] {$2\alpha + \beta$};
			\draw[->] (0, 0) -- (-0.5,0.866) node[left] {$\alpha + \beta$};   
  \end{tikzpicture}    & \begin{tabular}{@{}c@{}}$L\subset \mathfrak h$, $dim (L)=1$, \\ $\beta \cdot (2\alpha +\beta)(L)\geq 0$\end{tabular}
  \\ \hline	
	\begin{tabular}{@{}c@{}} $\mathfrak{h}\oplus \mathfrak{u}_{\alpha}$ \\ $\oplus \mathfrak{u}_{\alpha+\beta}\oplus \mathfrak{u}_{2\alpha+\beta}$ \\ $\oplus \mathfrak{u}_{3\alpha+\beta}\oplus \mathfrak{u}_{3\alpha+2\beta}$ \end{tabular} & \begin{tikzpicture}[baseline=0,scale=1]
			\draw[->] (0, 0) node[circle,fill,inner sep=1pt]{} -- (1,0) node[right] {$\alpha$};
			\draw[->] (0, 0) -- (0,1.732) node[above] {$3\alpha + 2\beta$};
			\draw[->] (0, 0) -- (1.5,0.866) node[right] {$3\alpha + \beta$};
			\draw[->] (0, 0) -- (0.5,0.866) node[above right] {$2\alpha + \beta$};
			\draw[->] (0, 0) -- (-0.5,0.866) node[left] {$\alpha + \beta$};   
  \end{tikzpicture}    & $\mathfrak{h}$
  \\ \hline	 \hline	
	\begin{tabular}{@{}c@{}} $\mathfrak{u}_{\beta}$ \\ $\oplus \mathfrak{u}_{\alpha+\beta}\oplus \mathfrak{u}_{2\alpha+\beta}$ \\ $\oplus \mathfrak{u}_{3\alpha+\beta}\oplus \mathfrak{u}_{3\alpha+2\beta}$ \end{tabular} & \begin{tikzpicture}[baseline=0,scale=1]
			\draw[->] (0, 0) node[circle,fill,inner sep=1pt]{} -- (1.732,0) node[right] {$3\alpha + \beta$};
			\draw[->] (0, 0) -- (0.866,1.5) node[right] {$3\alpha + 2\beta$};
			\draw[->] (0, 0) -- (-0.866,1.5) node[left] {$\beta$};
			\draw[->] (0, 0) -- (0,1) node[above] {$\alpha + \beta$};
			\draw[->] (0, 0) -- (0.866,0.5) node[right] {$2\alpha + \beta$};    
  \end{tikzpicture}    & $0$
	  \\ \hline	
		\begin{tabular}{@{}c@{}} $L\oplus \mathfrak{u}_{\beta}$ \\ $\oplus \mathfrak{u}_{\alpha+\beta}\oplus \mathfrak{u}_{2\alpha+\beta}$ \\ $\oplus \mathfrak{u}_{3\alpha+\beta}\oplus \mathfrak{u}_{3\alpha+2\beta}$ \end{tabular} & \begin{tikzpicture}[baseline=0,scale=1]
			\draw[->] (0, 0) node[circle,fill,inner sep=1pt]{} -- (1.732,0) node[right] {$3\alpha + \beta$};
			\draw[->] (0, 0) -- (0.866,1.5) node[right] {$3\alpha + 2\beta$};
			\draw[->] (0, 0) -- (-0.866,1.5) node[left] {$\beta$};
			\draw[->] (0, 0) -- (0,1) node[above] {$\alpha + \beta$};
			\draw[->] (0, 0) -- (0.866,0.5) node[right] {$2\alpha + \beta$};    
  \end{tikzpicture}    & \begin{tabular}{@{}c@{}}$L\subset \mathfrak h$, $dim (L)=1$, \\ $\alpha \cdot (3\alpha +2\beta)(L)\geq 0$\end{tabular}
	  \\ \hline	
		\begin{tabular}{@{}c@{}} $\mathfrak{h}\oplus \mathfrak{u}_{\beta}$ \\ $\oplus \mathfrak{u}_{\alpha+\beta}\oplus \mathfrak{u}_{2\alpha+\beta}$ \\ $\oplus \mathfrak{u}_{3\alpha+\beta}\oplus \mathfrak{u}_{3\alpha+2\beta}$ \end{tabular} & \begin{tikzpicture}[baseline=0,scale=1]
			\draw[->] (0, 0) node[circle,fill,inner sep=1pt]{} -- (1.732,0) node[right] {$3\alpha + \beta$};
			\draw[->] (0, 0) -- (0.866,1.5) node[right] {$3\alpha + 2\beta$};
			\draw[->] (0, 0) -- (-0.866,1.5) node[left] {$\beta$};
			\draw[->] (0, 0) -- (0,1) node[above] {$\alpha + \beta$};
			\draw[->] (0, 0) -- (0.866,0.5) node[right] {$2\alpha + \beta$};    
  \end{tikzpicture}    & $\mathfrak{h}$
	  \\ \hline	\hline	
	\begin{tabular}{@{}c@{}} $A_1$ \\ $\oplus \mathfrak{u}_{2\alpha +\beta} \oplus \mathfrak{u}_{3\alpha +\beta}$ \\ $ \oplus \mathfrak{u}_{3\alpha +2\beta}$ \end{tabular} & \begin{tikzpicture}[baseline=0,scale=1]
			\draw[->] (0, 0) node[circle,fill,inner sep=1pt]{} -- (1.732,0) node[right] {$-\beta$};
			\draw[->] (0, 0) -- (-1.732,0) node[left] {$\beta$};
			\draw[->] (0, 0) -- (0.866,1.5) node[right] {$3\alpha + \beta$};
			\draw[->] (0, 0) -- (-0.866,1.5) node[left] {$3\alpha + 2\beta$};
			\draw[->] (0, 0) -- (0,1) node[above] {$2\alpha + \beta$};
  \end{tikzpicture}    & $\mathbb C H_{\beta}$
  \\ \hline	
	\begin{tabular}{@{}c@{}} $A_1 \oplus \mathbb C H_{2\alpha +\beta}$ \\ $\oplus \mathfrak{u}_{2\alpha +\beta} \oplus \mathfrak{u}_{3\alpha +\beta} $ \\ $\oplus \mathfrak{u}_{3\alpha +2\beta}$ \end{tabular} & \begin{tikzpicture}[baseline=0,scale=1]
			\draw[->] (0, 0) node[circle,fill,inner sep=1pt]{} -- (1.732,0) node[right] {$-\beta$};
			\draw[->] (0, 0) -- (-1.732,0) node[left] {$\beta$};
			\draw[->] (0, 0) -- (0.866,1.5) node[right] {$3\alpha + \beta$};
			\draw[->] (0, 0) -- (-0.866,1.5) node[left] {$3\alpha + 2\beta$};
			\draw[->] (0, 0) -- (0,1) node[above] {$2\alpha + \beta$};
  \end{tikzpicture}    & $\mathfrak h$
  \\ \hline	\hline
	$A_1+\tilde{A_1}=G_2(\beta)$	& \begin{tikzpicture}[baseline=0,scale=1]
			\draw[->] (0, 0) node[circle,fill,inner sep=1pt]{} -- (1.732,0) node[above] {$-(3\alpha + 2\beta)$};
			\draw[->] (0, 0) -- (-1.732,0) node[above] {$3\alpha + 2\beta$};
			\draw[->] (0, 0) -- (0,1) node[right] {$\alpha$};
			\draw[->] (0, 0) -- (0,-1) node[right] {$-\alpha$};
  \end{tikzpicture}    & $\mathfrak h$
  \\ \hline		
	$A_2=G_2(\alpha)$ & \begin{tikzpicture}[baseline=0,scale=1]
			\draw[->] (0, 0) node[circle,fill,inner sep=1pt]{} -- (1.732,0) node[right] {$-\beta$};
			\draw[->] (0, 0) -- (-1.732,0) node[left] {$\beta$};
			\draw[->] (0, 0) -- (0.866,1.5) node[right] {$3\alpha + \beta$};
			\draw[->] (0, 0) -- (-0.866,1.5) node[left] {$3\alpha + 2\beta$};
			\draw[->] (0, 0) -- (0.866,-1.5) node[right] {$-(3\alpha + 2\beta)$};
			\draw[->] (0, 0) -- (-0.866,-1.5) node[left] {$-(3\alpha + \beta)$};
  \end{tikzpicture}    & $\mathfrak h$
  \\ \hline	\hline	
	\begin{tabular}{@{}c@{}} $\mathfrak{u}_{\alpha}\oplus \mathfrak{u}_{\beta}$ \\ $\oplus \mathfrak{u}_{\alpha+\beta}\oplus \mathfrak{u}_{2\alpha+\beta} $ \\ $\oplus \mathfrak{u}_{3\alpha+\beta}\oplus \mathfrak{u}_{3\alpha+2\beta}$ \end{tabular} & \begin{tikzpicture}[baseline=0,scale=1]
			\draw[->] (0, 0) node[circle,fill,inner sep=1pt]{} -- (1,0) node[right] {$\alpha$};
			\draw[->] (0, 0) -- (0,1.732) node[above] {$3\alpha + 2\beta$};
			\draw[->] (0, 0) -- (1.5,0.866) node[right] {$3\alpha + \beta$};
			\draw[->] (0, 0) -- (0.5,0.866) node[above right] {$2\alpha + \beta$};
			\draw[->] (0, 0) -- (-0.5,0.866) node[above] {$\alpha + \beta$};
			\draw[->] (0, 0) -- (-1.5,0.866) node[left] {$\beta$};
  \end{tikzpicture}    & $0$
	\\ \hline	
	\begin{tabular}{@{}c@{}} $L\oplus \mathfrak{u}_{\alpha}\oplus \mathfrak{u}_{\beta}$ \\ $\oplus \mathfrak{u}_{\alpha+\beta}\oplus \mathfrak{u}_{2\alpha+\beta} $ \\ $\oplus \mathfrak{u}_{3\alpha+\beta}\oplus \mathfrak{u}_{3\alpha+2\beta}$ \end{tabular} & \begin{tikzpicture}[baseline=0,scale=1]
			\draw[->] (0, 0) node[circle,fill,inner sep=1pt]{} -- (1,0) node[right] {$\alpha$};
			\draw[->] (0, 0) -- (0,1.732) node[above] {$3\alpha + 2\beta$};
			\draw[->] (0, 0) -- (1.5,0.866) node[right] {$3\alpha + \beta$};
			\draw[->] (0, 0) -- (0.5,0.866) node[above right] {$2\alpha + \beta$};
			\draw[->] (0, 0) -- (-0.5,0.866) node[above] {$\alpha + \beta$};
			\draw[->] (0, 0) -- (-1.5,0.866) node[left] {$\beta$};
  \end{tikzpicture}    & $L\subset \mathfrak h$, $dim(L)=1$
	\\ \hline	
	$\mathfrak{b}=\mathfrak{h}\oplus \bigoplus\limits_{\gamma\in {\Phi}^{+}}\mathbb C X_{\gamma}$ & \begin{tikzpicture}[baseline=0,scale=1]
			\draw[->] (0, 0) node[circle,fill,inner sep=1pt]{} -- (1,0) node[right] {$\alpha$};
			\draw[->] (0, 0) -- (0,1.732) node[above] {$3\alpha + 2\beta$};
			\draw[->] (0, 0) -- (1.5,0.866) node[right] {$3\alpha + \beta$};
			\draw[->] (0, 0) -- (0.5,0.866) node[above right] {$2\alpha + \beta$};
			\draw[->] (0, 0) -- (-0.5,0.866) node[above] {$\alpha + \beta$};
			\draw[->] (0, 0) -- (-1.5,0.866) node[left] {$\beta$};
  \end{tikzpicture}    & $\mathfrak h$
	\\ \hline	 \hline	
	\begin{tabular}{@{}c@{}} $\tilde{A_1}\oplus \mathfrak{u}_{\beta}$ \\ $\oplus \mathfrak{u}_{\alpha +\beta} \oplus \mathfrak{u}_{2\alpha +\beta}$ \\ $\oplus \mathfrak{u}_{3\alpha +\beta} \oplus \mathfrak{u}_{3\alpha +2\beta}$ \end{tabular} & \begin{tikzpicture}[baseline=0,scale=1]
			\draw[->] (0, 0) node[circle,fill,inner sep=1pt]{} -- (1,0) node[right] {$\alpha$};
			\draw[->] (0, 0) -- (-1,0) node[left] {$-\alpha$};
			\draw[->] (0, 0) -- (0,1.732) node[above] {$3\alpha + 2\beta$};
			\draw[->] (0, 0) -- (1.5,0.866) node[right] {$3\alpha + \beta$};
			\draw[->] (0, 0) -- (0.5,0.866) node[above right] {$2\alpha + \beta$};
			\draw[->] (0, 0) -- (-0.5,0.866) node[above] {$\alpha + \beta$};
			\draw[->] (0, 0) -- (-1.5,0.866) node[left] {$\beta$};
  \end{tikzpicture}    & $\mathbb C H_{\alpha}$
	\\ \hline
	$G_2[\beta]$ & \begin{tikzpicture}[baseline=0,scale=1]
			\draw[->] (0, 0) node[circle,fill,inner sep=1pt]{} -- (1,0) node[right] {$\alpha$};
			\draw[->] (0, 0) -- (-1,0) node[left] {$-\alpha$};
			\draw[->] (0, 0) -- (0,1.732) node[above] {$3\alpha + 2\beta$};
			\draw[->] (0, 0) -- (1.5,0.866) node[right] {$3\alpha + \beta$};
			\draw[->] (0, 0) -- (0.5,0.866) node[above right] {$2\alpha + \beta$};
			\draw[->] (0, 0) -- (-0.5,0.866) node[above] {$\alpha + \beta$};
			\draw[->] (0, 0) -- (-1.5,0.866) node[left] {$\beta$};
  \end{tikzpicture}    & $\mathfrak h$
	\\ \hline\hline
	\begin{tabular}{@{}c@{}} $A_1\oplus \mathfrak{u}_{\alpha}$ \\ $\oplus \mathfrak{u}_{\alpha +\beta} \oplus \mathfrak{u}_{2\alpha +\beta} $ \\ $\oplus \mathfrak{u}_{3\alpha +\beta} \oplus \mathfrak{u}_{3\alpha +2\beta}$ \end{tabular} & \begin{tikzpicture}[baseline=0,scale=1]
			\draw[->] (0, 0) node[circle,fill,inner sep=1pt]{} -- (1.732,0) node[right] {$-\beta$};
			\draw[->] (0, 0) -- (-1.732,0) node[left] {$\beta$};
			\draw[->] (0, 0) -- (0,1) node[above] {$2\alpha + \beta$};
			\draw[->] (0, 0) -- (0.866,0.5) node[above right] {$\alpha$};
			\draw[->] (0, 0) -- (-0.866,0.5) node[above left] {$\alpha + \beta$};
			\draw[->] (0, 0) -- (0.866,1.5) node[above] {$3\alpha + \beta$};
			\draw[->] (0, 0) -- (-0.866,1.5) node[above] {$3\alpha + 2\beta$};
  \end{tikzpicture}    & $\mathbb C H_{\beta}$
	\\ \hline
	$G_2[\alpha]$ & \begin{tikzpicture}[baseline=0,scale=1]
			\draw[->] (0, 0) node[circle,fill,inner sep=1pt]{} -- (1.732,0) node[right] {$-\beta$};
			\draw[->] (0, 0) -- (-1.732,0) node[left] {$\beta$};
			\draw[->] (0, 0) -- (0,1) node[above] {$2\alpha + \beta$};
			\draw[->] (0, 0) -- (0.866,0.5) node[above right] {$\alpha$};
			\draw[->] (0, 0) -- (-0.866,0.5) node[above left] {$\alpha + \beta$};
			\draw[->] (0, 0) -- (0.866,1.5) node[above] {$3\alpha + \beta$};
			\draw[->] (0, 0) -- (-0.866,1.5) node[above] {$3\alpha + 2\beta$};
  \end{tikzpicture}    & $\mathfrak h$ 
	\\ \hline	 \hline	
	$G_2$ & $\Phi $   & $\mathfrak h $
  \\ \hline
\end{longtable}
\end{center}

\begin{proposition}\label{proposition:1}
The subalgebras listed in Table~\ref{table:1} are not pairwise conjugate in $G_2$.
\end{proposition}

\begin{proof}
We will consider each dimension separately. All subalgebras are considered up to conjugacy in $G_2$. Note that if regular subalgebras $\mathfrak{g} (\Sigma , L)$ and $\mathfrak{g} ({\Sigma }', L')$ in $G_2$ are conjugate, then their radicals and Levi subalgebras are also conjugate. In particular, $dim(L)=dim(L')$ and $card (\Sigma )=card (\Sigma ')$.\\

The claim is clear when $L=\mathfrak{h}$, but we will include these cases for completeness.\\

There are three regular subalgebras of dimension $1$:
$$
{\mathfrak u}_{\alpha},\quad {\mathfrak u}_{\beta},\quad L.
$$

None of them can be conjugate in $G_2$ to another, because $X_{\alpha}$, $X_{\beta}$ are not semisimple and their matrices in the adjoint representation have different Jordan blocks.\\

The Cartan subalgebra ${\mathfrak h}$ is the only regular subalgebra of dimension $2$ all of whose elements are semisimple in $G_2$.\\

There are three regular subalgebras of dimension $2$ all of whose elements are nilpotent in $G_2$:
$$
\mathbb C X_{\alpha} \oplus \mathbb C X_{3\alpha + \beta},\quad \mathbb C X_{\beta} \oplus \mathbb C X_{3\alpha + 2\beta},\quad \mathbb C X_{\alpha} \oplus \mathbb C X_{3\alpha + 2\beta}. 
$$

There are two types of regular subalgebras of dimension $2$ which contain both nilpotent and semisimple non-zero elements. Each of them is parametrized by a vector subspace $L\subset \mathfrak h$ of dimension $1$:
$$
{\mathfrak u}_{\alpha}\oplus L, \quad {\mathfrak u}_{\beta}\oplus L.
$$

Note that $X_{\beta}\in G_2$ can not be conjugate to an element of the form $x=aX_{\alpha}+bX_{\gamma}$ with $\gamma \in \{ 3\alpha +\beta , 3\alpha +2\beta \}$ and $a\neq 0$, because $rk(ad_{x})\geq 8>6=rk(ad_{X_{\beta}})$. Similarly, $X_{\alpha}\in G_2$ can not be conjugate to an element of the form $y=aX_{3\alpha + 2\beta}+bX_{\alpha}$ with $ab\neq 0$, because $rk(ad_{y})=10>8=rk(ad_{X_{\alpha}})$. Hence $\mathbb C X_{\beta} \oplus \mathbb C X_{3\alpha + 2\beta}$ can not be conjugate to any other of the listed subalgebras and if $\mathbb C X_{\alpha} \oplus \mathbb C X_{3\alpha + \beta}$ and $\mathbb C X_{\alpha} \oplus \mathbb C X_{3\alpha + 2\beta}$ are conjugate by $s\in Aut(G_2)$, then 
$$
s(X_{\alpha}) = \mu \cdot X_{\alpha},\quad s(X_{3\alpha + \beta}) = \lambda \cdot X_{3\alpha + 2\beta}, \;\; \mu, \lambda \in {\mathbb C}^{*}. 
$$

Since $s\in Aut(G_2)$ preserves the Lie bracket, $s(X_{3\alpha + 2\beta})=a'X_{\alpha} + b'X_{3\alpha + \beta} + c'X_{3\alpha + 2\beta}$, and for any $x\in \mathfrak{h}$,
$$
s(x)=x-\frac{{\beta}(x)}{2}\cdot H_{3\alpha + 2\beta} + aX_{\alpha} + bX_{3\alpha + \beta} + cX_{3\alpha + 2\beta}.
$$

Then $s\in Aut(G_2)$ preserves the Lie bracket $[x,X_{3\alpha + 2\beta}]$ only if $\beta (x)=0$ or $(\alpha + \beta )(x)=0$. Since $x\in \mathfrak{h}$ is arbitrary, we get a contradiction. Hence regular subalgebras of dimension $2$ with all elements nilpotent are not pairwise conjugate.\\

If ${\mathfrak u}_{\alpha}\oplus L$ were conjugate to ${\mathfrak u}_{\beta}\oplus L$, then its only (up to a scalar multiple) nilpotent element $X_{\alpha}$ would be conjugate to $X_{\beta}$. This can not happen as noted above.\\

Suppose $\mathbb C X_{\gamma}\oplus L$ is conjugate to $\mathbb C X_{\gamma}\oplus L'$ by $s\in Aut (G_2)$, where $\gamma \in \{ \alpha , \beta \}$. Let $L=\mathbb C x$, $x\in \mathfrak h$. Then $s(X_{\gamma})=\lambda X_{\gamma}$ and $s(x)=z + \mu X_{\gamma}$, $\lambda \in {\mathbb C}^{*}$, $\mu \in \mathbb C$, $z\in \mathfrak h$, $L'=\mathbb C z$. By composing $s$ with $\exp(c\cdot ad_{X_{\gamma}})$ with a suitable $c\in \mathbb C$, we may assume that $\mu =0$, i.e. $s(L)=L'$.\\

Since $s \in Aut (G_2)$ preserves the Lie bracket, ${\gamma} (z-x)=0$. Note that we may assume that ${\gamma}(x)={\gamma}(z)\neq 0$, because otherwise $L=L'$.\\

Let $s(H_{{\gamma}^{\perp}})=w+\sum\limits_{\mu \in \Phi} c_{\mu}\cdot X_{\mu}$, where $w\in \mathfrak h$ and ${\gamma}^{\perp}\in {\Phi}^{+}$ is such that $(\gamma , {\gamma}^{\perp})=0$. Since $s$ preserves the Lie bracket $[x, H_{{\gamma}^{\perp}}]=0$, $c_{\mu} \neq 0$ only if $\mu (z)=0$. Hence if $\mu (z)\neq 0$ for any $\mu \in \Phi$, then $s(\mathfrak h) = \mathfrak h$. In this case, $L$ and $L'$ lie in the same orbit of the Weyl group, which is possible only if $L=L'$ by the restrictions imposed on $L$ in Table~\ref{table:1}:
$$
\alpha \cdot (3\alpha +2\beta ) (L)\geq 0 \quad \mbox{and}\quad \beta \cdot (2\alpha +\beta ) (L)\geq 0 \quad \mbox{respectively}.
$$

If there exists $\mu \in {\Phi}^{+}$ such that $\mu (z)=0$, then 
$$
s(H_{{\gamma}^{\perp}})=w+c_{\mu} X_{\mu}+c_{-\mu} X_{-\mu}.
$$

Since $s$ preserves the Lie bracket $[X_{\gamma}, H_{{\gamma}^{\perp}}]=0$, $c_{\pm \mu}\neq 0$ only if $[X_{\pm \mu}, X_{\gamma}]=0$. Hence by composing $s\in Aut (G_2)$ with suitable automorphisms of the form $\exp(ad_{u})$, $u\in G_2$, we may assume that $c_{\mu}=c_{-\mu}=0$. Then $s(\mathfrak h) = \mathfrak h$ and we are done by the previous argument.\\

Now consider subalgebras of dimension $3$.\\

There are two semisimple regular subalgebras of dimension $3$:
$$
A_1 \quad \mbox{and} \quad \tilde{A_1}.
$$

There are five types of regular non-semisimple subalgebras of dimension $3$ which contain non-zero semisimple elements:
$$
{\mathfrak{u}}_{\alpha} \oplus \mathfrak{h},\quad {\mathfrak{u}}_{\beta} \oplus \mathfrak{h},
$$
$$
L\oplus \mathbb C X_{\alpha} \oplus \mathbb C X_{3\alpha +\beta}, \quad L\oplus \mathbb C X_{\alpha} \oplus \mathbb C X_{3\alpha +2\beta}, \quad L\oplus \mathbb C X_{\beta} \oplus \mathbb C X_{3\alpha +2\beta},
$$
where $L\subset \mathfrak{h}$, $dim(L)=1$.\\

There are four regular subalgebras of dimension $3$ all of whose elements are nilpotent in $G_2$:
$$
\mathbb C X_{\beta} \oplus \mathbb C X_{3\alpha +\beta} \oplus \mathbb C X_{3\alpha +2\beta}, \quad \mathbb C X_{\alpha} \oplus \mathbb C X_{2\alpha +\beta} \oplus \mathbb C X_{3\alpha +\beta},
$$
$$
\mathbb C X_{\beta} \oplus \mathbb C X_{\alpha +\beta} \oplus \mathbb C X_{3\alpha +2\beta}, \quad \mathbb C X_{\alpha} \oplus \mathbb C X_{3\alpha +\beta} \oplus \mathbb C X_{3\alpha +2\beta}.
$$

If $A_1$ were conjugate to $\tilde{A_1}$, then $X_{\beta}$ would be conjugate to $s(X_{\beta})=aH_{\alpha}+bX_{\alpha}+cX_{-\alpha}$, $s\in Aut(G_2)$. Then 
$$
0=ad^{3}_{s(X_{\beta})}(X_{\beta})=6b^3\cdot X_{3\alpha+\beta}+...
$$

Hence $b=0$. Then $a=0$, because $s(X_{\beta})$ is nilpotent. Hence $X_{\beta}$ would be conjugate to $X_{-\alpha}$, which was shown to be impossible above.\\

If ${\mathfrak{u}}_{\alpha} \oplus \mathfrak{h}$ were conjugate to ${\mathfrak{u}}_{\beta} \oplus \mathfrak{h}$, then ${\mathfrak{u}}_{\alpha}$ and ${\mathfrak{u}}_{\beta}$ would be conjugate. It was shown above to be not the case.\\

Suppose $L\oplus \mathbb C X_{\alpha} \oplus \mathbb C X_{\gamma}$ is conjugate to $L'\oplus \mathbb C X_{\alpha} \oplus \mathbb C X_{\gamma}$ by $s\in Aut(G_2)$, where $\gamma \in \{ 3\alpha +\beta , 3\alpha +2\beta \}$. Then $s(X_{\gamma})=\lambda X_{\gamma}$, $\lambda \in {\mathbb C}^{*}$.\\

Let $s(X_{\alpha})=b_0X_{\alpha}+b_1X_{\gamma}$, $s(x)=z+a_0X_{\alpha}+a_1X_{\gamma}$, $L=\mathbb C x$, $L'=\mathbb C z$. Since $s\in Aut(G_2)$ preserves the Lie bracket $[x,X_{\gamma}]$, $\gamma (x-z)=0$. Since $s$ preserves the Lie bracket $[x,X_{\alpha}]$ and $b_0\neq 0$, ${\alpha}(x)={\alpha}(z)$. Hence $x=z$, i.e. $L=L'$.\\

Suppose $L\oplus \mathbb C X_{\beta} \oplus \mathbb C X_{3\alpha +2\beta}$ is conjugate to $L'\oplus \mathbb C X_{\beta} \oplus \mathbb C X_{3\alpha +2\beta}$ by $s\in Aut(G_2)$. Let $s(X_{\beta})=b_0X_{\beta}+b_1X_{3\alpha +2\beta}$, $s(X_{3\alpha +2\beta})=c_0X_{\beta}+c_1X_{3\alpha +2\beta}$, $s(x)=z+a_0X_{\beta}+a_1X_{3\alpha +2\beta}$, $L=\mathbb C x$, $L'=\mathbb C z$.\\

Since $s\in Aut(G_2)$ preserves the Lie bracket $[x,X_{\beta}]$, either $\beta (x)= \beta (z)$ if $b_0\neq 0$, or $\beta (x) = (3\alpha +2\beta) (z)$ if $b_1\neq 0$. Since $s$ preserves the Lie bracket $[x,X_{3\alpha+2\beta}]$, either $(3\alpha+2\beta) (x)= (3\alpha+2\beta) (z)$ if $c_1\neq 0$, or $\beta (z) = (3\alpha +2\beta) (x)$ if $c_0\neq 0$. Hence either $L=L'$ or $b_0=c_1=0$, i.e. $s(X_{\beta})=\lambda \cdot X_{3\alpha+2\beta}$, $s(X_{3\alpha+2\beta})=\nu \cdot X_{\beta}$, $\lambda , \nu\in { \mathbb C}^{*}$. In the latter case, $L'=w(L)$, where $w$ is an element of the Weyl group interchanging roots $\beta$ and $3\alpha +2\beta$ (i.e. a reflection in the hyperplane spanned by $\alpha +\beta$). Hence either $L=L'$ or at most one of $L$, $L'$ satisfies the condition imposed on $L$ in Table~\ref{table:1}:
$$
(\alpha + \beta)\cdot (3\alpha + \beta) (L)\geq 0.
$$

Suppose $\mathbb C X_{\beta} \oplus \mathbb C X_{3\alpha +\beta} \oplus \mathbb C X_{3\alpha +2\beta}$ is conjugate to $\mathbb C X_{\alpha} \oplus \mathbb C X_{2\alpha +\beta} \oplus \mathbb C X_{3\alpha +\beta}$ by $s\in Aut(G_2)$. If $s(X_{\beta})=aX_{\alpha}+bX_{2\alpha +\beta}+cX_{3\alpha +\beta}$, then $X_{\beta}$ is conjugate to $x=aX_{2\alpha +\beta}+bX_{\alpha +\beta}+cX_{3\alpha +2\beta}$. Hence 
$$
0=ad^3_x(X_{-(3\alpha +\beta )})=-6a^3\cdot X_{3\alpha+2\beta}, \quad 0=ad^3_x(X_{-\beta})=-6b^3\cdot X_{3\alpha+2\beta}.
$$
Hence $a=b=0$, i.e. $s(X_{\beta})\in \mathbb C X_{3\alpha+\beta}$. For the same reason, $s(X_{3\alpha + \beta})\in \mathbb C X_{3\alpha+\beta}$. This is impossible.\\

Suppose $\mathbb C X_{\beta} \oplus \mathbb C X_{\alpha +\beta} \oplus \mathbb C X_{3\alpha +2\beta}$ is conjugate to $\mathbb C X_{\alpha} \oplus \mathbb C X_{3\alpha +\beta} \oplus \mathbb C X_{3\alpha +2\beta}$ by $s\in Aut(G_2)$. If $s(X_{\beta})=aX_{\alpha}+bX_{3\alpha+\beta}+cX_{3\alpha+2\beta}$, then $X_{\beta}$ is conjugate to $x=aX_{2\alpha +\beta}+bX_{3\alpha+2\beta}+cX_{\beta}$. Then 
$$
0=ad^3_{x}(X_{-(3\alpha +\beta)})=-6a^3\cdot X_{3\alpha +2\beta}.
$$

Hence $s(\mathbb C X_{\beta} \oplus \mathbb C X_{3\alpha +2\beta})\subset \mathbb C X_{3\alpha +\beta} \oplus \mathbb C X_{3\alpha +2\beta}$.\\

If $s(X_{\alpha +\beta})=aX_{\alpha}+bX_{3\alpha+\beta}+cX_{3\alpha+2\beta}$, then $X_{\alpha + \beta}$ is conjugate to $x=aX_{2\alpha +\beta}+bX_{3\alpha+2\beta}+cX_{\beta}$. Note that $rk(ad_{s(X_{\alpha + \beta})})=8$ and $rk(ad_x)\geq 9$ unless $a\cdot c=0$. Since $a\neq 0$, $c=0$. By composing $s$ with $\exp (u\cdot ad_{X_{2\alpha +\beta}})$, we can ensure that $s(X_{\alpha +\beta})=\lambda X_{\alpha}$, $\lambda \in {\mathbb C}^{*}$.\\

Let $s(x)=z+\sum\limits_{\gamma \in \Phi} u_{\gamma} X_{\gamma}$, where $x, z\in \mathfrak{h}$. Let $s(X_{\beta})=a_0X_{3\alpha+\beta}+a_1X_{3\alpha +2\beta}$, $s(X_{3\alpha + 2\beta})=b_0X_{3\alpha+\beta}+b_1X_{3\alpha +2\beta}$.\\

Since $s$ preserves the Lie bracket $[x, X_{\alpha +\beta}]$, $\alpha (z)=(\alpha +\beta) (x)$. Since $s$ preserves the Lie bracket $[x, X_{\beta}]$, 
$$
a_0 \cdot \beta (x)=a_0\cdot (3\alpha +\beta)(z) + u_{-\beta}\cdot a_1,\;\; a_1\cdot \beta (x)=a_1\cdot (3\alpha +2\beta )(z).
$$

Since $s$ preserves the Lie bracket $[x, X_{3\alpha + 2\beta}]$, 
$$
b_0 \cdot (3\alpha + 2\beta )(x)=b_0\cdot (3\alpha + \beta )(z) + u_{-\beta}\cdot b_1,\;\; b_1\cdot (3\alpha +2 \beta ) (x)=b_1\cdot (3\alpha +2\beta ) (z).
$$

So, if $b_1\neq 0$, then $\beta (z)=-\beta (x)/2$, $\alpha (z)=(\alpha + \beta )(x)$. If $a_1\neq 0$, then $(3\alpha +\beta )(x)=0$ for any $x\in \mathfrak{h}$, a contradiction. If $a_1=0$ and $a_0\neq 0$, then $(2\alpha +\beta )(x)=0$ for any $x\in \mathfrak{h}$, a contradiction.\\

Hence $b_1=0$. Then $\beta (z)=-\beta (x)$, $\alpha (z)=(\alpha + \beta )(x)$. If $a_1\neq 0$, then $3\alpha (x)=0$ for any $x\in \mathfrak{h}$, a contradiction. If $a_1=0$ and $a_0\neq 0$, then $(3\alpha +\beta )(x)=0$ for any $x\in \mathfrak{h}$, a contradiction.\\

Now consider subalgebras of dimension $4$.\\

There are two regular subalgebras of dimension $4$ with the $1$-dimensional radical generated by a semisimple element:
$$
A_1\oplus \mathbb C H_{2\alpha +\beta},\quad \tilde{A_1}\oplus \mathbb C H_{3\alpha +2\beta}.
$$

There are two regular subalgebras of dimension $4$ with the $1$-dimensional radical generated by a nilpotent element:
$$
A_1\oplus \mathbb C X_{2\alpha +\beta} ,\quad \tilde{A_1}\oplus \mathbb C X_{3\alpha +2\beta}.
$$

There are three regular solvable subalgebras of dimension $4$ with $2$-dimensional subalgebras of nilpotent elements:
$$
\mathfrak{h}\oplus \mathbb C X_{\alpha} \oplus \mathbb C X_{3\alpha +\beta} ,\quad \mathfrak{h}\oplus \mathbb C X_{\alpha} \oplus \mathbb C X_{3\alpha +2\beta} ,\quad \mathfrak{h}\oplus \mathbb C X_{\beta} \oplus \mathbb C X_{3\alpha +2\beta} .
$$

There are four types of regular solvable subalgebras of dimension $4$ with $3$-dimensional subalgebras of nilpotent elements:
$$
L\oplus \mathbb C X_{\alpha} \oplus \mathbb C X_{2\alpha +\beta} \oplus \mathbb C X_{3\alpha +\beta} ,\quad L\oplus \mathbb C X_{\beta} \oplus \mathbb C X_{3\alpha +\beta} \oplus \mathbb C X_{3\alpha +2\beta},
$$
$$
L\oplus \mathbb C X_{\alpha} \oplus \mathbb C X_{3\alpha +\beta} \oplus \mathbb C X_{3\alpha +2\beta} ,\quad L\oplus \mathbb C X_{\beta} \oplus \mathbb C X_{\alpha +\beta} \oplus \mathbb C X_{3\alpha +2\beta},
$$
parametrized by a vector subspace $L\subset \mathfrak{h}$, $dim(L)=1$.\\

There are two regular nilpotent subalgebras of dimension $4$:
$$
\mathbb C X_{\alpha} \oplus \mathbb C X_{2\alpha +\beta} \oplus \mathbb C X_{3\alpha +\beta} \oplus \mathbb C X_{3\alpha +2\beta} ,\quad \mathbb C X_{\beta} \oplus \mathbb C X_{2\alpha +\beta} \oplus \mathbb C X_{3\alpha +\beta} \oplus \mathbb C X_{3\alpha +2\beta}.
$$

Let $\mathfrak{u}_{\beta} + \mathfrak{u}_{2\alpha +\beta} + \mathfrak{u}_{3\alpha +\beta} + \mathfrak{u}_{3\alpha +2\beta}$ be conjugate to $\mathfrak{u}_{\alpha} + \mathfrak{u}_{2\alpha +\beta} + \mathfrak{u}_{3\alpha +\beta} + \mathfrak{u}_{3\alpha +2\beta}$ by $s\in Aut(G_2)$.\\

Then $s(X_{\gamma})=aX_{\alpha}+bX_{2\alpha+\beta}+cX_{3\alpha+\beta}+dX_{3\alpha+2\beta}$, $\gamma \in \{ \beta , 3\alpha +\beta , 3\alpha +2\beta \}$, is conjugate to $x=aX_{2\alpha+\beta}+bX_{\alpha+\beta}+cX_{3\alpha+2\beta}+dX_{\beta}$. Then 
$$
0=ad^3_x(X_{-(3\alpha+\beta)})=-6a^3\cdot X_{3\alpha+2\beta},\quad 0=ad^3_x(X_{-\beta})=(9abd-6b^3)\cdot X_{3\alpha+2\beta}.
$$
Hence $a=b=0$, i.e. $s(\mathbb C X_{\beta} \oplus \mathbb C X_{3\alpha +\beta} \oplus \mathbb C X_{3\alpha +2\beta})\subset \mathbb C X_{3\alpha +\beta} \oplus \mathbb C X_{3\alpha +2\beta}$, a contradiction.\\

Suppose $L\oplus \mathfrak{g}$ is conjugate to $L'\oplus \mathfrak{g}$ by $s\in Aut(G_2)$, where $\mathfrak{g}= \mathfrak{u}_{\alpha} + \mathfrak{u}_{2\alpha +\beta} + \mathfrak{u}_{3\alpha +\beta}$.\\

Then $s(\mathfrak{u}_{3\alpha +\beta})=s([\mathfrak{g},\mathfrak{g}])=[s(\mathfrak{g}),s(\mathfrak{g})]=\mathfrak{u}_{3\alpha +\beta}$, i.e. $s(X_{3\alpha +\beta})=\lambda \cdot X_{3\alpha +\beta}$.\\

Let $s(X_{\alpha})=aX_{\alpha}+bX_{2\alpha +\beta}+cX_{3\alpha +\beta}$. If $a\cdot b\neq 0$, then $rk (ad_{s(X_{\alpha})})=10>8=rk(ad_{X_{\alpha}})$. Hence $a\cdot b=0$. Similarly, $s(X_{2\alpha +\beta})=a_0X_{\alpha}+b_0X_{2\alpha +\beta}+c_0X_{3\alpha +\beta}$, where $a_0\cdot b_0=0$.\\

Let $s(x)=z+u_0X_{\alpha}+u_1X_{2\alpha +\beta}+u_2X_{3\alpha +\beta}$, $L=\mathbb C x$, $L'=\mathbb C z$.\\

Since $s$ preserves the Lie bracket $[x,X_{3\alpha +\beta}]$, $(3\alpha +\beta)(x-z)=0$.\\

Since $s$ preserves the Lie bracket $[x,X_{\alpha}]$,
$$
a\cdot \alpha (x-z) =0,\quad b\cdot \left( (2\alpha +\beta)(z)-\alpha (x)\right) =0.
$$

If $a\neq 0$, then $\alpha (x-z) =0$, i.e. $x=z$. Hence $L=L'$.\\

If $a=0$, then $b\neq 0$ and $a_0\neq 0$. Hence $b_0=0$. Compose $s$ with an element $w\in Aut(G_2)$ of the Weyl group, which fixes $3\alpha +\beta$ and interchanges $\alpha$ and $2\alpha +\beta$. Then $L\oplus \mathfrak{g}$ is conjugate to $w(L')\oplus \mathfrak{g}$ by an automorphism, which maps $X_{\alpha}$ to $b'X_{\alpha}+c'X_{3\alpha +\beta}$, $X_{2\alpha +\beta}$ to $a_0'X_{2\alpha +\beta}+c_0'X_{3\alpha +\beta}$ and $X_{3\alpha +\beta}$ to ${\lambda}'\cdot X_{3\alpha +\beta}$.\\

As we have seen, in this case, $L=w(L')$. Hence $L=L'$, because otherwise only one of $L$, $L'$ satisfies the condition imposed on $L$ in Table~\ref{table:1}:
$$
(\alpha +\beta)\cdot (3\alpha +\beta)(L)\geq 0.
$$

Suppose $L\oplus \mathfrak{g}$ is conjugate to $L'\oplus \mathfrak{g}$ by $s\in Aut(G_2)$, where $\mathfrak{g}= \mathfrak{u}_{\beta} + \mathfrak{u}_{3\alpha +\beta} + \mathfrak{u}_{3\alpha +2\beta}$.\\

Then $s(\mathfrak{u}_{3\alpha +2\beta})=s([\mathfrak{g},\mathfrak{g}])=[s(\mathfrak{g}),s(\mathfrak{g})]=\mathfrak{u}_{3\alpha +2\beta}$, i.e. $s(X_{3\alpha +2\beta})=\lambda \cdot X_{3\alpha +2\beta}$.\\

If $s(X_{\beta})=a_0X_{\beta}+a_1X_{3\alpha +\beta}+a_2X_{3\alpha +2\beta}$, then $0=ad^3_{s(X_{\beta})}(X_{-\beta})=3a_0^2a_1\cdot X_{3\alpha +2\beta}$, i.e. $a_0a_1=0$. Similarly, $s(X_{3\alpha +\beta})=b_0X_{\beta}+b_1X_{3\alpha +\beta}+b_2X_{3\alpha +2\beta}$, where $b_0b_1=0$. Hence either $a_0=b_1=0$ or $a_1=b_0=0$.\\

In the former case, compose $s$ with a Weyl group element $w\in Aut(G_2)$ acting on ${\mathfrak{h}}^{*}$ as a reflection through the line spanned by $3\alpha +2\beta$, i.e. $w$ interchanges $\beta$ with $3\alpha +\beta$ and fixes $3\alpha +2\beta$.\\

Hence we may assume that $s(X_{\beta})=a_0X_{\beta}+a_2X_{3\alpha +2\beta}$, $s(X_{3\alpha +\beta})=b_1X_{3\alpha +\beta}+b_2X_{3\alpha +2\beta}$. Let $s(x)=z+c_0X_{\beta}+c_1X_{3\alpha +\beta}+c_2X_{3\alpha +2\beta}$, $L=\mathbb C x$, $L'=\mathbb C z$.\\

Since $s$ preserves the Lie bracket $[x, X_{\beta}]$, $\beta (x-z)=0$. Since $s$ preserves the Lie bracket $[x, X_{3\alpha + \beta}]$, $(3\alpha + \beta )(x-z)=0$. Hence $x=z$, and so $L=L'$ or $L=w(L')$. In the latter case, $L=L'$ by the condition imposed on $L$ in Table~\ref{table:1}:
$$
\alpha \cdot (3\alpha +2\beta)(L)\geq 0.
$$

Suppose $L\oplus \mathbb C X_{\beta} \oplus \mathbb C X_{\alpha +\beta} \oplus \mathbb C X_{3\alpha +2\beta}$ is conjugate to $L'\oplus \mathbb C X_{\beta} \oplus \mathbb C X_{\alpha +\beta} \oplus \mathbb C X_{3\alpha +2\beta}$ by $s\in Aut(G_2)$.\\

If $s(X_{\beta})=a_0X_{\beta}+a_1X_{\alpha +\beta}+a_2X_{3\alpha +2\beta}$, then $0=ad^3_{s(X_{\beta})}(X_{-\beta})=-6a_1^3\cdot X_{3\alpha +2\beta}$. Hence $s(X_{\beta})=a_0X_{\beta}+a_2X_{3\alpha +2\beta}$, $s(X_{3\alpha +2\beta})=b_0X_{\beta}+b_2X_{3\alpha +2\beta}$.\\

Let $s(x)=z+c_0X_{\beta}+c_1X_{\alpha +\beta}+c_2X_{3\alpha +2\beta}$, $L=\mathbb C x$, $L'=\mathbb C z$.\\

Since $s$ preserves the Lie brackets $[x,X_{\beta}]$ and $[x,X_{3\alpha + 2\beta}]$,
$$
a_0\cdot \beta (x-z) = b_2\cdot (3\alpha +2\beta) (x-z)=0,
$$
$$
a_2\cdot \left( (3\alpha +2\beta)(z)-\beta (x)\right)=0,\quad b_0\cdot \left( (3\alpha +2\beta)(x)-\beta (z)\right)=0.
$$

So, if $a_0b_2\neq 0$, then $x=z$, i.e. $L=L'$. If $a_0=0$, then compose $s$ with a Weyl group element $w\in Aut(G_2)$ acting on ${\mathfrak{h}}^{*}$ as a reflection through $\alpha +\beta$. I.e. $w$ fixes $\alpha +\beta$ and interchanges $\beta$ with $3\alpha +2\beta$. Upon this composition, we may assume that $a_0\neq 0$ and $a_2=0$. Hence $b_2\neq 0$. By the previous analysis, $L=w(L')$. Hence $L=L'$ by the condition imposed on $L$ in Table~\ref{table:1}:
$$
(\alpha +\beta ) \cdot (3\alpha +\beta)(L)\geq 0.
$$

Suppose $L\oplus \mathbb C X_{\alpha} \oplus \mathbb C X_{3\alpha +\beta} \oplus \mathbb C X_{3\alpha +2\beta}$ is conjugate to $L'\oplus \mathbb C X_{\alpha} \oplus \mathbb C X_{3\alpha +\beta} \oplus \mathbb C X_{3\alpha +2\beta}$ by $s\in Aut(G_2)$.\\

If $s(X_{3\alpha + \beta})=a_0X_{\alpha}+a_1X_{3\alpha +\beta}+a_2X_{3\alpha +2\beta}$, then $X_{3\alpha + \beta}$ is conjugate to $x=a_0X_{2\alpha +\beta}+a_1X_{3\alpha +2\beta}+a_2X_{\beta}$. Then $0=ad^3_x (X_{-(3\alpha +\beta)})=-6a_0^3\cdot X_{3\alpha +2\beta}$. Hence $s(X_{3\alpha + \beta})=a_1X_{3\alpha +\beta}+a_2X_{3\alpha +2\beta}$, $s(X_{3\alpha + 2\beta})=b_1X_{3\alpha +\beta}+b_2X_{3\alpha +2\beta}$.\\

Let $s(x)=z+u_0X_{\alpha}+u_1X_{3\alpha +\beta}+u_2X_{3\alpha +2\beta}$, $L=\mathbb C x$, $L'=\mathbb C z$. Since $s$ preserves the Lie brackets $[x,X_{3\alpha +\beta}]$ and $[x,X_{3\alpha +2\beta}]$,
$$
a_1 \cdot (3\alpha +\beta)(x-z)=b_2 \cdot (3\alpha +2\beta) (x-z)=0,
$$
$$
a_2\cdot ((3\alpha +2\beta)(z)-(3\alpha +\beta)(x))=b_1\cdot ((3\alpha +\beta)(z)-(3\alpha +2\beta)(x))=0.
$$

If $a_1b_2\neq 0$, then $x=z$, i.e. $L=L'$. Otherwise, $a_2b_1\neq 0$. Then $(6\alpha +3\beta)(x-z)=0$.\\

Let $s(X_{\alpha })=c_0X_{\alpha}+c_1X_{3\alpha +\beta}+c_2X_{3\alpha +2\beta}$. Since $s$ preserves the Lie bracket $[x, X_{\alpha}]$ and $c_0\neq 0$, $\alpha (x-z)=0$. Hence $x=z$, i.e. $L=L'$.\\

Now consider subalgebras of dimension $5$.\\

There are two regular subalgebras of dimension $5$ with the $2$-dimensional radical containing a non-zero semisimple element:
$$
A_1\oplus \mathbb C H_{2\alpha +\beta} \oplus \mathbb C X_{2\alpha +\beta},\quad \tilde{A_1}\oplus \mathbb C H_{3\alpha +2\beta} \oplus \mathbb C X_{3\alpha +2\beta} .
$$

They are not conjugate, because otherwise the nilpotent generators $X_{3\alpha +2\beta}$ and $X_{2\alpha +\beta}$ of their radicals would be conjugate too. It was shown above that this is false.\\

There is a regular subalgebra of dimension $5$ with the $2$-dimensional radical consisting of nilpotent elements:
$$
A_1\oplus \mathbb C X_{3\alpha +\beta} \oplus \mathbb C X_{3\alpha +2\beta} .
$$

There are four regular solvable subalgebras of dimension $5$ with $3$-dimensional subalgebras of nilpotent elements:
$$
\mathfrak{h}\oplus \mathbb C X_{\alpha} \oplus \mathbb C X_{2\alpha +\beta} \oplus \mathbb C X_{3\alpha +\beta} ,\quad \mathfrak{h}\oplus \mathbb C X_{\beta} \oplus \mathbb C X_{\alpha +\beta} \oplus \mathbb C X_{3\alpha +2\beta} ,
$$
$$
\mathfrak{h}\oplus \mathbb C X_{\beta} \oplus \mathbb C X_{3\alpha +\beta} \oplus \mathbb C X_{3\alpha +2\beta} ,\quad \mathfrak{h}\oplus \mathbb C X_{\alpha} \oplus \mathbb C X_{3\alpha +\beta} \oplus \mathbb C X_{3\alpha +2\beta}.
$$

There are two types of regular solvable subalgebras of dimension $5$ with $4$-dimensional subalgebras of nilpotent elements:
$$
L\oplus \mathbb C X_{\alpha} \oplus \mathbb C X_{2\alpha +\beta} \oplus \mathbb C X_{3\alpha +\beta} \oplus \mathbb C X_{3\alpha +2\beta} ,\quad L\oplus \mathbb C X_{\beta} \oplus \mathbb C X_{2\alpha +\beta} \oplus \mathbb C X_{3\alpha +\beta} \oplus \mathbb C X_{3\alpha +2\beta},
$$
parametrized by a vector subspace $L\subset \mathfrak{h}$, $dim(L)=1$.\\

There are two regular nilpotent subalgebras of dimension $5$:
$$
\mathbb C X_{\alpha} \oplus \mathbb C X_{\alpha +\beta} \oplus \mathbb C X_{2\alpha +\beta} \oplus \mathbb C X_{3\alpha +\beta} \oplus \mathbb C X_{3\alpha +2\beta} ,\quad \mathbb C X_{\beta} \oplus \mathbb C X_{\alpha +\beta} \oplus \mathbb C X_{2\alpha +\beta} \oplus \mathbb C X_{3\alpha +\beta} \oplus \mathbb C X_{3\alpha +2\beta}.
$$

These two subalgebras are not conjugate, because otherwise so would be their derived ideals, which are $\mathbb C X_{2\alpha +\beta} \oplus \mathbb C X_{3\alpha +\beta} \oplus \mathbb C X_{3\alpha +2\beta}$ and $\mathbb C X_{3\alpha +2\beta}$ respectively.\\

Suppose $L\oplus \mathfrak{g}$ and $L'\oplus \mathfrak{g}$ are conjugate by $s\in Aut(G_2)$, where $\mathfrak{g}= \mathfrak{u}_{\alpha} + \mathfrak{u}_{2\alpha +\beta} + \mathfrak{u}_{3\alpha +\beta} + \mathfrak{u}_{3\alpha +2\beta}$.\\

Then $\mathfrak{u}_{3\alpha +\beta}=[\mathfrak{g},\mathfrak{g}]$ is mapped to itself, i.e. $s(X_{3\alpha +\beta})=\lambda X_{3\alpha +\beta}$.\\

If $s(X_{3\alpha +2\beta})=a_0X_{\alpha}+a_1X_{2\alpha+\beta}+a_2X_{3\alpha+\beta}+a_3X_{3\alpha+2\beta}$, then $X_{3\alpha +2\beta}$ is conjugate to $x=a_0X_{2\alpha +\beta}+a_1X_{\alpha+\beta}+a_2X_{3\alpha+2\beta}+a_3X_{\beta}$. Hence 
$$
0=ad^3_x (X_{-(3\alpha +\beta)})=-6a_0^3\cdot X_{3\alpha +2\beta},\quad 0=ad^3_x (X_{-\beta})=(9a_0a_1a_3-6a_1^3)\cdot X_{3\alpha +2\beta}.
$$

So, $s(X_{3\alpha +2\beta})=a_2X_{3\alpha+\beta}+a_3X_{3\alpha+2\beta}$.\\

Let $s(x)=z+c_0X_{\alpha}+c_1X_{2\alpha+\beta}+c_2X_{3\alpha+\beta}+c_3X_{3\alpha+2\beta}$, $L=\mathbb C x$, $L'=\mathbb C z$. Since $s$ preserves the Lie brackets $[x,X_{3\alpha +\beta}]$ and $[x,X_{3\alpha +2\beta}]$,
$$
\lambda \cdot (3\alpha +\beta) (x-z)=a_3\cdot (3\alpha +2\beta) (x-z)=0.
$$

Since $\lambda a_3 \neq 0$, then $x=z$, i.e. $L=L'$.\\

Suppose $L\oplus \mathfrak{g}$ and $L'\oplus \mathfrak{g}$ are conjugate by $s\in Aut(G_2)$, where $\mathfrak{g}=\mathfrak{u}_{\beta} + \mathfrak{u}_{2\alpha +\beta} + \mathfrak{u}_{3\alpha +\beta} + \mathfrak{u}_{3\alpha +2\beta}$.\\

Then $\mathfrak{u}_{3\alpha +2\beta}=[\mathfrak{g},\mathfrak{g}]$ is mapped to itself, i.e. $s(X_{3\alpha +2\beta})=\lambda X_{3\alpha +2\beta}$.\\

If $s(X_{3\alpha +\beta})=a_0X_{\beta}+a_1X_{2\alpha+\beta}+a_2X_{3\alpha+\beta}+a_3X_{3\alpha+2\beta}$, then 
$$
0=ad^3_{s(X_{3\alpha +\beta})} (X_{-\beta})=3a_0^2a_2\cdot X_{3\alpha +2\beta},\quad 0=ad^3_{s(X_{3\alpha +\beta})} (X_{-(\alpha+\beta)})=18a_0a_1^2\cdot X_{3\alpha +2\beta},
$$
$$
0=ad^3_{s(X_{3\alpha +\beta})} (X_{-(3\alpha +\beta)})=(-6a_1^3-3a_0a_2^2)\cdot X_{3\alpha +2\beta}.
$$

Hence $s(X_{3\alpha +\beta})=a_0X_{\beta}+a_2X_{3\alpha+\beta}+a_3X_{3\alpha+2\beta}$, where $a_0a_2=0$.\\

Similarly, $s(X_{\beta})=b_0X_{\beta}+b_2X_{3\alpha+\beta}+b_3X_{3\alpha+2\beta}$, where $b_0b_2=0$. Hence either $a_0=b_2=0$ or $a_2=b_0=0$.\\

Let $s(x)=z+u_0X_{\beta}+u_1X_{2\alpha+\beta}+u_2X_{3\alpha+\beta}+u_3X_{3\alpha+2\beta}$. Since $s$ preserves the Lie bracket $[x, X_{3\alpha +2\beta}]$, $(3\alpha +2\beta)(x-z)=0$, i.e. $x=z+\mu \cdot H_{\alpha}$.\\

Since $s$ preserves the Lie brackets $[x, X_{3\alpha +\beta}]$ and $[x, X_{\beta}]$, 
$$
a_0\cdot ((3\alpha +\beta)(x)-\beta (z))=a_2\cdot (3\alpha +\beta)(x-z)=0,
$$
$$
b_2\cdot ((3\alpha +\beta)(z)-\beta (x))=b_0\cdot \beta (x-z)=0.
$$

If $a_0=b_2=0$, then $a_2b_0 \neq 0$. Hence $x=z$, i.e. $L=L'$.\\

If $a_2=b_0=0$, then $a_0b_2 \neq 0$. Hence $(3\alpha +2\beta)(x-z)=0$.\\

If $s(X_{2\alpha +\beta})=c_0X_{\beta}+c_1X_{2\alpha+\beta}+c_2X_{3\alpha+\beta}+c_3X_{3\alpha+2\beta}$, then $c_1\neq 0$. Since $s$ preserves the Lie bracket $[x, X_{2\alpha +\beta}]$, $(2\alpha +\beta)(x-z)=0$. Hence $x=z$ and so $L=L'$.\\

Now consider subalgebras of dimension $6$.\\

There is one regular semisimple subalgebra of dimension $6$:
$$
A_1+\tilde{A_1}.
$$

There is one regular subalgebra of dimension $6$ with the $3$-dimensional radical containing a non-zero semisimple element:
$$
A_1\oplus \mathbb C H_{2\alpha +\beta } \oplus \mathbb C X_{3\alpha +\beta}\oplus \mathbb C X_{3\alpha +2\beta}.
$$

There is one regular subalgebra of dimension $6$ with the $3$-dimensional radical consisting of nilpotent elements:
$$
A_1\oplus \mathbb C X_{2\alpha +\beta} \oplus \mathbb C X_{3\alpha +\beta}\oplus \mathbb C X_{3\alpha +2\beta}.
$$

There are two regular solvable subalgebras of dimension $6$ with $4$-dimensional subalgebras of nilpotent elements:
$$
\mathfrak{h}\oplus \mathbb C X_{\alpha} \oplus \mathbb C X_{2\alpha +\beta} \oplus \mathbb C X_{3\alpha +\beta} \oplus \mathbb C X_{3\alpha +2\beta} ,\quad \mathfrak{h}\oplus \mathbb C X_{\beta} \oplus \mathbb C X_{2\alpha +\beta} \oplus \mathbb C X_{3\alpha +\beta} \oplus \mathbb C X_{3\alpha +2\beta}.
$$

There are two types of regular solvable subalgebras of dimension $6$ with $5$-dimensional subalgebras of nilpotent elements:
$$
L\oplus \mathbb C X_{\alpha} \oplus \mathbb C X_{\alpha +\beta} \oplus \mathbb C X_{2\alpha +\beta} \oplus \mathbb C X_{3\alpha +\beta} \oplus \mathbb C X_{3\alpha +2\beta} ,\quad L\oplus \mathbb C X_{\beta} \oplus \mathbb C X_{\alpha +\beta} \oplus \mathbb C X_{2\alpha +\beta} \oplus \mathbb C X_{3\alpha +\beta} \oplus \mathbb C X_{3\alpha +2\beta},
$$
parametrized by a vector subspace $L\subset \mathfrak{h}$, $dim(L)=1$.\\

There is a regular nilpotent subalgebra of dimension $6$:
$$
\mathbb C X_{\alpha} \oplus \mathbb C X_{\beta} \oplus \mathbb C X_{\alpha +\beta} \oplus \mathbb C X_{2\alpha +\beta} \oplus \mathbb C X_{3\alpha +\beta} \oplus \mathbb C X_{3\alpha +2\beta}.
$$

Let $\mathfrak{g}=\mathbb C X_{\beta} \oplus \mathbb C X_{\alpha +\beta} \oplus \mathbb C X_{2\alpha +\beta} \oplus \mathbb C X_{3\alpha +\beta} \oplus \mathbb C X_{3\alpha +2\beta}$. Suppose $L\oplus \mathfrak{g}$ is conjugate to $L'\oplus \mathfrak{g}$ by $s\in Aut(G_2)$.\\

Since $[\mathfrak{g},\mathfrak{g}]=\mathbb C X_{3\alpha +2\beta}$ maps to itself, $s(X_{3\alpha +2\beta})=\lambda X_{3\alpha +2\beta}$.\\

Let 
$$
s(x)=z+\sum\limits_{\gamma \in {\Phi}^{+}\setminus \{ \alpha \}} c_{\gamma}X_{\gamma}, \quad s(y)=w+\sum\limits_{\rho \in {\Phi}} u_{\rho}X_{\rho},
$$
where $L=\mathbb C x$, $L'=\mathbb C z$, $y\in \mathfrak{h}$ is arbitrary, $w\in \mathfrak{h}$. Since $s$ preserves the Lie bracket $[y, X_{3\alpha +2\beta}]$, $u_{\rho}=0$ unless $\rho \in {\Phi}^{+}\cup \{ -\alpha \}$. Then
$$
0=s([x,y])=[s(x),s(y)]=u_{-\alpha}\cdot (-\alpha)(z)\cdot X_{-\alpha}+u_{\alpha}\cdot \alpha (z)\cdot X_{\alpha}+...
$$

Hence either $\alpha (z)=0$ or $u_{\alpha}=u_{-\alpha}=0$. If $\alpha (L)=\alpha (L')=0$, then $L=L'$. Otherwise, $u_{\alpha}=u_{-\alpha}=0$, i.e. $s(\mathfrak{h}\oplus \mathfrak{g})\subset \mathfrak{h}\oplus \mathfrak{g}$. By the conjugacy of Cartan subalgebras, we may assume that $s(\mathfrak{h})= \mathfrak{h}$. Hence $L$ and $L'$ lie in the same orbit of the Weyl group. Hence $L=L'$ by the condition imposed on $L$ in Table~\ref{table:1}:
$$
\alpha \cdot (3\alpha +2\beta) (L)\geq 0.
$$

Let $\mathfrak{g}=\mathbb C X_{\alpha} \oplus \mathbb C X_{\alpha +\beta} \oplus \mathbb C X_{2\alpha +\beta} \oplus \mathbb C X_{3\alpha +\beta} \oplus \mathbb C X_{3\alpha +2\beta}$. Suppose $L\oplus \mathfrak{g}$ is conjugate to $L'\oplus \mathfrak{g}$ by $s\in Aut(G_2)$. Then $[\mathfrak{g},\mathfrak{g}]=\mathbb C X_{2\alpha +\beta} \oplus \mathbb C X_{3\alpha +\beta} \oplus \mathbb C X_{3\alpha +2\beta}$ and $[\mathfrak{g},[\mathfrak{g},\mathfrak{g}]]=\mathbb C X_{3\alpha +\beta} \oplus \mathbb C X_{3\alpha +2\beta}$ map to themselves.\\

Hence $s(X_{3\alpha +\beta})=a_3X_{3\alpha +\beta}+a_4X_{3\alpha +2\beta}$, $s(X_{3\alpha +2\beta})=b_3X_{3\alpha +\beta}+b_4X_{3\alpha +2\beta}$ and $s(X_{2\alpha +\beta})=c_2X_{2\alpha +\beta}+c_3X_{3\alpha +\beta}+c_4X_{3\alpha +2\beta}$.\\

Let $s(x)=z+\sum\limits_{\gamma\in {\Phi}^{+}\setminus\{\beta\}} u_{\gamma} X_{\gamma}$, $L=\mathbb C x$, $L'=\mathbb C z$. Since $s$ preserves the Lie bracket $[x,X_{2\alpha +\beta}]$ and $c_2\neq 0$, $(2\alpha +\beta)(x-z)=0$.\\ 

Since $s$ preserves the Lie bracket $[x,X_{3\alpha +\beta}]$, $a_3\cdot (3\alpha+\beta) (x-z)=0$. If $a_3\neq 0$, then $(3\alpha+\beta) (x-z)=0$, i.e. $x=z$. Hence $L=L'$. If $a_3=0$, compose $s$ with a Weyl group element $w\in Aut(G_2)$, which acts on ${\mathfrak{h}}^{*}$ as a reflection through $2\alpha + \beta$. Hence we may assume that $a_3\neq 0$, and so $L=L'$ or $L=w(L')$. In the latter case, $L=L'$ by the condition imposed on $L$ in Table~\ref{table:1}:
$$
\beta \cdot (2\alpha +\beta)(L)\geq 0.
$$

Now consider subalgebras of dimension $7$.\\

There is one regular subalgebra of dimension $7$ with the $4$-dimensional radical:
$$
A_1\oplus \mathbb C H_{2\alpha +\beta} \oplus \mathbb C X_{2\alpha +\beta} \oplus \mathbb C X_{3\alpha +\beta}\oplus \mathbb C X_{3\alpha +2\beta}.
$$

There are two regular solvable subalgebras of dimension $7$ with $5$-dimensional subalgebras of nilpotent elements:
$$
\mathfrak{h}\oplus \mathbb C X_{\alpha} \oplus \mathbb C X_{\alpha +\beta} \oplus \mathbb C X_{2\alpha +\beta} \oplus \mathbb C X_{3\alpha +\beta} \oplus \mathbb C X_{3\alpha +2\beta} ,\quad \mathfrak{h}\oplus \mathbb C X_{\beta} \oplus \mathbb C X_{\alpha +\beta} \oplus \mathbb C X_{2\alpha +\beta} \oplus \mathbb C X_{3\alpha +\beta} \oplus \mathbb C X_{3\alpha +2\beta}.
$$

There is one type of regular solvable subalgebras of dimension $7$ with $6$-dimensional subalgebras of nilpotent elements:
$$
L\oplus \mathbb C X_{\alpha} \oplus \mathbb C X_{\beta} \oplus \mathbb C X_{\alpha +\beta} \oplus \mathbb C X_{2\alpha +\beta} \oplus \mathbb C X_{3\alpha +\beta} \oplus \mathbb C X_{3\alpha +2\beta},
$$
parametrized by a vector subspace $L\subset \mathfrak{h}$, $dim(L)=1$.\\

Let $\mathfrak{g}=\mathbb C X_{\alpha} \oplus \mathbb C X_{\beta} \oplus \mathbb C X_{\alpha +\beta} \oplus \mathbb C X_{2\alpha +\beta} \oplus \mathbb C X_{3\alpha +\beta} \oplus \mathbb C X_{3\alpha +2\beta}$. Suppose $L\oplus \mathfrak{g}$ is conjugate to $L'\oplus \mathfrak{g}$ by $s\in Aut(G_2)$. Then $[\mathfrak{g},[\mathfrak{g},[\mathfrak{g},\mathfrak{g}]]]=\mathbb C X_{3\alpha +\beta} \oplus \mathbb C X_{3\alpha +2\beta}$ and $[\mathfrak{g},[\mathfrak{g},[\mathfrak{g},[\mathfrak{g},\mathfrak{g}]]]]=\mathbb C X_{3\alpha +2\beta}$ map to themselves.\\

Hence $s(X_{3\alpha +2\beta})=\lambda X_{3\alpha +2\beta}$, $s(X_{3\alpha +\beta})=a_4X_{3\alpha +\beta}+a_5X_{3\alpha +2\beta}$, where $a_4\neq 0$.\\

Let $s(x)=z+\sum\limits_{\gamma\in {\Phi}^{+}} u_{\gamma} X_{\gamma}$, $L=\mathbb C x$, $L'=\mathbb C z$. Since $s$ preserves the Lie bracket $[x,X_{3\alpha +2\beta}]$, $(3\alpha +2\beta)(x-z)=0$.\\ 

Since $s$ preserves the Lie bracket $[x,X_{3\alpha +\beta}]$, $(3\alpha+\beta) (x-z)=0$. Hence $x=z$, i.e. $L=L'$.\\

Now consider subalgebras of dimension $8$.\\

There is one regular semisimple subalgebra of dimension $8$:
$$
A_2.
$$

There is one regular solvable subalgebra of dimension $8$:
$$
\mathfrak{b}.
$$

There are two regular subalgebras of dimension $8$, which are neither solvable nor semisimple:
$$
A_1\oplus \mathbb C X_{\alpha} \oplus \mathbb C X_{\alpha +\beta} \oplus \mathbb C X_{2\alpha +\beta} \oplus \mathbb C X_{3\alpha +\beta} \oplus \mathbb C X_{3\alpha +2\beta},
$$
$$
\tilde{A_1}\oplus \mathbb C X_{\beta} \oplus \mathbb C X_{\alpha +\beta} \oplus \mathbb C X_{2\alpha +\beta} \oplus \mathbb C X_{3\alpha +\beta} \oplus \mathbb C X_{3\alpha +2\beta}.
$$

They are not conjugate, because their radicals are not conjugate.\\

Finally, there are two regular subalgebras of dimension $9$:
$$
G_2[\alpha],\quad G_2[\beta].
$$

They are not conjugate, because their radicals are not conjugate.\\

\end{proof}

An element $x\in G_2$ is called \textit{regular} if $ad_x \colon G_2\to G_2$, $y\mapsto [x,y]$ has maximal possible number of distinct eigenvalues \cite{Dynkin}. The following criterion will be useful for us later.\\

\begin{theorem}\label{lemma:1}(\cite{Dynkin}, Theorem 6.1)
If $x\in G_2$ is a regular element, then any subalgebra $\gg\subset G_2$ containing $x$ is regular. 
\end{theorem}

\section{Semisimple subalgebras (\cite{Malcev}, \cite{Dynkin})}

Semisimple subalgebras of complex semisimple Lie algebras were classified by Malcev \cite{Malcev} and Dynkin \cite{Dynkin}. In this section we collect their results for $G_2$. The superscript in the notation for a simple subalgebra in Table~\ref{table:2} denotes its Dynkin index in $G_2$ (see \cite{Dynkin} for the definition). An embedding of $\mathfrak{sl}(2,\mathbb C)$ into $G_2$ is specified in terms of the generators $f, e_{+}, e_{-}\in G_2$ of $\mathfrak{sl}(2,\mathbb C)$ satisfying the relations
$$
[f,e_{+}]=2e_{+},\quad [f,e_{-}]=-2e_{-},\quad [e_{+},e_{-}]=f.
$$

\begin{center}
\begin{longtable}{|c|c|}
\caption{Semisimple non-regular subalgebras of $G_2$ up to conjugacy \cite{Malcev}, \cite{Dynkin}}\label{table:2}\\
\hline
Subalgebra & Embedding in $G_2$
\\ \hline\hline
	$A_1^4$ & $f=2H_{3\alpha +\beta}$, $e_{+}=\sqrt{2}\cdot (X_{-\beta}+X_{3\alpha +2\beta})$, $e_{-}=\sqrt{2}\cdot (X_{\beta}+X_{-(3\alpha +2\beta )})$
	\\ \hline
	$A_1^{28}$ & $f=14\cdot H_{9\alpha +5\beta}$, $e_{+}=\sqrt{6}\cdot X_{\alpha}+\sqrt{10}\cdot X_{\beta}$, $e_{-}=\sqrt{6}\cdot X_{-\alpha}+\sqrt{10}\cdot X_{-\beta}$
  \\ \hline
\end{longtable}
\end{center}

Together with Table~\ref{table:1}, this gives the complete list of semisimple subalgebras of $G_2$ up to conjugacy \cite{Malcev}, \cite{Dynkin}:
$$
A_2,\quad A_1+\tilde{A_1},\quad A_1=A_1^1,\quad \tilde{A_1}=A_1^3,\quad A_1^4,\quad A_1^{28}.
$$

We will need the following decompositions of $G_2$ as the adjoint representation of its semisimple subalgebras. Let $D_s$, $s\geq 0$, $s\in \frac{1}{2}{\mathbb Z}$, denote the irreducible representation of $\mathfrak{sl}(2,\mathbb C)$ of dimension $2s+1$, $[1,0]$ the defining representation of $\mathfrak{sl}(3,\mathbb C)$ on ${\mathbb C}^3$, $[0,1]$ its dual, ${\mathbb C}^k=D_0^{\oplus k}$ the trivial (zero) representation of $\mathfrak{sl}(2,\mathbb C)$ of dimension $k$.\\

\begin{proposition}\label{proposition:2}
Let $\mathfrak{g}\subset G_2$ be a semisimple subalgebra. The restriction of the adjoint representation of $G_2$ to $\mathfrak{g}$ decomposes as follows.
\begin{itemize}
\item If $\mathfrak{g}=A_2=G_2(\alpha)$, then
$$
G_2=G_2(\alpha)\oplus W_1 \oplus W_2,
$$ 
where 
$$
W_1=\mathbb C X_{\alpha}\oplus\mathbb C X_{\alpha +\beta}\oplus\mathbb C X_{-(2\alpha +\beta)} \cong [1,0],\quad W_2=\mathbb C X_{-\alpha}\oplus\mathbb C X_{-(\alpha +\beta )}\oplus\mathbb C X_{2\alpha +\beta} \cong [0,1]
$$
as representations of $A_2$. 
\item If $\mathfrak{g}=A_1+\tilde{A_1}=G_2(\beta)$, then
$$
G_2=G_2(\beta)\oplus W_{1\tilde{1}},
$$ 
where 
\begin{align*}
& W_{1\tilde{1}}=\mathbb C X_{\beta}\oplus \mathbb C X_{-\beta}\oplus\mathbb C X_{\alpha +\beta}\oplus\mathbb C X_{-(\alpha +\beta )}\oplus\mathbb C X_{2\alpha +\beta} \oplus\mathbb C X_{-(2\alpha +\beta)}\\
& \oplus\mathbb C X_{3\alpha +\beta}\oplus\mathbb C X_{-(3\alpha +\beta)} \cong D_{1/2}\otimes_{\mathbb C} D_{3/2}
\end{align*}
as representations of $A_1+\tilde{A_1}$.
\item If $\mathfrak{g}=A_1^{28}$, then
$$
G_2=A_1^{28}\oplus V_1^{28},
$$ 
where 
\begin{align*}
& V_1^{28}=\mathbb C (H_{\alpha}-3H_{\beta})\oplus \mathbb C X_{\alpha +\beta}\oplus\mathbb C X_{-(\alpha +\beta )}\oplus\mathbb C X_{2\alpha +\beta}\oplus\mathbb C X_{-(2\alpha +\beta)}\\
& \oplus\mathbb C X_{3\alpha +\beta}\oplus\mathbb C X_{-(3\alpha +\beta)}\mathbb C X_{3\alpha +2\beta}\oplus \mathbb C X_{-(3\alpha +2\beta)}\\
& \oplus \mathbb C (\sqrt{15}\cdot X_{\alpha}-9\cdot X_{\beta})\oplus \mathbb C (\sqrt{15}\cdot X_{-\alpha}-9\cdot X_{-\beta}) \cong D_{5}
\end{align*}
as representations of $A_1$.
\item If $\mathfrak{g}=A_1^{4}$, then
$$
G_2=A_1^{4}\oplus V_1^{4}\oplus V_2^{4}\oplus V_3^{4},
$$ 
where 
\begin{align*}
& V_1^{4}=\mathbb C X_{\alpha} \oplus \mathbb C X_{\alpha +\beta}\oplus\mathbb C X_{-(2\alpha +\beta )} \cong D_{1},\\
& V_2^{4}=\mathbb C X_{-\alpha} \oplus \mathbb C X_{-(\alpha +\beta )}\oplus\mathbb C X_{2\alpha +\beta} \cong D_{1},\\
& V_3^{4}= \mathbb C X_{3\alpha +\beta}\oplus\mathbb C X_{-(3\alpha +\beta)}\oplus\mathbb C H_{\alpha +\beta}\oplus \mathbb C (X_{\beta}- X_{-(3\alpha +2\beta)})\oplus \mathbb C (X_{-\beta}- X_{3\alpha +2\beta}) \cong D_{2}
\end{align*}
as representations of $A_1$.
\item If $\mathfrak{g}=A_1^{3}=\tilde{A_1}$, then  
$$
G_2=\tilde{A_1}\oplus \tilde{V_1}\oplus \tilde{V_2}\oplus \tilde{V_3},
$$ 
where 
\begin{align*}
& \tilde{V_1}=\mathbb C X_{\beta} \oplus \mathbb C X_{\alpha +\beta}\oplus\mathbb C X_{2\alpha +\beta}\oplus\mathbb C X_{3\alpha +\beta} \cong D_{3/2},\\
& \tilde{V_2}=\mathbb C X_{-\beta} \oplus \mathbb C X_{-(\alpha +\beta )}\oplus\mathbb C X_{-(2\alpha +\beta )}\oplus\mathbb C X_{-(3\alpha +\beta )} \cong D_{3/2},\\
& \tilde{V_3}= \mathbb C H_{3\alpha +2\beta}\oplus \mathbb C X_{3\alpha +2\beta}\oplus\mathbb C X_{-(3\alpha +2\beta)} \cong {\mathbb C}^3
\end{align*}
as representations of $\tilde{A_1}$.
\item If $\mathfrak{g}=A_1^{1}=A_1$, then
$$
G_2=A_1\oplus V_1\oplus V_2\oplus V_3\oplus V_4\oplus V_5,
$$ 
where 
\begin{align*}
& V_1=\mathbb C X_{\alpha} \oplus \mathbb C X_{\alpha +\beta} \cong D_{1/2},\\
& V_2=\mathbb C X_{-\alpha} \oplus \mathbb C X_{-(\alpha +\beta )} \cong D_{1/2},\\
& V_3=\mathbb C X_{3\alpha +\beta} \oplus \mathbb C X_{3\alpha +2\beta} \cong D_{1/2},\\
& V_4=\mathbb C X_{-(3\alpha +\beta )} \oplus \mathbb C X_{-(3\alpha +2\beta )} \cong D_{1/2},\\
& V_5=\mathbb C H_{2\alpha +\beta} \oplus \mathbb C X_{2\alpha +\beta} \oplus \mathbb C X_{-(2\alpha +\beta )} \cong {\mathbb C}^3
\end{align*}
as representations of $A_1$.
\end{itemize}
\end{proposition}

\begin{proof}
This can be checked by a matrix computation. 
\end{proof}

\section{Non-solvable non-semisimple subalgebras}

\begin{theorem}\label{Theorem2}
If $\mathfrak{g}\subset G_2$ is a non-semisimple non-solvable subalgebra, then $\mathfrak{g}\subset G_2$ is regular.
\end{theorem}

\begin{proof}
This follows from Lemma~\ref{lemma:2}, Lemma~\ref{lemma:3} and Lemma~\ref{lemma:4}.
\end{proof}

\begin{lemma}\label{lemma:2}
If $\mathfrak{g}\subset G_2$ is a non-semisimple non-solvable subalgebra, then a Levi subalgebra of $\mathfrak{g}$ is conjugate in $G_2$ to either $A_1^{1}=A_1$ or $A_1^{3}=\tilde{A_1}$.
\end{lemma}

\begin{proof}

By \cite{Dynkin}, Theorem~$7.3$ and Theorem~$5.5$ (see also \cite{Morozov} and \cite{Karpelevich}), any non-semisimple subalgebra $\mathfrak{g}\subset G_2$ is conjugate to a subalgebra of $G_2(\alpha)=A_2$, $G_2(\beta)=A_1+\tilde{A_1}$, $G_2[\alpha]$ or $G_2[\beta]$.\\

Let $\mathfrak{g}\subset G_2[\alpha]$ or $\mathfrak{g}\subset G_2[\beta]$ be non-semisimple and non-solvable, $\mathfrak{s}\subset \mathfrak{g}$ its Levi subalgebra. Then
$$
3\leq dim(\mathfrak{s})\leq dim(Levi(G_2[\alpha]))=dim(Levi(G_2[\beta]))=3.
$$

Hence by the conjugacy of Levi subalgebras, we may assume that either
$$
\mathfrak{s}=A_1=Levi(G_2[\alpha])\subset \mathfrak{g}\subset G_2[\alpha]
$$
or
$$
\mathfrak{s}=\tilde{A_1}=Levi(G_2[\beta])\subset \mathfrak{g}\subset G_2[\beta].
$$

If $\mathfrak{g}\subset G_2(\alpha)$ or $\mathfrak{g}\subset G_2(\beta)$, the claim follows from Proposition~\ref{proposition:2} (or from \cite{RepkaA2}, \cite{RepkaD2}).

\end{proof}

\begin{lemma}\label{lemma:3}
Let $\mathfrak{g}\subset G_2$ be a non-semisimple non-solvable subalgebra. If a Levi subalgebra of $\mathfrak{g}$ is conjugate in $G_2$ to $\tilde{A_1}=A_1^{3}$, then $\mathfrak{g}\subset G_2$ is regular.
\end{lemma}

\begin{proof}

Suppose $\mathfrak{g}\subset G_2$ is not regular. By the remark in the proof of Lemma~\ref{lemma:2}, we may assume that either $\mathfrak{g}\subset G_2[\beta ]$ or $\mathfrak{g}\subset G_2(\beta )=A_1+\tilde{A_1}$. Note that $\tilde{A_1}\subset G_2$ is not conjugate to a subalgebra of $A_2=G_2(\alpha )$ (see \cite{Malcev} or \cite{RepkaA2}).\\

Let $\tilde{A_1} \subset \mathfrak{g}\subset G_2[\beta ]$. By Proposition~\ref{proposition:2},
$$
G_2[\beta ]=\tilde{A_1}\oplus \tilde{V_1}\oplus \left( \mathbb C H_{3\alpha +2\beta} \oplus \mathbb C X_{3\alpha +2\beta} \right),
$$
where $\tilde{V_1}\cong D_{3/2}$ and $\left( \mathbb C H_{3\alpha +2\beta} \oplus \mathbb C X_{3\alpha +2\beta} \right)\cong {\mathbb C}^2$ as representations of $\tilde{A_1}$.\\

If $rad(\mathfrak{g})$ contains a subrepresentation of $\tilde{A_1}$ isomorphic to $D_{3/2}$, then $\tilde{A_1}\oplus \tilde{V_1}\subset \mathfrak{g}$. Since $\mathfrak{g}\subset G_2$ is non-regular, $\mathfrak{g}=\tilde{A_1}\oplus \tilde{V_1}\oplus \mathbb C (a\cdot H_{3\alpha +2\beta}+b\cdot X_{3\alpha +2\beta})$. Applying $\exp(c\cdot ad_{X_{3\alpha +2\beta }})\in Aut(G_2)$, we see that $\mathfrak{g}\subset G_2$ is regular, a contradiction.\\

Hence $rad(\mathfrak{g})\cong \mathbb C$ or $rad(\mathfrak{g})\cong {\mathbb C}^2$ as representations of $\tilde{A_1}$. Then $\mathfrak{g}=\tilde{A_1}\oplus V$, where $V\subset \mathbb C H_{3\alpha +2\beta} \oplus \mathbb C X_{3\alpha +2\beta}$ is a vector subspace of dimension $1$. Then $\mathfrak{g}\subset G_2$ is regular by the same argument as above, a contradiction.\\

Let $\tilde{A_1} \subset \mathfrak{g}\subset G_2(\beta )=A_1+\tilde{A_1}$. By Proposition~\ref{proposition:2},
$$
G_2(\beta )=\tilde{A_1}\oplus \tilde{V_3},
$$
where $\tilde{V_3}=\mathbb C H_{3\alpha +2\beta }\oplus \mathbb C X_{3\alpha +2\beta }\oplus \mathbb C X_{-(3\alpha +2\beta )}  \cong {\mathbb C}^3$ as representations of $\tilde{A_1}$. Hence $\mathfrak{g}=\tilde{A_1}\oplus V$, where $V\subset \tilde{V_3}$ is a vector subspace of dimension $1$ or $2$.\\

Let $x=a\cdot H_{3\alpha +2\beta} +b\cdot X_{3\alpha +2\beta}+c\cdot X_{-(3\alpha +2\beta )}\in V$, $x\neq 0$. Conjugating by $\exp(d\cdot ad_{X_{3\alpha +2\beta }})$, we can ensure that either $b=0$ or $a=c=0$. If $b=0$, then $x=a\cdot H_{3\alpha +2\beta} +c\cdot X_{-(3\alpha +2\beta )}\in V$, $x\neq 0$. Conjugating by $\exp(d\cdot ad_{X_{-(3\alpha +2\beta )}})$, we can ensure that $a\cdot c=0$. If $c=0$, then $\mathfrak{h}\subset \mathfrak{g}$, and so $\mathfrak{g}\subset G_2$ would be regular. Hence $a=0$, i.e. $X_{-(3\alpha +2\beta )}\in V$.\\

So, we may assume that either $X_{3\alpha +2\beta}\in V$ or $X_{-(3\alpha +2\beta )}\in V$. Then $dim(V)=2$ and we can find $z=a'\cdot H_{3\alpha +2\beta} +b'\cdot X_{\mp (3\alpha +2\beta )}\in V$, $z\neq 0$. Since $[X_{\pm (3\alpha +2\beta )},z]\in V$, $H_{3\alpha +2\beta}\in V$. Hence $\mathfrak{g}\subset G_2$ is regular, a contradiction.\\

The second half of the proof also follows immediately from \cite{RepkaA2}.

\end{proof}

\begin{lemma}\label{lemma:4}
Let $\mathfrak{g}\subset G_2$ be a non-semisimple non-solvable subalgebra. If a Levi subalgebra of $\mathfrak{g}$ is conjugate in $G_2$ to $A_1=A_1^{1}$, then $\mathfrak{g}\subset G_2$ is regular.
\end{lemma}

\begin{proof}

Suppose $\mathfrak{g}\subset G_2$ is not regular. By the remark in the proof of Lemma~\ref{lemma:2}, we may assume that either $\mathfrak{g}\subset G_2[\alpha ]$ or $\mathfrak{g}\subset G_2(\beta )=A_1+\tilde{A_1}$ or $\mathfrak{g}\subset G_2(\alpha )=A_2$.\\

In the latter two cases, the claim follows from \cite{RepkaD2} and \cite{RepkaA2}. However, we include these cases for completeness.\\

Let $A_1 \subset \mathfrak{g}\subset G_2[\alpha ]$. By Proposition~\ref{proposition:2},
$$
G_2[\alpha ]=A_1 \oplus V_1\oplus V_3\oplus \left( \mathbb C H_{2\alpha +\beta} \oplus \mathbb C X_{2\alpha +\beta} \right),
$$
where $V_1\cong V_3\cong D_{1/2}$ and $\left( \mathbb C H_{2\alpha +\beta} \oplus \mathbb C X_{2\alpha +\beta} \right)\cong {\mathbb C}^2$ as representations of $A_1$.\\

Then $rad(\mathfrak{g})\cong D_{1/2}^{\oplus p}\oplus {\mathbb C}^q$, where $p,q\in \{\ 0,1,2 \}$.\\

If $p=2$, then $V_1\oplus V_3\subset rad(\mathfrak{g})$, i.e. $\mathfrak{g}=A_1\oplus V_1\oplus V_3\oplus \mathbb C (a\cdot H_{2\alpha +\beta }+b\cdot X_{2\alpha +\beta })$. Conjugating by $\exp(c\cdot ad_{X_{2\alpha +\beta }})$, we can ensure that $a\cdot b=0$. Hence $\mathfrak{g}\subset G_2$ is regular, a contradiction.\\

If $p=0$, then $\mathfrak{g}=A_1\oplus V$, where $V\subset \mathbb C H_{2\alpha +\beta }\oplus \mathbb C X_{2\alpha +\beta }$. Since $\mathfrak{g}\subset G_2$ is not regular, $dim(V)=1$, i.e. $V=\mathbb C \left( a\cdot H_{2\alpha +\beta }+ b\cdot X_{2\alpha +\beta } \right)$, where $a\cdot b\neq 0$. Applying $\exp(c\cdot ad_{X_{2\alpha +\beta }})$ as above, we conclude that $\mathfrak{g}\subset G_2$ is regular, a contradiction.\\

Hence $p=1$, i.e. $rad(\mathfrak{g})\cong D_{1/2}\oplus {\mathbb C}^q$, $q\in \{ 0,1,2 \}$. Hence we may assume that either $\mathfrak{g}=A_1\oplus V$ or $\mathfrak{g}=A_1\oplus V\oplus \mathbb C X_{2\alpha +\beta }$, where $V\subset V_1\oplus V_3$ is a subrepresentation of $A_1$ isomorphic to $D_{1/2}$.\\

Let $x=a\cdot X_{\alpha}+b\cdot X_{3\alpha +\beta}+a'\cdot X_{\alpha+\beta}+b'\cdot X_{3\alpha +2\beta}\in V$, $x\neq 0$. Taking the bracket $[x,X_{-\beta}]$, we can find a non-zero element $x'=a''\cdot X_{\alpha } + b''\cdot X_{3\alpha +\beta}\in V$. Applying $\exp(c\cdot ad_{X_{2\alpha +\beta}})\in Aut(G_2)$, we can ensure that either $X_{\alpha} \in V$ or $X_{3\alpha +\beta} \in V$. Acting by $A_1\subset \mathfrak{g}$, we conclude that either $V=V_1$ or $V=V_3$ in this case. Hence $\mathfrak{g}\subset G_2$ is regular, a contradiction.\\

Let $A_1 \subset \mathfrak{g}\subset G_2(\alpha )=A_2$. By Proposition~\ref{proposition:2},
$$
G_2(\alpha )=A_1 \oplus V_3\oplus V_4\oplus \mathbb C H_{2\alpha +\beta},
$$
where $V_3\cong V_4\cong D_{1/2}$ and $\mathbb C H_{2\alpha +\beta}\cong {\mathbb C}$ as representations of $A_1$.\\

Note that $H_{2\alpha +\beta}\notin \mathfrak{g}$. Otherwise, $\mathfrak{h}\subset \mathfrak{g}$, and so $\mathfrak{g}\subset G_2$ would be regular.\\

Hence $rad(\mathfrak{g})\cong D_{1/2}^{\oplus p}$, where $p\in \{\ 0,1,2 \}$. Since $\mathfrak{g}\subset G_2$ is not regular, $p=1$, i.e. $\mathfrak{g}=A_1\oplus V$, where $V\subset V_3\oplus V_4$ is a subrepresentation of $A_1$ isomorphic to $D_{1/2}$.\\

Let $x=a\cdot X_{3\alpha +\beta}+b\cdot X_{-(3\alpha +2\beta )}+a'\cdot X_{-(3\alpha+\beta )}+b'\cdot X_{3\alpha +2\beta}\in V$, $x\neq 0$. Acting by $ad_{X_{-\beta}}\in End(V)$, we may assume that $a'=b'=0$, i.e. $x=a\cdot X_{3\alpha +\beta} + b\cdot X_{-(3\alpha +2\beta )}\in V$, $x\neq 0$.\\

Then $[X_{\beta},x]=a\cdot X_{3\alpha +2\beta}-b\cdot X_{-(3\alpha +\beta)}\in\mathfrak{g}$ and hence $[[X_{\beta}, x],x]=ab\cdot H_{3\alpha +2\beta}+ab\cdot H_{3\alpha +\beta}$ should be proportional to $H_{\beta }$. Hence $a\cdot b=0$. Then $V=V_3$ or $V=V_4$, i.e. $\mathfrak{g}\subset G_2$ is regular, a contradiction.\\

Let $A_1 \subset \mathfrak{g}\subset G_2(\beta )=A_1+\tilde{A_1}$. By Proposition~\ref{proposition:2},
$$
G_2(\beta )=A_1 \oplus V_5,
$$
where $V_5=\mathbb C H_{2\alpha +\beta }\oplus \mathbb C X_{2\alpha +\beta }\oplus \mathbb C X_{-(2\alpha +\beta )}\cong {\mathbb C}^3$ as representations of $A_1$.\\

Hence $\mathfrak{g}=A_1\oplus V$, where $V\subset V_5$ is a vector subspace of dimension $1$ or $2$.\\

Let $x=a\cdot H_{2\alpha +\beta }+b\cdot X_{2\alpha +\beta }+c\cdot  X_{-(2\alpha +\beta )}\in V$, $x\neq 0$. Conjugating by $\exp(d\cdot ad_{X_{2\alpha +\beta}})\in Aut(G_2)$, we can ensure that either $b=0$ or $a=c=0$. If $b=0$, then $x=a\cdot H_{2\alpha +\beta }+c\cdot  X_{-(2\alpha +\beta )}\in V$, $x\neq 0$.\\

Conjugation by $\exp(d\cdot ad_{X_{-(2\alpha +\beta )}})\in Aut(G_2)$ would imply that $H_{2\alpha +\beta} \in \mathfrak{g}$, i.e. $\mathfrak{g}\subset G_2$ is regular, unless $a=0$. Hence $X_{-(2\alpha+\beta)}\in V$.\\

So, we may assume that either $X_{2\alpha +\beta}\in V$ or $X_{-(2\alpha +\beta )}\in V$. Since $\mathfrak{g}\subset G_2$ is not regular, $dim(V)=2$ and we can find $z=a'\cdot H_{2\alpha +\beta} +b'\cdot X_{\mp (2\alpha +\beta )}\in V$, $z\neq 0$. Since $[X_{\pm (2\alpha +\beta )},z]\in V$, $H_{2\alpha +\beta}\in V$. Hence $\mathfrak{g}\subset G_2$ is regular, a contradiction.

\end{proof}

\section{Solvable subalgebras}

Let $\mathfrak{b}=\mathfrak{h}\oplus \mathfrak{n}\subset G_2$ be the fixed Borel subalgebra, $\mathfrak{n}=\bigoplus\limits_{\gamma \in {\Phi}^{+}}\mathbb C X_{\gamma}\subset G_2$.\\

We start with a few Lemmata.\\

\begin{lemma}\label{lemma:5}
Let $\mathfrak{g}\subset G_2$ be a solvable non-regular subalgebra of $G_2$. Then it is conjugate to a subalgebra $\mathfrak{g}\subset \mathfrak{b}$ such that
\begin{itemize}
\item either $\mathfrak{g}\subset \mathfrak{n}$
\item or $\mathfrak{g}/{\mathfrak{g} \cap \mathfrak{n}} \subset \mathfrak{b}/\mathfrak{n}$ is generated by an element of the form $f=x+\lambda \cdot X_{\gamma}\in \mathfrak{g}$, where $x\in \mathfrak{h}$, $\gamma \in {\Phi}^{+}$, $\gamma (x)=0$, $\lambda \in {\mathbb C}^{*}$,
\item or $\mathfrak{g}/{\mathfrak{g} \cap \mathfrak{n}} \subset \mathfrak{b}/\mathfrak{n}$ is generated by an element $f\in \mathfrak{h}\cap \mathfrak{g}$.
\end{itemize}
\end{lemma}

\begin{proof}

If $\mathfrak{g}/{\mathfrak{g} \cap \mathfrak{n}} = \mathfrak{b}/\mathfrak{n}$, then $\mathfrak{g}$ contains an element $x+n$, where $x\in \mathfrak{h}$, $n\subset \mathfrak{n}$ and $\gamma (x)$, $\gamma \in \Phi$, are all distinct and non-zero. By Theorem~\ref{lemma:1}, $\mathfrak{g}\subset G_2$ is regular, a contradiction.\\

Suppose that $\mathfrak{g}/{\mathfrak{g} \cap \mathfrak{n}} \subset \mathfrak{b}/\mathfrak{n}$ is non-zero and generated by $f=x+n\in \mathfrak{g}$, where $x\in \mathfrak{h}$, $n\in \mathfrak{n}$. Let $f=(x+n_s)+n_n$ be the Jordan decomposition: $x+n_s\in G_2$ semisimple, $n_n\in G_2$ nilpotent and $[x+n_s,n_n]=0$. Note that $n_n, n_s\in \mathfrak{n}$, because $ad_{n_n}$ is a polynomial in $ad_{f}$, and hence $n_n$ normalizes $\mathfrak{b}$.\\

Note that $dim(ker(ad_{x+n_s}))$ is $2$ or $4$. If it is $2$, then the characteristic polynomial of $ad_{x+n_s}$ on $\mathfrak{b}$ has the lowest possible multiplicity at $0$. Hence $x+n_s$ lies in a Cartan subalgebra of $\mathfrak{b}$.\\

If $dim(ker(ad_{x+n_s}))=4$, then $dim(ker(ad_{x+n_s} {\mid}_{\mathfrak{b}}))\geq 3$. Hence there is a \textit{regular} element (see the definition before Theorem~\ref{lemma:1}) $f'\in \mathfrak{b}$ such that $[x+n_s , f']=0$. Then the abelian subalgebra of $\mathfrak{b}$ spanned by $x+n_s$ and $f'$ consists of semisimple elements. Hence it is a Cartan subalgebra of $G_2$.\\

By the conjugacy of Cartan subalgebras, we may assume that $x+n_s \in \mathfrak{h}$, i.e. $n_s=0$, $n=n_n$. Since $[x+n_s,n_n]=0$, $n=\lambda \cdot X_{\gamma}$, where $\lambda \in \mathbb C$, $\gamma \in {\Phi}^{+}$, $\gamma (x)=0$, if such $\gamma$ exists. If there is no $\gamma \in {\Phi}^{+}$ such that $\gamma (x)=0$, then $n=0$, i.e. $f=x\in \mathfrak{h}$.

\end{proof}

In Lemma~\ref{lemma:6} and Table~\ref{table:10} below we do not identify conjugate subalgebras. In the notation $X_{\gamma} + \cdots$, the dots denote a linear combination of monomials $X_{\gamma '}$ with ${\gamma}' \succ {\gamma}$ (also, if $\gamma = \alpha$, then ${\gamma}'\neq \beta$). Their coefficients in Table~\ref{table:10} are arbitrary such that the Lie algebra axioms are fulfilled.\\

\begin{lemma}\label{lemma:6}
The complete list of subalgebras of $\mathfrak{n}$ is given in Table~\ref{table:10}.
\end{lemma}

\begin{proof}

Let $\mathfrak{g}\subset \mathfrak{n}$ be a subalgebra.\\

If $\mathfrak{g}/\mathfrak{g}\cap [\mathfrak{n},\mathfrak{n}]= \mathfrak{n}/[\mathfrak{n},\mathfrak{n}]$, then $\mathfrak{g}$ contains $x=X_{\beta}+\cdots$ and $y=X_{\alpha}+\cdots$. Hence also $[x,y]=X_{\alpha+\beta}+\cdots$, $z=[[x,y],y]=2X_{2\alpha+\beta}+\cdots$, $[y,z]=6X_{3\alpha+\beta}+\cdots$ and $[x,[y,z]]=6X_{3\alpha+2\beta}$. In other words, $\mathfrak{g}=\mathfrak{n}$ in this case.\\

Alternatively, either $\mathfrak{g}\subset [\mathfrak{n},\mathfrak{n}]$ or $\mathfrak{g}/\mathfrak{g}\cap [\mathfrak{n},\mathfrak{n}]\subset \mathfrak{n}/[\mathfrak{n},\mathfrak{n}]$ is generated by $a_1\cdot X_{\alpha } + b_1\cdot X_{\beta}+\cdots \in \mathfrak{g}$ for some $a_1,b_1\in \mathbb C$.\\

Let $\mathfrak{g_1}\subset [\mathfrak{n},\mathfrak{n}]$ be a subalgebra. Then either $\mathfrak{g_1}$ contains an element $X_{\alpha+\beta}+\cdots$ or $\mathfrak{g_1}\subset [\mathfrak{n},[\mathfrak{n},\mathfrak{n}]]$.\\

Let $\mathfrak{g_2}\subset [\mathfrak{n},[\mathfrak{n},\mathfrak{n}]]$ be a subalgebra. Then either $\mathfrak{g_2}$ contains an element $X_{2\alpha+\beta}+\cdots$ or $\mathfrak{g_2}\subset [\mathfrak{n},[\mathfrak{n},[\mathfrak{n},\mathfrak{n}]]]$.\\

Let $\mathfrak{g_3}\subset [\mathfrak{n},[\mathfrak{n},[\mathfrak{n},\mathfrak{n}]]]$ be a subalgebra. Then either $\mathfrak{g_3}=0$ or $\mathfrak{g_3}$ is spanned by an element $a_4\cdot X_{3\alpha +\beta} + b_4\cdot X_{3\alpha +2\beta}$, $a_4,b_4\in \mathbb C$, or $\mathfrak{g_3}= [\mathfrak{n},[\mathfrak{n},[\mathfrak{n},\mathfrak{n}]]]$.\\

Combining all cases together, we obtain Table~\ref{table:10}. 

\end{proof}

\begin{center}
\begin{longtable}{|c|c|}
\caption{Subalgebras of $\mathfrak{n}$}\label{table:10}\\
\hline
No. & Subalgebra $\mathfrak{g}\subset \mathfrak{n}\subset G_2$
\\ \hline\hline
	$1$ & $0$
  \\ \hline
	$2$ & $[\mathfrak{n},[\mathfrak{n},[\mathfrak{n},\mathfrak{n}]]]$
	\\ \hline
	$3$ & $[\mathfrak{n},[\mathfrak{n},\mathfrak{n}]]$
  \\ \hline
	$4$ & $[\mathfrak{n},\mathfrak{n}]$
  \\ \hline	
	$5$ & $\mathfrak{n}$
  \\ \hline
	$6$ & $\mathbb C (a_4\cdot X_{3\alpha +\beta} + b_4\cdot X_{3\alpha +2\beta})$
  \\ \hline
	$7$ & $\mathbb C (X_{2\alpha +\beta} + \cdots)$
	\\ \hline
	$8$ & $\mathbb C (X_{\alpha +\beta} + \cdots)$
	\\ \hline
	$9$ & $\mathbb C (a_1\cdot X_{\alpha } + b_1\cdot X_{\beta}+\cdots )$
  \\ \hline
	$10$ & $\mathbb C (X_{2\alpha +\beta} + \cdots) \oplus \mathbb C (a_4\cdot X_{3\alpha +\beta} + b_4\cdot X_{3\alpha +2\beta})$
	\\ \hline
	$11$ & $\mathbb C (X_{\alpha +\beta} + \cdots)\oplus [\mathfrak{n},[\mathfrak{n},[\mathfrak{n},\mathfrak{n}]]]$
	\\ \hline		
	$12$ & $\mathbb C (X_{\alpha +\beta} + \cdots)\oplus \mathbb C (a_4\cdot X_{3\alpha +\beta} + b_4\cdot X_{3\alpha +2\beta})$
	\\ \hline
	$13$ & $\mathbb C (X_{\alpha } + \cdots)\oplus \mathbb C (a_4\cdot X_{3\alpha +\beta} + b_4\cdot X_{3\alpha +2\beta})$
  \\ \hline	
	$14$ & $\mathbb C (a_1\cdot X_{\alpha } + b_1\cdot X_{\beta}+\cdots )\oplus \mathbb C X_{3\alpha +2\beta}$
  \\ \hline
	$15$ & $\mathbb C (a_1\cdot X_{\alpha } + b_1\cdot X_{\beta}+\cdots )\oplus [\mathfrak{n},[\mathfrak{n},[\mathfrak{n},\mathfrak{n}]]]$
  \\ \hline
	$16$ & $\mathbb C (X_{\beta} + \cdots)\oplus \mathbb C (X_{2\alpha +\beta} + \cdots)$
	\\ \hline
	$17$ & $\mathbb C (a_1\cdot X_{\alpha } + b_1\cdot X_{\beta}+\cdots )\oplus [\mathfrak{n},[\mathfrak{n},\mathfrak{n}]]$
  \\ \hline	
	$18$ & $\mathbb C (X_{\beta} + \cdots)\oplus \mathbb C (X_{\alpha +\beta} + \cdots)$
	\\ \hline
	$19$ & $\mathbb C (a_1\cdot X_{\alpha } + b_1\cdot X_{\beta}+\cdots )\oplus [\mathfrak{n},\mathfrak{n}]$
  \\ \hline
	$20$ & $\mathbb C (X_{\alpha +\beta} + \cdots)\oplus \mathbb C (X_{2\alpha +\beta} + \cdots)\oplus \mathbb C X_{3\alpha +2\beta}$
	\\ \hline			
	$21$ & $\mathbb C (X_{\alpha} + \cdots)\oplus \mathbb C (X_{2\alpha +\beta} + \cdots)\oplus \mathbb C (X_{3\alpha +\beta} + b_4\cdot X_{3\alpha +2\beta})$
	\\ \hline
	$22$ & $\mathbb C (X_{\beta} + \cdots)\oplus \mathbb C (X_{2\alpha +\beta} + \cdots)\oplus \mathbb C X_{3\alpha +2\beta}$
  \\ \hline
	$23$ & $\mathbb C (X_{\beta} + \cdots)\oplus \mathbb C (X_{\alpha +\beta} + \cdots) \oplus [\mathfrak{n},[\mathfrak{n},[\mathfrak{n},\mathfrak{n}]]]$
	\\ \hline
	$24$ & $\mathbb C (X_{\beta} + \cdots)\oplus \mathbb C (X_{\alpha +\beta} + \cdots) \oplus \mathbb C X_{3\alpha +2\beta}$
	\\ \hline
	$25$ & $\mathbb C (X_{\beta} + \cdots)\oplus \mathbb C (X_{\alpha +\beta} + \cdots)\oplus \mathbb C (X_{2\alpha +\beta} + \cdots) \oplus \mathbb C X_{3\alpha +2\beta}$
	\\ \hline
\end{longtable}
\end{center}

\begin{corollary}\label{corollary:1}
Every non-regular subalgebra of $G_2$ consisting of nilpotent elements is conjugate to one and only one of the following subalgebras of $\mathfrak{n}$:
$$
\mathbb C (X_{\alpha}+X_{3\alpha +2\beta}), \quad \mathbb C (X_{\alpha +\beta}+X_{3\alpha +\beta})\oplus \mathbb C X_{3\alpha +2\beta}, \quad \mathbb C (X_{\alpha}+X_{\beta})\oplus \mathbb C X_{3\alpha +\beta}\oplus \mathbb C X_{3\alpha +2\beta},
$$
$$
\mathbb C (X_{\alpha}+X_{\beta}), \quad \mathbb C (X_{\beta}+X_{3\alpha +\beta})\oplus \mathbb C X_{2\alpha +\beta}, \quad \mathbb C (X_{\alpha}+X_{\beta})\oplus \mathbb C X_{3\alpha +2\beta},
$$
$$
\mathbb C (X_{\alpha +\beta}+X_{3\alpha +\beta})\oplus \mathbb C X_{2\alpha +\beta}\oplus \mathbb C X_{3\alpha +2\beta}, \quad \mathbb C (X_{\beta}+X_{3\alpha +\beta})\oplus \mathbb C X_{2\alpha +\beta}\oplus \mathbb C X_{3\alpha +2\beta},
$$
$$
\mathbb C (X_{\alpha}+X_{\beta})\oplus \mathbb C X_{2\alpha +\beta}\oplus \mathbb C X_{3\alpha +\beta}\oplus \mathbb C X_{3\alpha +2\beta}, \quad \mathbb C (X_{\beta}+X_{2\alpha + \beta})\oplus \mathbb C X_{\alpha +\beta}\oplus \mathbb C X_{3\alpha +\beta}\oplus \mathbb C X_{3\alpha +2\beta},
$$
$$
\mathbb C (X_{\alpha}+X_{\beta})\oplus \mathbb C X_{\alpha +\beta}\oplus \mathbb C X_{2\alpha +\beta}\oplus \mathbb C X_{3\alpha +\beta}\oplus \mathbb C X_{3\alpha +2\beta},
$$
$$
\mathbb C (X_{\beta}+X_{\alpha +\beta}+\lambda \cdot X_{3\alpha +\beta})\oplus \mathbb C X_{2\alpha +\beta}\oplus \mathbb C X_{3\alpha +2\beta},\;\; \lambda \in {\mathbb C}^{*}.
$$
\end{corollary}

\begin{proof}
First, recall that the nilpotent orbits in a complex semisimple Lie algebra are parametrized by the conjugacy classes of $\mathfrak{sl}(2,\mathbb C)$ subalgebras (Theorems of Jacobson-Morozov and Kostant, see \cite{Orbits} for a survey). The latter were classified by Malcev \cite{Malcev} and Dynkin \cite{Dynkin}. In particular, $G_2$ has $4$ non-zero nilpotent orbits, Table~\ref{table:20}.\\

\begin{center}
\begin{longtable}{|c|c|c|}
\caption{Nilpotent orbits in $G_2$, \cite{Dynkin}}\label{table:20}\\
\hline
$\mathfrak{sl}(2,\mathbb C)$ subalgebra & Orbit representative & Dimension of the orbit
\\ \hline\hline
	$A_1$ & $X_{\beta}$ & 6
  \\ \hline
	$\tilde{A_1}$ & $X_{\alpha}$ & 8
	\\ \hline
	$A_1^4$ & $X_{\alpha} +X_{3\alpha +2\beta}$ & 10
  \\ \hline
	$A_1^{28}$ & $X_{\alpha} +X_{\beta}$ & 12
	\\ \hline
\end{longtable}
\end{center}

The nilpotent orbits in $G_2$ are distinguished by their dimensions or, equivalently, by the dimensions of the centralizers of their elements. They give the four subalgebras of dimension $1$ of $\mathfrak{n}$ up to conjugacy in $G_2$.\\

We wlll use Table~\ref{table:10} to find the remaining subalgebras. Applying Weyl reflections and automorphisms of the form $\exp(c\cdot X_{\gamma})$, $c\in \mathbb C$, $\gamma \in \Phi$, one can remove some terms from the generators of subalgebras listed in Table~\ref{table:10}. Then we show that the resulting subalgebras are not pairwise conjugate in $G_2$.\\

We will abuse notation by writing $x\in A_1$ or saying that $x$ is 'of type $A_1$', when the element $x\in G_2$ lies in the nilpotent orbit corresponding to the $\mathfrak{sl}(2,\mathbb C)$ subalgebra $A_1$ in Table~\ref{table:20}. Similarly for the other nilpotent orbits.\\

Table~\ref{table:10} gives three regular subalgebras of dimension $2$:
$$
\mathbb C X_{\alpha} \oplus \mathbb C X_{3\alpha +\beta}, \; \mbox{contained in the nilpotent orbit of type}\; \tilde{A_1}, \; \mbox{except for}\; X_{3\alpha +\beta}\in A_1,
$$
$$
\mathbb C X_{\alpha} \oplus \mathbb C X_{3\alpha +2\beta}, \; \mbox{contained in}\; A_1^4, \; \mbox{except for}\; X_{3\alpha +2\beta}\in A_1, X_{\alpha}\in \tilde{A_1},
$$
$$
\mathbb C X_{\beta} \oplus \mathbb C X_{3\alpha +2\beta}, \; \mbox{entirely contained in the nilpotent orbit of type }\; A_1.
$$

It also gives three non-regular subalgebras of dimension $2$:
$$
\mathbb C (X_{\alpha +\beta} + X_{3\alpha +\beta}) \oplus \mathbb C X_{3\alpha +2\beta}, \; \mbox{contained in}\; A_1^4, \; \mbox{except for}\; X_{3\alpha +2\beta}\in A_1,
$$
$$
\mathbb C (X_{\beta } + X_{3\alpha +\beta}) \oplus \mathbb C X_{2\alpha +\beta}, \; \mbox{contained in}\; A_1^4, \; \mbox{except for four elements of}\; \tilde{A_1},
$$
$$
\mathbb C (X_{\alpha } + X_{\beta}) \oplus \mathbb C X_{3\alpha +2\beta}, \; \mbox{contained in}\; A_1^{28}, \; \mbox{except for}\; X_{3\alpha +2\beta}\in A_1.
$$

See Lemma~\ref{lemma:S0}.\\

Table~\ref{table:10} gives four regular subalgebras of dimension $3$:
$$
\mathbb C X_{\beta} \oplus \mathbb C X_{3\alpha +\beta} \oplus \mathbb C X_{3\alpha +2\beta},
$$
$$
\mathbb C X_{\alpha} \oplus \mathbb C X_{2\alpha +\beta} \oplus \mathbb C X_{3\alpha +\beta},
$$
$$
\mathbb C X_{\beta} \oplus \mathbb C X_{\alpha +\beta} \oplus \mathbb C X_{3\alpha +2\beta},\;\;\mbox{abelian},
$$
$$
\mathbb C X_{\alpha} \oplus \mathbb C X_{3\alpha +\beta} \oplus \mathbb C X_{3\alpha +2\beta},\;\;\mbox{abelian}.
$$

It also gives four non-regular subalgebras of dimension $3$:
$$
\mathbb C (X_{\beta} + X_{2\alpha +\beta}) \oplus \mathbb C X_{3\alpha +\beta} \oplus \mathbb C X_{3\alpha +2\beta},
$$
$$
\mathbb C (X_{\alpha} + X_{\beta}) \oplus \mathbb C X_{3\alpha +\beta} \oplus \mathbb C X_{3\alpha +2\beta},
$$
$$
\mathbb C (X_{\alpha +\beta} + X_{3\alpha + \beta}) \oplus \mathbb C X_{2\alpha +\beta} \oplus \mathbb C X_{3\alpha +2\beta},
$$
$$
\mathbb C (X_{\beta} + X_{3\alpha + \beta}) \oplus \mathbb C X_{2\alpha +\beta} \oplus \mathbb C X_{3\alpha +2\beta},\;\;\mbox{abelian}.
$$
There is also a one-parameter family of non-regular subalgebras of dimension three:
$$
\mathbb C (X_{\beta} + X_{\alpha + \beta} + \lambda \cdot X_{3\alpha + \beta}) \oplus \mathbb C X_{2\alpha +\beta} \oplus \mathbb C X_{3\alpha +2\beta},\;\;\lambda \in {\mathbb C}^{*}.
$$

By Lemma~\ref{lemma:S1}, $\mathbb C (X_{\beta}+b_1\cdot X_{2\alpha +\beta}+b_2\cdot X_{3\alpha +\beta})\oplus \mathbb C (X_{\alpha +\beta }+ X_{3\alpha +\beta})\oplus \mathbb C X_{3\alpha +2\beta}$, $b_1,b_2\in \mathbb C$, is conjugate to one of the subalgebras listed above.\\

By \cite{MalcevAbelian}, $G_2$ has exactly $3$ conjugacy classes of abelian subalgebras of dimension $3$ consisting of nilpotent elements. They are shown above.\\

Let $\mathfrak{g}_{\lambda}=\mathbb C (X_{\beta}+X_{\alpha +\beta}+\lambda \cdot X_{3\alpha +\beta})\oplus \mathbb C X_{2\alpha +\beta}\oplus \mathbb C X_{3\alpha +2\beta}\subset G_2$, $\lambda \in {\mathbb C}^{*}$. By Lemma~\ref{lemma:S2}, if $\lambda , \lambda ' \in \mathbb C ^{*}$, $\lambda \neq \lambda '$, then $\mathfrak{g}_{\lambda}$ and $\mathfrak{g}_{\lambda '}$ are not conjugate in $G_2$.\\

Note that $\mathbb C (X_{\beta} + X_{2\alpha +\beta}) \oplus \mathbb C X_{3\alpha +\beta} \oplus \mathbb C X_{3\alpha +2\beta}$ is conjugate to $\mathfrak{g}_{-1}$. Indeed,
$$
\exp(c\cdot ad_{X_{\alpha}})\circ \exp(ad_{X_{-\alpha}}) (X_{\beta}+X_{\alpha+\beta}+X_{2\alpha+\beta}-X_{3\alpha+\beta})=-X_{3\alpha+\beta},
$$
where $c=-2/3$.\\

The dimensions of the normalizers for some of these subalgebras are computed in Table~\ref{table:40}.\\

Suppose $\mathbb C (X_{\beta}+X_{2\alpha +\beta}) \oplus \mathbb C X_{3\alpha +\beta} \oplus \mathbb C X_{3\alpha +2\beta}$ is conjugate to $\mathbb C (X_{\alpha +\beta}+X_{3\alpha +\beta}) \oplus \mathbb C X_{2\alpha + \beta} \oplus \mathbb C X_{3\alpha +2\beta}$ by $s \in Aut(G_2)$ and let $s(X_{\gamma})=a_0\cdot (X_{\alpha +\beta}+X_{3\alpha +\beta}) +a_1\cdot X_{2\alpha +\beta}+a_2\cdot X_{3\alpha +2\beta}$, $\gamma \in \{ 3\alpha +\beta  , 3\alpha +2\beta  \}$. Then
$$
0=ad^3_{s(X_{\gamma})}(X_{-\beta})=-6a_0^3\cdot X_{3\alpha +2\beta},\quad 0=ad^3_{s(X_{\gamma})}(X_{-(3\alpha +\beta )})=(-6a_1^3 -9a_0^2a_1)\cdot X_{3\alpha +2\beta}.
$$

Hence $s(\mathbb C X_{3\alpha +\beta} \oplus \mathbb C X_{3\alpha +2\beta})\subset \mathbb C X_{3\alpha +2\beta}$, a contradiction.\\

Suppose $\mathbb C (X_{\alpha +\beta}+ X_{3\alpha +\beta}) \oplus \mathbb C X_{2\alpha +\beta} \oplus \mathbb C X_{3\alpha +2\beta}$ is conjugate to ${\mathfrak{g}}_{\lambda}=\mathbb C (X_{\beta}+X_{\alpha +\beta}+\lambda \cdot X_{3\alpha +\beta}) \oplus \mathbb C X_{2\alpha + \beta} \oplus \mathbb C X_{3\alpha +2\beta}$, $\lambda \neq -1,0$, by $s \in Aut(G_2)$.\\

By the same argument as in Lemma~\ref{lemma:S2}, we may assume that 
$$
s(\mathfrak{h})=\mathfrak{h},\quad s(X_{2\alpha +\beta})=\eta \cdot X_{2\alpha +\beta},\quad s(X_{3\alpha +2\beta})=\xi \cdot X_{3\alpha +2\beta},\;\; \xi ,\eta \in {\mathbb C}^{*}.
$$

Then
$$
s\mid_{\mathfrak{h}}=Id \;\; \mbox{and}\;\; s(X_{k\alpha +l\beta})=u^kv^l\cdot X_{k\alpha +l\beta}, \;\; u,v\in \mathbb C.
$$

Hence $s(X_{\alpha +\beta}+X_{3\alpha +\beta})=uv\cdot (X_{\alpha +\beta}+u^2\cdot X_{3\alpha +\beta})$ can not lie in ${\mathfrak{g}}_{\lambda}\subset G_2$. This is a contradiction.\\

Table~\ref{table:10} gives two regular subalgebras of dimension $4$:
$$
\mathbb C X_{\alpha} \oplus \mathbb C X_{2\alpha +\beta}  \oplus \mathbb C X_{3\alpha +\beta} \oplus \mathbb C X_{3\alpha +2\beta}, 
$$
with the derived ideal $\mathbb C X_{3\alpha +\beta}$, and
$$
\mathbb C X_{\beta} \oplus \mathbb C X_{2\alpha +\beta}  \oplus \mathbb C X_{3\alpha +\beta} \oplus \mathbb C X_{3\alpha +2\beta},
$$
with the derived ideal $\mathbb C X_{3\alpha +2\beta}$.\\

It also gives two non-regular subalgebras of dimension $4$:
$$
\mathbb C (X_{\alpha} + X_{\beta}) \oplus \mathbb C X_{2\alpha +\beta} \oplus \mathbb C X_{3\alpha +\beta} \oplus \mathbb C X_{3\alpha +2\beta},
$$
with the derived ideal $\mathbb C X_{3\alpha +\beta}\oplus \mathbb C X_{3\alpha +2\beta}$, and 
$$
\mathbb C (X_{\beta} + X_{2\alpha + \beta}) \oplus \mathbb C X_{\alpha +\beta} \oplus \mathbb C X_{3\alpha +\beta} \oplus \mathbb C X_{3\alpha +2\beta},
$$
with the derived ideal $\mathbb C X_{3\alpha +2\beta}$.\\

See Table~\ref{table:40} and Lemma~\ref{lemma:S3}.\\

Table~\ref{table:10} gives two regular subalgebras of dimension $5$:
$$
\mathbb C X_{\alpha} \oplus \mathbb C X_{\alpha +\beta} \oplus \mathbb C X_{2\alpha +\beta}  \oplus \mathbb C X_{3\alpha +\beta} \oplus \mathbb C X_{3\alpha +2\beta}, 
$$
with the derived ideal $\mathbb C X_{2\alpha +\beta}  \oplus \mathbb C X_{3\alpha +\beta} \oplus \mathbb C X_{3\alpha +2\beta}$ and the center $\mathbb C X_{3\alpha +\beta} \oplus \mathbb C X_{3\alpha +2\beta}$, and
$$
\mathbb C X_{\beta} \oplus \mathbb C X_{\alpha +\beta} \oplus \mathbb C X_{2\alpha +\beta}  \oplus \mathbb C X_{3\alpha +\beta} \oplus \mathbb C X_{3\alpha +2\beta}, 
$$
with the derived ideal $\mathbb C X_{3\alpha +2\beta}$.\\

It also gives a non-regular subalgebra of dimension $5$:
$$
\mathbb C (X_{\alpha} + X_{\beta}) \oplus \mathbb C X_{\alpha +\beta}  \oplus \mathbb C X_{2\alpha +\beta} \oplus \mathbb C X_{3\alpha +\beta} \oplus \mathbb C X_{3\alpha +2\beta},
$$
with the derived ideal $\mathbb C X_{2\alpha +\beta}\oplus \mathbb C X_{3\alpha +\beta}\oplus \mathbb C X_{3\alpha +2\beta}$ and the center $\mathbb C X_{3\alpha +2\beta}$.

\end{proof}

The subalgebras found in Corollary~\ref{corollary:1} are recorded in Table~\ref{table:3}, Nos.~1-12.\\

\begin{lemma}\label{lemma:S0}
A subalgebra
$$
\mathfrak{g}_{\mu}=\mathbb C (X_{\beta}+X_{2\alpha +\beta}+\mu\cdot X_{3\alpha +\beta})\oplus \mathbb C (X_{\alpha +\beta }+3 \cdot X_{3\alpha +\beta})\subset G_2, \;\; \mu \in \mathbb C 
$$
is either regular or conjugate to $\mathbb C (X_{\beta} +X_{3\alpha +\beta})\oplus \mathbb C X_{\alpha +\beta}$. 
\end{lemma}

\begin{proof}

By applying $\exp(c\cdot ad_{X_{\alpha}})$, $c\in \mathbb C$, we see that $\mathfrak{g}_{\mu}$ is conjugate to
$$
\tilde{\mathfrak{g}}_{\lambda}=\mathbb C (X_{\beta}+\lambda\cdot X_{3\alpha +\beta})\oplus \mathbb C (X_{\alpha +\beta }+X_{2\alpha +\beta}), \;\; \lambda \in \mathbb C.
$$

If $\lambda =0$, then $\tilde{\mathfrak{g}}_{\lambda}$ is regular. Hence we may assume that $\lambda \in {\mathbb C }^{*}$.\\

Let $a_1$ be any solution of 
$$
3\cdot a_1^4+4(\lambda -1)\cdot a_1^3-6\lambda \cdot a_1^2-\lambda ^2=0
$$
if $\lambda \neq -1$, and $a_1=-1/3$ if $\lambda =-1$. Note that $a_1^2-a_1\neq 0$. Let 
$$
c=\frac{\lambda +a_1^2}{2(a_1^2-a_1)},\quad d=\frac{1}{3\cdot (a_1-c)}.
$$

Then 
$$
\exp(d\cdot ad_{X_{-\alpha}})\circ \exp(c\cdot ad_{X_{\alpha}}) (\tilde{\mathfrak{g}}_{\lambda})= \mathbb C X_{\alpha +\beta}\oplus \mathbb C (b_1\cdot X_{3\alpha +\beta}+b_2 \cdot X_{\beta}) ,
$$
which is either regular or conjugate to $\mathbb C (X_{\beta} +X_{3\alpha +\beta})\oplus \mathbb C X_{\alpha +\beta}$.

\end{proof}

\begin{lemma}\label{lemma:S1}
A subalgebra
$$
\mathfrak{g}_{\mu ,\nu}=\mathbb C (X_{\beta}+\mu\cdot X_{2\alpha +\beta}+\nu\cdot X_{3\alpha +\beta})\oplus \mathbb C (X_{\alpha +\beta }+ X_{3\alpha +\beta})\oplus \mathbb C X_{3\alpha +2\beta}\subset G_2, \;\; \mu ,\nu\in \mathbb C 
$$
is either regular or conjugate to one of the following subalgebras:
$$
\mathbb C (X_{\beta} +X_{2\alpha +\beta}+\lambda \cdot X_{3\alpha +\beta})\oplus \mathbb C X_{\alpha +\beta}\oplus \mathbb C X_{3\alpha +2\beta},\;\; \lambda \in \mathbb C,
$$
$$
\mathbb C (X_{\beta} +X_{3\alpha +\beta})\oplus \mathbb C X_{2\alpha +\beta}\oplus \mathbb C X_{3\alpha +2\beta},\quad \mathbb C (X_{\alpha + \beta} +X_{3\alpha +\beta})\oplus \mathbb C X_{2\alpha +\beta}\oplus \mathbb C X_{3\alpha +2\beta},
$$
$$
\mathbb C (X_{\beta} +X_{2\alpha +\beta})\oplus \mathbb C X_{3\alpha +\beta}\oplus \mathbb C X_{3\alpha +2\beta}.
$$

\end{lemma}

\begin{proof}

Let $a_1$ be any solution of
\begin{equation}\label{equation:5}
a_1^4+\nu \cdot a_1^3+\left( \frac{3\mu ^2 -6\mu -1}{4} \right) \cdot a_1^2-\nu\cdot \left( \frac{3\mu+1}{2} \right) \cdot a_1 -\left( \mu ^3 +\frac{\nu ^2}{4} \right)=0,
\end{equation}
different from $\pm \sqrt{\mu}$, and take
$$
c=\frac{\nu+\mu a_1+a_1}{2\cdot (a_1^2-\mu )}.
$$

Then
$$
\exp(c\cdot ad_{X_{\alpha}})(X_{\beta}+\mu\cdot X_{2\alpha +\beta}+\nu\cdot X_{3\alpha +\beta} +a_1\cdot (X_{\alpha +\beta}+X_{3\alpha +\beta}))=X_{\beta}+(a_1-c)\cdot X_{\alpha +\beta}.
$$

Hence $\mathfrak{g}_{\mu ,\nu}$ is either regular or conjugate to 
$$
\mathbb C (b_0\cdot X_{\beta} +b_1\cdot X_{2\alpha +\beta}+b_2\cdot X_{3\alpha +\beta}) \oplus \mathbb C X_{\alpha +\beta}\oplus \mathbb C X_{3\alpha +2\beta}.
$$

If all solutions of (\ref{equation:5}) are $\pm \sqrt{\mu}$, then either $\mu =\nu =0$ or $\mu =-1$, $\nu =0$.\\

In the latter case, we proceed as before with $c=a_1=\sqrt{\mu}$ and obtain
$$
\exp(c\cdot ad_{X_{\alpha}})(X_{\beta}+\mu\cdot X_{2\alpha +\beta} +a_1\cdot (X_{\alpha +\beta}+X_{3\alpha +\beta}))=X_{\beta},
$$
i.e. $\mathfrak{g}_{\mu ,\nu}$ is conjugate to $\mathbb C X_{\beta} \oplus \mathbb C X_{3\alpha +\beta}\oplus \mathbb C X_{3\alpha +2\beta}$.

\end{proof}

\begin{lemma}\label{lemma:S2}
The subalgebras 
$$
\mathfrak{g}_{\lambda}=\mathbb C (X_{\beta}+X_{\alpha +\beta}+\lambda \cdot X_{3\alpha +\beta})\oplus \mathbb C X_{2\alpha +\beta}\oplus \mathbb C X_{3\alpha +2\beta}\subset G_2, \;\; \lambda \in {\mathbb C }^{*}
$$
are not pairwise conjugate. 
\end{lemma}

\begin{proof}

Suppose $\lambda , {\lambda }'\in {\mathbb C}^{*}$, $\lambda \neq {\lambda }'$, $\lambda \neq -1$ and there exists $s\in Aut(G_2)$ such that $s(\mathfrak{g}_{\lambda '})=\mathfrak{g}_{\lambda}$. Since $s$ preserves the derived ideals, $s(X_{3\alpha +2\beta })=\xi \cdot X_{3\alpha +2\beta }$, $\xi \in {\mathbb C}^{*}$.\\

Let $s(X_{2\alpha +\beta})=a_0\cdot (X_{\beta}+X_{\alpha +\beta}+\lambda \cdot X_{3\alpha +\beta}) + a_1\cdot X_{2\alpha +\beta} +a_2\cdot X_{3\alpha +2\beta}$. By composing $s$ with $\exp(c\cdot ad_{X_{\gamma}})$ for some $c\in \mathbb C$ and $\gamma \in \{ \alpha +\beta , 2\alpha +\beta \}$, we may assume that $a_2=0$.\\

Suppose $a_0\neq 0$. Then we may assume that $a_0=1$. The condition that the centralizer of $s(X_{2\alpha +\beta})$ has dimension $6$ can be written as the following equation:
\begin{equation*}\label{equation:1}
4\cdot (a_1-1)^3+9\cdot (a_1-1)^2+6\cdot (\lambda +1)(a_1-1)+(\lambda +1)^2=0.
\end{equation*}

This implies that $a_1\neq 1$. Take 
$$
c=\frac{\lambda +a_1}{2\cdot (1-a_1)},\quad d=\frac{1}{3\cdot (1-c)}, \quad {\theta}=\exp(ad_{X_{\alpha}})\circ \exp(-ad_{X_{-\alpha}})\circ \exp(ad_{X_{\alpha}})\in Aut(G_2)
$$
and consider $s_1=\theta \circ \exp(d\cdot ad_{X_{-\alpha}}) \circ \exp(c\cdot ad_{X_{\alpha}})\circ s$. After a rescaling, we obtain $s_2\in Aut(G_2)$ such that
$$
s_2(\mathfrak{g}_{\lambda '})=\mathfrak{g}_{\lambda},\quad s_2(X_{3\alpha +2\beta })=\xi \cdot X_{3\alpha +2\beta },\quad s_2(X_{2\alpha +\beta })=\eta \cdot X_{2\alpha +\beta },\quad \xi ,\eta \in {\mathbb C}^{*}.
$$

Suppose $s\in Aut (G_2)$ has these properties. Since $s$ preserves the Lie brackets $[x,X_{3\alpha +2\beta}]$, $[x,X_{2\alpha +\beta}]$, $x\in \mathfrak{h}$,
$$
s(x)=x+b_0\cdot X_{\beta} +b_1\cdot X_{2\alpha + \beta} +b_2\cdot X_{3\alpha + \beta} +b_3\cdot X_{3\alpha + 2\beta}.
$$

By the conjugacy of Cartan subalgebras in the Lie algebra 
$$
\mathfrak{h} \oplus \mathbb C X_{\beta} \oplus \mathbb C X_{2\alpha +\beta} \oplus \mathbb C X_{3\alpha +\beta}\oplus \mathbb C X_{3\alpha +2\beta},
$$
we can find $\tau \in Aut(G_2)$ such that $\tau (s(\mathfrak{h}))=\mathfrak{h}$, $\tau (\mathbb C X_{2\alpha +\beta})=\mathbb C X_{2\alpha +\beta}$, $\tau (\mathbb C X_{3\alpha +2\beta})=\mathbb C X_{3\alpha +2\beta}$, $\tau (\mathfrak{g}_{\lambda})=\mathfrak{g}_{\lambda}$.\\

By composing $s$ with $\tau$, we may assume that 
$$
s(\mathfrak{g}_{\lambda '})=\mathfrak{g}_{\lambda},\quad s(X_{3\alpha +2\beta })=\xi \cdot X_{3\alpha +2\beta },\quad s(X_{2\alpha +\beta })=\eta \cdot X_{2\alpha +\beta },\quad s(\mathfrak{h})=\mathfrak{h}.
$$

This implies that $s\mid _{\mathfrak{h}}=Id$. Hence $s$ is a rescaling: $s(X_{k\alpha +l\beta})=u^kv^l\cdot X_{k\alpha +l\beta}$, $u,v\in {\mathbb C}^{*}$.\\

Then $s(X_{\beta}+X_{\alpha +\beta}+{\lambda }' \cdot X_{3\alpha +\beta})=v\cdot (X_{\beta}+u\cdot X_{\alpha +\beta}+{\lambda }' \cdot u^3\cdot X_{3\alpha +\beta})$. This element can not lie in $\mathfrak{g}_{\lambda}$ if $\lambda \neq {\lambda}'$. This is a contradiction.

\end{proof}

\begin{lemma}\label{lemma:S3}
A subalgebra
$$
\mathfrak{g}_{\lambda}=\mathbb C (X_{\beta}+X_{3\alpha +\beta})\oplus \mathbb C (X_{\alpha +\beta }+\lambda \cdot X_{3\alpha +\beta})\oplus \mathbb C X_{2\alpha +\beta}\oplus \mathbb C X_{3\alpha +2\beta}\subset G_2, \;\; \lambda \in \mathbb C 
$$
is either regular or conjugate to $\mathbb C (X_{\beta} +X_{2\alpha +\beta})\oplus \mathbb C X_{\alpha +\beta} \oplus \mathbb C X_{3\alpha +\beta}\oplus \mathbb C X_{3\alpha +2\beta}$. 
\end{lemma}

\begin{proof}

Let $\omega = X_{\beta }+X_{3\alpha +\beta}+a_1\cdot (X_{\alpha +\beta}+\lambda \cdot X_{3\alpha +\beta})+a_2\cdot X_{2\alpha +\beta}$, where
$$
a_1^3+a_1\cdot \lambda +1=0,\quad a_2=a_1^2.
$$

If $c=1/a_1$, then $\exp(c\cdot ad_{X_{-\alpha }})(\omega)=(1+a_1\lambda)\cdot X_{3\alpha +\beta}$. Hence $\mathfrak{g}_{\lambda}$ is conjugate to 
$$
\mathbb C (X_{\beta}+b_1\cdot X_{\alpha +\beta}+b_2\cdot X_{2\alpha +\beta})\oplus \mathbb C (b_3\cdot X_{\alpha +\beta }+b_4 \cdot X_{2\alpha +\beta})\oplus \mathbb C X_{3\alpha +\beta}\oplus \mathbb C X_{3\alpha +2\beta},
$$
which is either regular or conjugate to 
$$
\mathbb C (X_{\beta} +X_{2\alpha +\beta})\oplus \mathbb C X_{\alpha +\beta} \oplus \mathbb C X_{3\alpha +\beta}\oplus \mathbb C X_{3\alpha +2\beta}.
$$
\end{proof}

%\newpage

\begin{center}
\begin{longtable}{|c|c|c|}
\caption{Some normalizers}\label{table:40}\\
\hline
Subalgebra of $G_2$ & Normalizer in $G_2$ & \begin{tabular}{@{}c@{}} Dimension of \\ the normalizer \end{tabular} 
\\ \hline\hline
	$\mathfrak{u}_{\alpha}$ & \begin{tabular}{@{}c@{}} $\mathfrak{h}\oplus \mathfrak{u}_{-(3\alpha +2\beta)}\oplus \mathfrak{u}_{-\beta}$ \\ $\oplus \mathfrak{u}_{\alpha } \oplus \mathfrak{u}_{3\alpha +\beta}\oplus \mathfrak{u}_{3\alpha +2\beta}$ \end{tabular} & $7$
  \\ \hline
	$\mathfrak{u}_{\beta}$ & \begin{tabular}{@{}c@{}} $\mathfrak{h}\oplus \mathfrak{u}_{-(3\alpha +\beta)}\oplus \mathfrak{u}_{-(2\alpha +\beta)}\oplus \mathfrak{u}_{-\alpha}$ \\ $\oplus \mathfrak{u}_{\beta } \oplus \mathfrak{u}_{\alpha +\beta}\oplus \mathfrak{u}_{2\alpha +\beta}\oplus \mathfrak{u}_{3\alpha +2\beta}$ \end{tabular} & $9$
  \\ \hline
	$\mathbb C (X_{\alpha}+X_{\beta})$ & \begin{tabular}{@{}c@{}} $\mathbb C H_{9\alpha +5\beta}\oplus \mathbb C (X_{\alpha}+X_{\beta} )$ \\ $\oplus \mathbb C X_{3\alpha +2\beta}$ \end{tabular} & $3$
  \\ \hline
	$\mathbb C (X_{\alpha}+X_{3\alpha +2\beta})$ & \begin{tabular}{@{}c@{}} $\mathbb C H_{3\alpha +\beta}\oplus \mathbb C (3\cdot X_{-\beta}+X_{2\alpha +\beta} )$ \\ $\oplus \mathbb C X_{\alpha }\oplus \mathbb C X_{3\alpha +\beta}\oplus \mathbb C X_{3\alpha +2\beta}$ \end{tabular} & $5$
  \\ \hline\hline	
	$\mathfrak{u}_{\alpha}+\mathfrak{u}_{3\alpha +\beta}$ & \begin{tabular}{@{}c@{}} $\mathfrak{h}\oplus \mathfrak{u}_{-\beta}\oplus \mathfrak{u}_{\alpha}\oplus \mathfrak{u}_{2\alpha +\beta}$ \\ $\oplus \mathfrak{u}_{3\alpha +\beta}\oplus \mathfrak{u}_{3\alpha +2\beta}$ \end{tabular} & $7$
  \\ \hline
	$\mathfrak{u}_{\alpha}+\mathfrak{u}_{3\alpha +2\beta}$ & $\mathfrak{h}\oplus \mathfrak{u}_{\alpha}\oplus \mathfrak{u}_{3\alpha +\beta} \oplus \mathfrak{u}_{3\alpha +2\beta}$ & $5$
  \\ \hline
	$\mathfrak{u}_{\beta}+\mathfrak{u}_{3\alpha +2\beta}$ & \begin{tabular}{@{}c@{}} $\mathfrak{h}\oplus \mathfrak{u}_{-(3\alpha +\beta)}\oplus \mathfrak{u}_{-\alpha}$ \\ $\oplus \bigoplus\limits_{\gamma \in {\Phi}^{+}\setminus \{ \alpha \} } \mathfrak{u}_{\gamma}$ \end{tabular} & $9$
  \\ \hline
	$\mathbb C (X_{\alpha +\beta} + X_{3\alpha +\beta}) \oplus \mathbb C X_{3\alpha +2\beta}$ & $\mathbb C H_{3\alpha +2\beta}\oplus \bigoplus\limits_{\gamma \in {\Phi}^{+}\setminus \{ \alpha \} } \mathbb C X_{\gamma}$  & $6$
  \\ \hline
	$\mathbb C (X_{\beta} + X_{3\alpha +\beta}) \oplus \mathbb C X_{2\alpha +\beta}$ & \begin{tabular}{@{}c@{}} $\mathbb C H_{3\alpha +2\beta}\oplus \mathbb C (X_{\beta} + X_{3\alpha +\beta})$ \\ $\oplus \mathbb C X_{2\alpha +\beta}\oplus \mathbb C X_{3\alpha +2\beta}$ \end{tabular} & $4$
  \\ \hline	
	$\mathbb C (X_{\alpha} + X_{\beta}) \oplus \mathbb C X_{3\alpha +2\beta}$ & \begin{tabular}{@{}c@{}} $\mathbb C H_{9\alpha +5\beta}\oplus \mathbb C (X_{\alpha} + X_{\beta})$ \\ $\oplus \mathbb C X_{3\alpha +\beta}\oplus \mathbb C X_{3\alpha +2\beta}$ \end{tabular} & $4$
  \\ \hline\hline
	$\mathfrak{u}_{\beta} + \mathfrak{u}_{3\alpha +\beta}+\mathfrak{u}_{3\alpha +2\beta}$ & $\mathfrak{h}\oplus \bigoplus\limits_{\gamma \in {\Phi}^{+}\setminus \{ \alpha \} } \mathfrak{u}_{\gamma}$ & $7$
  \\ \hline
	$\mathfrak{u}_{\alpha} + \mathfrak{u}_{2\alpha +\beta} + \mathfrak{u}_{3\alpha +\beta}$ & \begin{tabular}{@{}c@{}} $\mathfrak{h}\oplus \mathfrak{u}_{-\beta}\oplus \mathfrak{u}_{\alpha}\oplus \mathfrak{u}_{2\alpha +\beta}$ \\ $\oplus \mathfrak{u}_{3\alpha +\beta}\oplus \mathfrak{u}_{3\alpha +2\beta}$ \end{tabular} & $7$
	\\ \hline
	$\mathfrak{u}_{\beta} + \mathfrak{u}_{\alpha +\beta}+\mathfrak{u}_{3\alpha +2\beta}$ & \begin{tabular}{@{}c@{}} $\mathfrak{h}\oplus \mathfrak{u}_{-(3\alpha +\beta)}\oplus \mathfrak{u}_{-\alpha}$ \\ $\oplus \bigoplus\limits_{\gamma \in {\Phi}^{+}\setminus \{ \alpha \} } \mathfrak{u}_{\gamma}$ \end{tabular} &  $9$
  \\ \hline
	$\mathfrak{u}_{\alpha} + \mathfrak{u}_{3\alpha +\beta} + \mathfrak{u}_{3\alpha +2\beta}$ & \begin{tabular}{@{}c@{}} $\mathfrak{h}\oplus \mathfrak{u}_{-\beta}\oplus \mathfrak{u}_{\alpha}\oplus \mathfrak{u}_{2\alpha +\beta}$ \\ $\oplus \mathfrak{u}_{3\alpha +\beta}\oplus \mathfrak{u}_{3\alpha +2\beta}$ \end{tabular} &  $7$
	\\ \hline
	\begin{tabular}{@{}c@{}} $\mathbb C (X_{\beta} + X_{3\alpha +\beta})$ \\ $\oplus \mathbb C X_{2\alpha +\beta} \oplus \mathbb C X_{3\alpha +2\beta}$ \end{tabular} & $\mathbb C H_{3\alpha +2\beta}\oplus \bigoplus\limits_{\gamma \in {\Phi}^{+}\setminus \{ \alpha \} } \mathbb C X_{\gamma}$ & $6$
	\\ \hline
	\begin{tabular}{@{}c@{}} $\mathbb C (X_{\alpha} + X_{\beta}) \oplus \mathbb C X_{3\alpha +\beta}$ \\ $\oplus \mathbb C X_{3\alpha +2\beta}$ \end{tabular} & \begin{tabular}{@{}c@{}} $\mathbb C H_{9\alpha +5\beta}\oplus \mathbb C (X_{\alpha}+X_{\beta})$ \\ $\oplus \mathbb C X_{2\alpha +\beta}\oplus \mathbb C X_{3\alpha +\beta} \oplus \mathbb C X_{3\alpha +2\beta}$ \end{tabular} & $5$
  \\ \hline
	\begin{tabular}{@{}c@{}} $\mathbb C (X_{\alpha +\beta} + X_{3\alpha + \beta})$ \\ $\oplus \mathbb C X_{2\alpha +\beta} \oplus \mathbb C X_{3\alpha +2\beta}$ \end{tabular} & $\mathbb C H_{3\alpha +2\beta}\oplus \bigoplus\limits_{\gamma \in {\Phi}^{+}\setminus \{ \alpha \} } \mathbb C X_{\gamma}$ & $6$
	\\ \hline
	\begin{tabular}{@{}c@{}} $\mathbb C (X_{\beta}+X_{\alpha +\beta} + \lambda \cdot X_{3\alpha + \beta})$ \\ $\oplus \mathbb C X_{2\alpha +\beta} \oplus \mathbb C X_{3\alpha +2\beta}$, $\lambda \in \mathbb C ^{*}$ \end{tabular} & $\mathbb C H_{3\alpha +2\beta}\oplus \bigoplus\limits_{\gamma \in {\Phi}^{+}\setminus \{ \alpha \} } \mathbb C X_{\gamma}$ & $6$
	\\ \hline\hline
	\begin{tabular}{@{}c@{}} $\mathfrak{u}_{\alpha} + \mathfrak{u}_{2\alpha + \beta} + \mathfrak{u}_{3\alpha + \beta}$ \\ $ + \mathfrak{u}_{3\alpha + 2\beta}$ \end{tabular} & $\mathfrak{h} \oplus \mathfrak{u}_{-\beta}\oplus \bigoplus\limits_{\gamma \in {\Phi}^{+}\setminus \{ \beta \} } \mathfrak{u}_{\gamma}$ & $8$
	\\ \hline	
	\begin{tabular}{@{}c@{}} $\mathfrak{u}_{\beta} + \mathfrak{u}_{2\alpha + \beta} + \mathfrak{u}_{3\alpha + \beta}$ \\ $ + \mathfrak{u}_{3\alpha + 2\beta}$ \end{tabular} & $\mathfrak{h} \oplus \bigoplus\limits_{\gamma \in {\Phi}^{+}\setminus \{ \alpha \} } \mathfrak{u}_{\gamma}$ & $7$
	\\ \hline	
	\begin{tabular}{@{}c@{}} $\mathbb C (X_{\beta} + X_{2\alpha + \beta})\oplus \mathbb C X_{\alpha +\beta} $ \\ $\oplus \mathbb C X_{3\alpha +\beta}\oplus \mathbb C X_{3\alpha +2\beta}$ \end{tabular} & $\mathbb C H_{3\alpha +2\beta}\oplus \bigoplus\limits_{\gamma \in {\Phi}^{+}\setminus \{ \alpha \} } \mathbb C X_{\gamma}$ & $6$
	\\ \hline
	\begin{tabular}{@{}c@{}} $\mathbb C (X_{\alpha} + X_{\beta}) \oplus \mathbb C X_{2\alpha +\beta}$ \\ $\oplus \mathbb C X_{3\alpha +\beta}\oplus \mathbb C X_{3\alpha +2\beta}$ \end{tabular} & \begin{tabular}{@{}c@{}} $\mathbb C H_{9\alpha +5\beta}\oplus \mathbb C (X_{\alpha}+X_{\beta})$ \\ $\oplus \mathbb C X_{\alpha +\beta}\oplus \mathbb C X_{2\alpha +\beta}$ \\ $\oplus \mathbb C X_{3\alpha +\beta}\oplus \mathbb C X_{3\alpha +2\beta}$ \end{tabular} & $6$
  \\ \hline\hline
	$\bigoplus\limits_{\gamma \in {\Phi}^{+}\setminus \{ \beta \} } \mathbb C X_{\gamma}$ & $\mathfrak{h}\oplus \mathbb C X_{-\beta}\oplus \bigoplus\limits_{\gamma \in {\Phi}^{+}} \mathbb C X_{\gamma}$ &  $9$
  \\ \hline
	$\bigoplus\limits_{\gamma \in {\Phi}^{+}\setminus \{ \alpha \} } \mathbb C X_{\gamma}$ & $\mathfrak{h}\oplus \mathbb C X_{-\alpha}\oplus \bigoplus\limits_{\gamma \in {\Phi}^{+}} \mathbb C X_{\gamma}$ &  $9$
	\\ \hline
	\begin{tabular}{@{}c@{}} $\mathbb C (X_{\alpha} + X_{\beta}) \oplus \mathbb C X_{\alpha +\beta}$ \\ $\oplus \mathbb C X_{2\alpha +\beta}\oplus \mathbb C X_{3\alpha +\beta}$ \\ $\oplus \mathbb C X_{3\alpha +2\beta}$ \end{tabular} & \begin{tabular}{@{}c@{}} $\mathbb C H_{9\alpha +5\beta}\oplus \mathbb C (X_{\alpha}+X_{\beta})$ \\ $\oplus \mathbb C X_{\alpha +\beta}\oplus \mathbb C X_{2\alpha +\beta}$ \\ $\oplus \mathbb C X_{3\alpha +\beta}\oplus \mathbb C X_{3\alpha +2\beta}$ \end{tabular} & $6$
  \\ \hline
\end{longtable}
\end{center}

\begin{proposition}\label{proposition:3}
Let $\mathfrak{g}\subset G_2$ be a solvable non-regular subalgebra. If $\mathfrak{g}$ contains a non-zero semisimple element of $G_2$, then $\mathfrak{g}$ is conjugate to one and only one of the subalgebras of $\mathfrak{b}$ listed in Table~\ref{table:3}, Nos.~13-24.
\end{proposition}

\begin{proof}

By Lemma~\ref{lemma:5}, we may assume that $\mathfrak{g}\subset \mathfrak{b}$ contains a non-zero element $f\in \mathfrak{h}$. Then ${\mathfrak{g}}_0=\mathfrak{g} \cap \mathfrak{n}$ is one of the subalgebras listed in Table~\ref{table:10}. We proceed as in the proof of Corollary~\ref{corollary:1}.\\

Note that the presence of several terms in the generators of ${\mathfrak{g}}_0$ imposes a restriction on $f\in \mathfrak{h}\cap \mathfrak{g}$. For example, $\mathbb C f \oplus \mathbb C (X_{\alpha}+X_{\beta})$ is a subalgebra of $G_2$ only if $\alpha (f)=\beta (f)$, i.e. $f\in \mathbb C H_{9\alpha +5\beta}$. The other cases are similar.

\end{proof}

Now we prove the main result of this section.\\

\begin{theorem}\label{theorem:100}
Every non-regular solvable subalgebra of $G_2$ is conjugate to one of the subalgebras listed in Table~\ref{table:3}.\\

The subalgebras of $G_2$ listed in Table~\ref{table:3} are pairwise non-conjugate.
\end{theorem}

\begin{proof}
Let $\mathfrak{g}\subset G_2$ be a non-regular solvable subalgebra. If all elements of $\mathfrak{g}$ are nilpotent, we apply Corollary~\ref{corollary:1}. If $\mathfrak{g}$ contains a non-zero semisimple element, we apply Proposition~\ref{proposition:3}. Hence we may assume that $\mathfrak{g}$ contains no non-zero semisimple elements of $G_2$ and not all of its elements are nilpotent.\\

By Lemma~\ref{lemma:5}, we may assume that $\mathfrak{g}\subset \mathfrak{b}$ contains $f=x+X_{\gamma}$, where $x\in \mathfrak{h}$, $x\neq 0$, $\gamma \in {\Phi}^{+}$, $\gamma (x)=0$. Let ${\mathfrak{g}}_0=\mathfrak{g} \cap \mathfrak{n}$.\\

Note that ${\mathfrak{g}}_0$ must be regular:
$$
{\mathfrak{g}}_0 = \mathbb C X_{{\gamma}_1} \oplus \cdots \oplus \mathbb C X_{{\gamma}_k},\quad 0\prec {\gamma}_1 \prec \cdots \prec {\gamma}_k ,\;\; k\geq 0.
$$

For example, if ${\mathfrak{g}}_0=\mathbb C (X_{3\alpha +\beta}+X_{3\alpha +2\beta})$, then $\gamma =\beta$. Hence $X_{3\alpha +2\beta}=[f,X_{3\alpha +\beta}+X_{3\alpha +2\beta}]-3\alpha (x)\cdot (X_{3\alpha +\beta}+X_{3\alpha +2\beta})\in {\mathfrak{g}}_0$, a contradiction. The other cases are similar.\\

If ${\mathfrak{g}}_0=0$, then $\mathfrak{g}$ is conjugate to 
$$
\mbox{either}\;\; \mathbb C (H_{\alpha}+X_{3\alpha +2\beta})\;\; \mbox{or}\;\; \mathbb C (H_{\beta}+X_{2\alpha +\beta}).
$$

See Table~\ref{table:3}, Nos.~$25$, $26$. These subalgebras are not conjugate, because ${\mathfrak{u}}_{3\alpha +2\beta}$ and ${\mathfrak{u}}_{2\alpha +\beta}$ are not conjugate.\\

If ${\mathfrak{g}}_0=\mathbb C X_{\mu}$, $\mu \in {\Phi}^{+}$, then $[X_{\gamma},X_{\mu}]=0$ and $\gamma \neq \mu$. There are five possibilities up to the Weyl group action:
$$
\gamma =\alpha :\;\; \mu \in \{ 3\alpha +\beta , 3\alpha +2\beta \},
$$
$$
\gamma =\beta :\;\; \mu \in \{ \alpha +\beta , 2\alpha +\beta , 3\alpha +2\beta \}.
$$

Every such pair gives a subalgebra ${\mathfrak{g}}_{\gamma ,\mu}=\mathbb C (H_{\gamma ^{\perp}}+X_{\gamma})\oplus \mathbb C X_{\mu}$, where $\gamma ^{\perp}\in {\Phi}^{+}$, $(\gamma , \gamma ^{\perp})=0$. See Table~\ref{table:3}, Nos.~$27$-$31$.\\

Suppose ${\mathfrak{g}}_{\gamma ,\mu}$ is conjugate to ${\mathfrak{g}}_{\gamma ',\mu '}$ by $s\in Aut(G_2)$. Then 
$$
s(X_{\mu})=\xi\cdot X_{\mu '}, \;\;\xi\in {\mathbb C}^{*}.
$$

After composing $s$ with $\exp(c\cdot ad_{X_{\mu '}})$ for some $c\in \mathbb C$, $s(H_{\gamma ^{\perp}}+X_{\gamma})=\eta \cdot (H_{{\gamma ' } ^{\perp}}+X_{\gamma '})$, $\eta \in {\mathbb C}^{*}$, i.e. 
$$
s(H_{\gamma ^{\perp}})=\eta\cdot H_{{\gamma ' } ^{\perp}} \;\; \mbox{and}\;\; s(X_{\gamma})=\eta \cdot X_{\gamma '}.
$$

Let $y\in \mathfrak{h}$ be arbitrary. Since $s$ preserves the Lie bracket $[y, H_{\gamma ^{\perp}}]$,
$$
s(y)=z+c_{-\gamma '}\cdot X_{-\gamma '}+c_{\gamma '}\cdot X_{\gamma '},\;\; \mbox{where}\; z\in \mathfrak{h}.
$$

Since $s$ preserves the Lie bracket $[y, X_{\gamma}]$, $c_{-\gamma '}=0$. Composing $s$ with $\exp(c\cdot ad_{X_{\gamma '}})$ for some $c\in \mathbb C$, we can ensure that $s(\mathfrak{h})=\mathfrak{h}$. Hence $\{ {\gamma}'; {\mu}' \}$ and $\{ {\gamma}; {\mu} \}$ lie in the same orbit of the Weyl group.\\

If ${\mathfrak{g}}_0=\mathbb C X_{\mu}\oplus \mathbb C X_{\nu}$, $\mu\prec \nu$, $\mu ,\nu \in {\Phi}^{+}$, then $\mu ,\nu \neq \gamma$ and
$$
[X_{\gamma},X_{\mu}]\in \mathbb C X_{\nu},\quad [X_{\gamma},X_{\nu}]=[X_{\mu},X_{\nu}]=0.
$$

There are seven possibilities up to the Weyl group action:
$$
\mu =\alpha ,\;\; \nu=3\alpha +\beta :\;\; \gamma \in \{ 2\alpha +\beta , 3\alpha +2\beta \},
$$
$$
\mu =\alpha ,\;\; \nu=3\alpha +2\beta :\;\; \gamma =3\alpha +\beta ,
$$
$$
\mu =\beta ,\;\; \nu=3\alpha +2\beta :\;\; \gamma \in \{ \alpha +\beta , 2\alpha +\beta , 3\alpha +\beta \},
$$
$$
\mu =\beta ,\;\; \nu=\alpha +\beta :\;\; \gamma =3\alpha +2\beta .
$$

Every such triple gives a subalgebra ${\mathfrak{g}}_{\gamma ,\mu ,\nu}=\mathbb C (H_{\gamma ^{\perp}}+X_{\gamma})\oplus \mathbb C X_{\mu}\oplus \mathbb C X_{\nu}$, where $\gamma ^{\perp}\in {\Phi}^{+}$, $(\gamma , \gamma ^{\perp})=0$. See Table~\ref{table:3}, Nos.~$32$-$38$.\\

Suppose ${\mathfrak{g}}_{\gamma ,\mu ,\nu}$ is conjugate to ${\mathfrak{g}}_{\gamma ',\mu ' , \nu '}$ by $s\in Aut(G_2)$. As above, we may assume that $s(\mathfrak{h})=\mathfrak{h}$. Hence $\{ {\gamma}'; {\mu}' ,\nu ' \}$ and $\{ {\gamma}; {\mu} , \nu \}$ lie in the same orbit of the Weyl group.\\

If ${\mathfrak{g}}_0=\mathbb C X_{\mu}\oplus \mathbb C X_{\nu} \oplus \mathbb C X_{\rho}$, $\mu\prec \nu \prec \rho$, $\mu ,\nu ,\rho \in {\Phi}^{+}$, then $\mu ,\nu ,\rho \neq \gamma$ and
$$
[X_{\gamma},X_{\mu}]\in \mathbb C X_{\nu}\oplus \mathbb C X_{\rho},\quad [X_{\gamma},X_{\nu}]\in \mathbb C X_{\rho},\quad [X_{\mu},X_{\nu}]\in \mathbb C X_{\rho},
$$
$$
[X_{\gamma},X_{\rho}]=[X_{\mu},X_{\rho}]=[X_{\nu},X_{\rho}]=0.
$$

There are six possibilities up to the Weyl group action:
$$
\mu =\alpha +\beta ,\;\; \nu=3\alpha +\beta ,\;\; \rho=3\alpha +2\beta :\;\; \gamma =\beta,
$$
$$
\mu =\beta ,\;\; \nu=3\alpha +\beta ,\;\; \rho=3\alpha +2\beta :\;\; \gamma =\alpha +\beta,
$$
$$
\mu =\alpha ,\;\; \nu=3\alpha +\beta ,\;\; \rho=3\alpha +2\beta :\;\; \gamma =2\alpha +\beta,
$$
$$
\mu =\alpha ,\;\; \nu=2\alpha +\beta ,\;\; \rho=3\alpha +\beta :\;\; \gamma =3\alpha +2\beta,
$$
$$
\mu =\beta ,\;\; \nu=\alpha +\beta ,\;\; \rho=3\alpha +2\beta :\;\; \gamma \in \{ 2\alpha +\beta , 3\alpha +\beta \} .
$$

Every such fourtuple gives a subalgebra ${\mathfrak{g}}_{\gamma ,\mu ,\nu ,\rho}=\mathbb C (H_{\gamma ^{\perp}}+X_{\gamma})\oplus \mathbb C X_{\mu}\oplus \mathbb C X_{\nu}\oplus \mathbb C X_{\rho}$, where $\gamma ^{\perp}\in {\Phi}^{+}$, $(\gamma , \gamma ^{\perp})=0$. See Table~\ref{table:3}, Nos.~$39$-$44$.\\

Suppose ${\mathfrak{g}}_{\gamma ,\mu ,\nu ,\rho}$ is conjugate to ${\mathfrak{g}}_{\gamma ',\mu ' , \nu ' ,\rho '}$ by $s\in Aut(G_2)$. As above, we may assume that $s(\mathfrak{h})=\mathfrak{h}$. Hence $\{ {\gamma}'; {\mu}' ,\nu ' ,  \rho ' \}$ and $\{ {\gamma}; {\mu} , \nu  ,\rho \}$ lie in the same orbit of the Weyl group.\\

If ${\mathfrak{g}}_0=\mathbb C X_{\mu}\oplus \mathbb C X_{\nu} \oplus \mathbb C X_{\rho} \oplus \mathbb C X_{\sigma}$, $\mu\prec \nu \prec \rho\prec \sigma$, $\mu ,\nu ,\rho ,\sigma \in {\Phi}^{+}$, then $\mu ,\nu ,\rho , \sigma \neq \gamma$ and
$$
[X_{\gamma},X_{\mu}]\in \mathbb C X_{\nu}\oplus \mathbb C X_{\rho}\oplus \mathbb C X_{\sigma},\quad [X_{\gamma},X_{\nu}]\in \mathbb C X_{\rho}\oplus \mathbb C X_{\sigma},\quad [X_{\gamma},X_{\rho}]\in \mathbb C X_{\sigma},
$$
$$
[X_{\mu},X_{\nu}]\in \mathbb C X_{\rho}\oplus \mathbb C X_{\sigma},\quad [X_{\mu},X_{\rho}]\in \mathbb C X_{\sigma},\quad [X_{\nu},X_{\rho}]\in \mathbb C X_{\sigma},
$$
$$
[X_{\gamma},X_{\sigma}]=[X_{\mu},X_{\sigma}]=[X_{\nu},X_{\sigma}]=[X_{\rho},X_{\sigma}]=0.
$$

There are three possibilities up to the Weyl group action:
$$
\mu =\beta ,\;\; \nu=2\alpha +\beta ,\;\; \rho=3\alpha +\beta ,\;\; \sigma=3\alpha +2\beta :\;\; \gamma =\alpha +\beta,
$$
$$
\mu =\alpha ,\;\; \nu=2\alpha +\beta ,\;\; \rho=3\alpha +\beta ,\;\; \sigma=3\alpha +2\beta :\;\; \gamma =\alpha +\beta,
$$
$$
\mu =\beta ,\;\; \nu=\alpha +\beta ,\;\; \rho=2\alpha +\beta ,\;\; \sigma=3\alpha +2\beta :\;\; \gamma =3\alpha +\beta .
$$

Every such fourtuple gives a subalgebra ${\mathfrak{g}}_{\gamma ,\mu ,\nu ,\rho ,\sigma}=\mathbb C (H_{\gamma ^{\perp}}+X_{\gamma})\oplus \mathbb C X_{\mu}\oplus \mathbb C X_{\nu}\oplus \mathbb C X_{\rho}\oplus \mathbb C X_{\sigma}$, where $\gamma ^{\perp}\in {\Phi}^{+}$, $(\gamma , \gamma ^{\perp})=0$. See Table~\ref{table:3}, Nos.~$45$-$47$.\\

Suppose ${\mathfrak{g}}_{\gamma ,\mu ,\nu ,\rho ,\sigma}$ is conjugate to ${\mathfrak{g}}_{\gamma ',\mu ' , \nu ' ,\rho ' ,\sigma '}$ by $s\in Aut(G_2)$. As above, we may assume that $s(\mathfrak{h})=\mathfrak{h}$. Hence $\{ {\gamma}'; {\mu}' ,\nu ' ,  \rho ' , \sigma ' \}$ and $\{ {\gamma}; {\mu} , \nu  ,\rho  ,\sigma \}$ lie in the same orbit of the Weyl group.\\

Finally, if $dim({\mathfrak{g}}_0)=5$, then either
$$
{\mathfrak{g}}_0=\mathbb C X_{\beta}\oplus \mathbb C X_{\alpha +\beta}\oplus \mathbb C X_{2\alpha +\beta}\oplus \mathbb C X_{3\alpha +\beta}\oplus \mathbb C X_{3\alpha +2\beta}, \quad \gamma = \alpha
$$
or
$$
{\mathfrak{g}}_0=\mathbb C X_{\alpha}\oplus \mathbb C X_{\alpha +\beta}\oplus \mathbb C X_{2\alpha +\beta}\oplus \mathbb C X_{3\alpha +\beta}\oplus \mathbb C X_{3\alpha +2\beta}, \quad \gamma = \beta .
$$

See Table~\ref{table:3}, Nos.~$48$, $49$.

\end{proof}

\begin{center}
\begin{longtable}{|c|c|}
\caption{Solvable non-regular subalgebras of $G_2$ up to conjugacy}\label{table:3}\\
\hline
No. & Subalgebra 
\\ \hline\hline
	$1$ & $\mathbb C (X_{\alpha}+X_{3\alpha +2\beta})$
  \\ \hline
	$2$ & $\mathbb C (X_{\alpha}+X_{\beta})$
  \\ \hline
	$3$	& $\mathbb C (X_{\alpha}+X_{\beta})\oplus \mathbb C X_{3\alpha +2\beta}$
  \\ \hline		
	$4$	& $\mathbb C (X_{\beta}+X_{3\alpha +\beta})\oplus \mathbb C X_{2\alpha +\beta}$
  \\ \hline		
	$5$	& $\mathbb C (X_{\alpha +\beta}+X_{3\alpha +\beta})\oplus \mathbb C X_{3\alpha +2\beta}$
  \\ \hline
	$6$	& $\mathbb C (X_{\alpha}+X_{\beta})\oplus \mathbb C X_{3\alpha +\beta}\oplus \mathbb C X_{3\alpha +2\beta}$
  \\ \hline
	$7$	& $\mathbb C (X_{\beta}+X_{3\alpha +\beta})\oplus \mathbb C X_{2\alpha +\beta}\oplus \mathbb C X_{3\alpha +2\beta}$
  \\ \hline
	$8$	& $\mathbb C (X_{\alpha +\beta}+X_{3\alpha +\beta})\oplus \mathbb C X_{2\alpha +\beta}\oplus \mathbb C X_{3\alpha +2\beta}$
  \\ \hline
	$9$	& $\mathbb C (X_{\beta}+X_{\alpha +\beta}+\lambda \cdot X_{3\alpha +\beta})\oplus \mathbb C X_{2\alpha +\beta}\oplus \mathbb C X_{3\alpha +2\beta},\;\; \lambda \in {\mathbb C}^{*}$
  \\ \hline			
	$10$	& $\mathbb C (X_{\alpha}+X_{\beta})\oplus \mathbb C X_{2\alpha +\beta}\oplus \mathbb C X_{3\alpha +\beta}\oplus \mathbb C X_{3\alpha +2\beta}$
  \\ \hline		
	$11$	& $\mathbb C (X_{\beta}+X_{2\alpha + \beta})\oplus \mathbb C X_{\alpha +\beta}\oplus \mathbb C X_{3\alpha +\beta}\oplus \mathbb C X_{3\alpha +2\beta}$
  \\ \hline				
	$12$	& $\mathbb C (X_{\alpha}+X_{\beta})\oplus \mathbb C X_{\alpha +\beta}\oplus \mathbb C X_{2\alpha +\beta}\oplus \mathbb C X_{3\alpha +\beta}\oplus \mathbb C X_{3\alpha +2\beta}$
  \\ \hline	\hline	
	$13$ & $\mathbb C H_{3\alpha +\beta}\oplus \mathbb C (X_{\alpha}+X_{3\alpha +2\beta})$
  \\ \hline
	$14$ & $\mathbb C H_{9\alpha +5\beta}\oplus \mathbb C (X_{\alpha}+X_{\beta})$
  \\ \hline
	$15$	& $\mathbb C H_{9\alpha +5\beta}\oplus \mathbb C (X_{\alpha}+X_{\beta})\oplus \mathbb C X_{3\alpha +2\beta}$
  \\ \hline		
	$16$	& $\mathbb C H_{3\alpha +2\beta}\oplus \mathbb C (X_{\beta}+X_{3\alpha +\beta})\oplus \mathbb C X_{2\alpha +\beta}$
  \\ \hline		
	$17$	& $\mathbb C H_{3\alpha +2\beta}\oplus \mathbb C (X_{\alpha +\beta}+X_{3\alpha +\beta})\oplus \mathbb C X_{3\alpha +2\beta}$
  \\ \hline
	$18$	& $\mathbb C H_{9\alpha +5\beta}\oplus \mathbb C (X_{\alpha}+X_{\beta})\oplus \mathbb C X_{3\alpha +\beta}\oplus \mathbb C X_{3\alpha +2\beta}$
  \\ \hline
	$19$	& $\mathbb C H_{3\alpha +2\beta}\oplus \mathbb C (X_{\beta}+X_{3\alpha +\beta})\oplus \mathbb C X_{2\alpha +\beta}\oplus \mathbb C X_{3\alpha +2\beta}$
  \\ \hline
	$20$	& $\mathbb C H_{3\alpha +2\beta}\oplus \mathbb C (X_{\alpha +\beta}+X_{3\alpha +\beta})\oplus \mathbb C X_{2\alpha +\beta}\oplus \mathbb C X_{3\alpha +2\beta}$
  \\ \hline
	$21$	& \begin{tabular}{@{}c@{}} $\mathbb C H_{3\alpha +2\beta}\oplus \mathbb C (X_{\beta}+X_{\alpha +\beta}+\lambda \cdot X_{3\alpha +\beta})$ \\ $\oplus \mathbb C X_{2\alpha +\beta}\oplus \mathbb C X_{3\alpha +2\beta},\;\; \lambda \in {\mathbb C}^{*}$ \end{tabular}
  \\ \hline			
	$22$	& \begin{tabular}{@{}c@{}} $\mathbb C H_{9\alpha +5\beta}\oplus \mathbb C (X_{\alpha}+X_{\beta})\oplus \mathbb C X_{2\alpha +\beta}$ \\ $\oplus \mathbb C X_{3\alpha +\beta}\oplus \mathbb C X_{3\alpha +2\beta}$ \end{tabular}
  \\ \hline		
	$23$	& \begin{tabular}{@{}c@{}} $\mathbb C H_{3\alpha +2\beta}\oplus \mathbb C (X_{\beta}+X_{2\alpha + \beta})\oplus \mathbb C X_{\alpha +\beta}$ \\ $\oplus \mathbb C X_{3\alpha +\beta}\oplus \mathbb C X_{3\alpha +2\beta}$ \end{tabular}
  \\ \hline				
	$24$	& \begin{tabular}{@{}c@{}} $\mathbb C H_{9\alpha +5\beta}\oplus \mathbb C (X_{\alpha}+X_{\beta})\oplus \mathbb C X_{\alpha +\beta}$ \\ $\oplus \mathbb C X_{2\alpha +\beta}\oplus \mathbb C X_{3\alpha +\beta}\oplus \mathbb C X_{3\alpha +2\beta}$ \end{tabular}
	\\ \hline	\hline
	$25$	& $\mathbb C (H_{3\alpha +2\beta}+X_{\alpha})$
  \\ \hline				
	$26$	& $\mathbb C (H_{2\alpha +\beta}+X_{\beta})$
	\\ \hline\hline
	$27$	& $\mathbb C (H_{3\alpha +2\beta}+X_{\alpha})\oplus \mathbb C X_{3\alpha +\beta}$
  \\ \hline
	$28$	& $\mathbb C (H_{3\alpha +2\beta}+X_{\alpha})\oplus \mathbb C X_{3\alpha +2\beta}$
  \\ \hline					
	$29$	& $\mathbb C (H_{2\alpha +\beta}+X_{\beta})\oplus \mathbb C X_{\alpha +\beta}$
  \\ \hline
	$30$	& $\mathbb C (H_{2\alpha +\beta}+X_{\beta})\oplus \mathbb C X_{2\alpha +\beta}$
  \\ \hline
	$31$	& $\mathbb C (H_{2\alpha +\beta}+X_{\beta})\oplus \mathbb C X_{3\alpha +2\beta}$
  \\ \hline \hline
	$32$	& $\mathbb C (H_{\alpha}+X_{3\alpha +2\beta})\oplus \mathbb C X_{\alpha}\oplus \mathbb C X_{3\alpha +\beta}$
  \\ \hline
	$33$	& $\mathbb C (H_{\alpha}+X_{3\alpha +2\beta})\oplus \mathbb C X_{\beta}\oplus \mathbb C X_{\alpha +\beta}$
  \\ \hline
	$34$	& $\mathbb C (H_{\alpha +\beta}+X_{3\alpha +\beta})\oplus \mathbb C X_{\alpha}\oplus \mathbb C X_{3\alpha +2\beta}$
  \\ \hline
	$35$	& $\mathbb C (H_{\alpha +\beta}+X_{3\alpha +\beta})\oplus \mathbb C X_{\beta}\oplus \mathbb C X_{3\alpha +2\beta}$
  \\ \hline
	$36$	& $\mathbb C (H_{\beta}+X_{2\alpha +\beta})\oplus \mathbb C X_{\alpha}\oplus \mathbb C X_{3\alpha +\beta}$
  \\ \hline
	$37$	& $\mathbb C (H_{\beta}+X_{2\alpha +\beta})\oplus \mathbb C X_{\beta}\oplus \mathbb C X_{3\alpha +2\beta}$
  \\ \hline
	$38$	& $\mathbb C (H_{3\alpha +\beta}+X_{\alpha +\beta})\oplus \mathbb C X_{\beta}\oplus \mathbb C X_{3\alpha +2\beta}$
  \\ \hline\hline
	$39$	& $\mathbb C (H_{2\alpha +\beta}+X_{\beta})\oplus \mathbb C X_{\alpha +\beta}\oplus \mathbb C X_{3\alpha +\beta}\oplus \mathbb C X_{3\alpha +2\beta}$
  \\ \hline						
	$40$	& $\mathbb C (H_{3\alpha +\beta}+X_{\alpha +\beta})\oplus \mathbb C X_{\beta}\oplus \mathbb C X_{3\alpha +\beta}\oplus \mathbb C X_{3\alpha +2\beta}$
  \\ \hline	
	$41$	& $\mathbb C (H_{\beta}+X_{2\alpha +\beta})\oplus \mathbb C X_{\alpha}\oplus \mathbb C X_{3\alpha +\beta}\oplus \mathbb C X_{3\alpha +2\beta}$
  \\ \hline
	$42$	& $\mathbb C (H_{\beta}+X_{2\alpha +\beta})\oplus \mathbb C X_{\beta}\oplus \mathbb C X_{\alpha +\beta}\oplus \mathbb C X_{3\alpha +2\beta}$
  \\ \hline
	$43$	& $\mathbb C (H_{\alpha +\beta}+X_{3\alpha +\beta})\oplus \mathbb C X_{\beta}\oplus \mathbb C X_{\alpha +\beta}\oplus \mathbb C X_{3\alpha +2\beta}$
  \\ \hline	
	$44$	& $\mathbb C (H_{\alpha}+X_{3\alpha +2\beta})\oplus \mathbb C X_{\alpha}\oplus \mathbb C X_{2\alpha +\beta}\oplus \mathbb C X_{3\alpha +\beta}$
  \\ \hline\hline	
	$45$	& $\mathbb C (H_{3\alpha +\beta}+X_{\alpha +\beta})\oplus \mathbb C X_{\beta}\oplus \mathbb C X_{2\alpha +\beta}\oplus \mathbb C X_{3\alpha +\beta}\oplus \mathbb C X_{3\alpha +2\beta}$
	 \\ \hline
	$46$ & $\mathbb C (H_{3\alpha +\beta}+X_{\alpha +\beta})\oplus \mathbb C X_{\alpha}\oplus \mathbb C X_{2\alpha +\beta}\oplus \mathbb C X_{3\alpha +\beta}\oplus \mathbb C X_{3\alpha +2\beta}$
	 \\ \hline
	$47$	& $\mathbb C (H_{\alpha +\beta}+X_{3\alpha +\beta})\oplus \mathbb C X_{\beta}\oplus \mathbb C X_{\alpha +\beta}\oplus \mathbb C X_{2\alpha +\beta}\oplus \mathbb C X_{3\alpha +2\beta}$
  \\ \hline\hline
	$48$ & \begin{tabular}{@{}c@{}} $\mathbb C (H_{3\alpha +2\beta} + X_{\alpha})\oplus \mathbb C X_{\beta}\oplus \mathbb C X_{\alpha +\beta}\oplus \mathbb C X_{2\alpha +\beta}$ \\ $\oplus \mathbb C X_{3\alpha +\beta}\oplus \mathbb C X_{3\alpha +2\beta}$ \end{tabular}
	  \\ \hline
	$49$ & \begin{tabular}{@{}c@{}} $\mathbb C (H_{2\alpha +\beta} + X_{\beta})\oplus \mathbb C X_{\alpha}\oplus \mathbb C X_{\alpha +\beta}\oplus \mathbb C X_{2\alpha +\beta}$ \\ $\oplus \mathbb C X_{3\alpha +\beta}\oplus \mathbb C X_{3\alpha +2\beta}$ \end{tabular}
  \\ \hline
\end{longtable}
\end{center}

\section*{Acknowledgement}

The author is grateful to Beijing International Center for Mathematical Research, the Simons Foundation and Peking University for support, excellent working conditions and encouraging atmosphere.\\

\newpage
\bibliographystyle{ams-plain}

\bibliography{SubalgebrasG2}

\providecommand{\bysame}{\leavevmode\hbox to3em{\hrulefill}\thinspace}
\begin{thebibliography}{10}

\bibitem{Aschbacher}
M.~Aschbacher, \emph{Chevalley groups of type {$G_2$} as the group of a
  trilinear form}, Journal of {A}lgebra \textbf{109} (1987), no.~1, 193--259.

\bibitem{Carter}
R.~W. Carter, \emph{Simple groups of {L}ie type}, Pure and {A}pplied
  {M}athematics, vol.~28, John {W}iley \& {S}ons,
  London--{N}ew~{Y}ork--{S}ydney, 1972.

\bibitem{Chebotarev}
N.~G. Chebotarev, \emph{A theorem of the theory of semi-simple {L}ie groups},
  Matematicheskii {S}bornik {N}.{S}. \textbf{11(53)} (1942), no.~3, 239--244.

\bibitem{Orbits}
D.~H. Collingwood and W.~M. McGovern, \emph{Nilpotent orbits in semisimple
  {L}ie algebras}, Van {N}ostrand {R}einhold {M}athematics {S}eries, Van
  {N}ostrand {R}einhold, New {Y}ork, 1993.

\bibitem{RepkaC2}
A.~Douglas and J.~Repka, \emph{Levi decomposable subalgebras of the symplectic
  algebra {$C_2$}}, Journal of {M}athematical {P}hysics \textbf{56} (2015),
  no.~5, 051703, 10 pp.

\bibitem{RepkaA2}
\bysame, \emph{The subalgebras of {$A_2$}}, Journal of {P}ure and {A}pplied
  {A}lgebra \textbf{220} (2016), no.~6, 2389--2413.

\bibitem{RepkaD2}
\bysame, \emph{The subalgebras of {$\mathfrak{so}(4,\mathbb C)$}},
  Communications in {A}lgebra \textbf{44} (2016), no.~12, 5269--5286.

\bibitem{Dynkin}
E.~B. Dynkin, \emph{Semisimple subalgebras of semisimple {L}ie algebras},
  Matematicheskii {S}bornik {N}.{S}. \textbf{30(72)} (1952), no.~2, 349--462,
  (Russian).

\bibitem{Karpelevich}
F.~I. Karpelevich, \emph{On nonsemisimple maximal subalgebras of semisimple
  {L}ie algebras}, Doklady {A}kademii {N}auk {S}{S}{S}{R} ({N}.{S}.)
  \textbf{76} (1951), no.~6, 775--778, (Russian).

\bibitem{Malcev}
A.~I. Malcev, \emph{On semi-simple subgroups of {L}ie groups}, Izvestia
  {A}kademii {N}auk {S}{S}{S}{R} \textbf{8} (1944), 143--174, (Russian).

\bibitem{MalcevAbelian}
\bysame, \emph{Commutative subalgebras of semi-simple {L}ie algebras}, Izvestia
  {A}kademii {N}auk {S}{S}{S}{R} \textbf{9} (1945), 291--300, (Russian).

\bibitem{Morozov}
V.~V. Morozov, \emph{On non-semisimple maximal subgroups of simple groups},
  Ph.D. thesis, Kazan {S}tate {U}niversity, 1943.

\bibitem{Samelson}
H.~Samelson, \emph{Notes on {L}ie algebras}, {S}econd ed., Universitext,
  Springer-{V}erlag, New {Y}ork, 1990.

\end{thebibliography}

\end{document}